\documentclass[12pt]{amsart}
\usepackage{amsmath,amssymb}
\usepackage{latexsym}
\usepackage{hyperref}
\usepackage{cancel}
\usepackage{mathrsfs}
\usepackage{amsthm}
\usepackage{graphics}
\usepackage{color,graphicx}
\usepackage{colortbl}
\usepackage{fullpage}

\usepackage[hang,small,bf, font=sf]{caption}
\usepackage{subfig}
\usepackage{epsfig}
\usepackage{epstopdf}
\usepackage{graphpap,latexsym,epsf}
\usepackage[all]{xy}

\def\sek~{\S{}}

\usepackage{amsthm}

\newtheorem{theorem}{Theorem}[section]

\newtheorem{lemma}[theorem]{Lemma}

\usepackage{color}
\usepackage{graphicx}
\usepackage{color,graphicx}
\usepackage{epstopdf}
\usepackage{graphpap,latexsym,epsf}
\usepackage{slashbox}

\usepackage[latin1]{inputenc}
\usepackage[english]{babel}

\usepackage{amssymb}
\usepackage{amsmath}
\usepackage{latexsym}
\usepackage{mathrsfs}
\usepackage{amsthm}
\usepackage{color,graphicx}

\newcommand{\rosso}[1]{\color{red} #1} 

\renewcommand{\div}{{\rm div}} 

\renewcommand{\div}{\textrm{div}}

\newcommand{\re}{\mathbb{R}}

\numberwithin{equation}{section}

\renewcommand\lll{|\kern-1pt|\kern-1pt|} 

\newcommand{\K}{T} 
\renewcommand{\O}{\Omega} 
\renewcommand{\div}{{\rm div}} 
\newcommand{\jump}[1]{\lbrack\!\lbrack\,#1\,\rbrack\!\rbrack} 
\newcommand{\av}[1]{\{#1\}} 
\newcommand{\Th}{\mathcal{T}_h} 

\newcommand{\dyle}{\displaystyle} 

\newcommand{\calZ}{\mathcal{Z}}

\newcommand{\calP}{ \mathcal{P}}
\newcommand{\calI}{ \mathcal{I}}

\newcommand{\calB}{ \mathcal{B}}

\newcommand{\tf}{t_{f}}

\renewcommand{\lor }{\longrightarrow}

\newcommand{\Q}{\mathbb{Q}}

\definecolor{soheil}{rgb}{0.2,0.45,0.2}



\begin{document}

\title[ High Order  and Energy preserving DG methods for Vlasov-Poisson]{High Order and Energy preserving Discontinuous Galerkin Methods for the Vlasov-Poisson system}

\author{Blanca Ayuso de Dios}
\address{Centre de Recerca Matem\'{a}tica, UAB Science Faculty, 08193 Bellaterra, Barcelona, Spain}
\author{Soheil Hajian}
\address{Section de Math\'ematiques, Universit\'e de Gen\`eve, CP 64, 1211 Gen\`eve 4, Switzerland}

\begin{abstract}

We present a computational study for a family of discontinuous Galerkin methods for the one dimensional Vlasov-Poisson system, recently introduced in \cite{acs0}. We introduce a slight modification of the methods to allow for feasible computations while preserving the properties of the original methods. We study numerically the verification  of the theoretical and convergence analysis, discussing also the conservation properties of the schemes.  The methods are validated through  their application to some of the benchmarks in the simulation of plasma physics. 
\end{abstract}

\maketitle


Numerical simulation has become a major tool for  understanding the complex behavior of a plasma or a particle beam in many situations. This is due not only  to the large number of physical applications and technological implications of the behavior of plasmas, but  also to the intrinsic  difficulties of the models used to describe such behavior. In fact,  it was recognized long time ago that there does not exist any fully satisfactory macroscopic model (fluid equations) which can be used to describe the particle interaction in laser-fusion problems. In contrast, microscopic models (kinetic equations) can provide a more accurate description of the plasmas.

One of the simplest model problems that is currently used in the simulation of plasmas is the Vlasov-Poisson system.  Such system describes the evolution of a plasma of charged particles (electrons and ions) under the effects of the transport and self-consistent electric field. The unknown, typically denoted by  $f(x,v,t)$ (with  $x$ standing for position, $v$ for velocity and $t$ for time), represents the distribution function of particles (ions, electrons, etc.) in the phase space. The coupling with a self-consistent electrostatic field (neglecting magnetic effects)  is taken into account through the Poisson equation. The nonlinear structure of the system prevents from obtaining analytical solutions, except for a few academic cases (see the surveys \cite{glassey,Bouchut,dolbeault} for a good
description on the state of the art of the mathematical analysis of the problem). Therefore, numerical simulations have to be performed to study realistic physical phenomena.

At the present time, there can be distinguished two main  classes of numerical methods  for  simulating plasmas;  Lagrangian (or probabilistic)  and  Eulerian (or deterministic) methods. The former class include all different types of particle methods  \cite{BirdLang,CoRa84,WONumer96, GanVic89, pic1,pic3,pic2} and has been a preferred choice since the beginnings of numerical simulations in  plasma physics  in the $60'$s, due to their simplicity and low computational cost. The basic idea behind these methods is to approximate the motion of the plasma by a finite number of macro-particles in the phase space whose trajectories are computed from the characteristics of the Vlasov equation, while the electrostatic field is computed by collecting the  charge density on a fixed mesh of the physical space. Although this class of methods represents a feasible option and potentially might allow for  resolving the whole $3+3+1$ dimensional problem, their inherent numerical noise precludes from obtaining an accurate description of the distribution function in
the phase space in many interesting cases. This lack of precision can be overcome by using a method from the second class; an Eulerian solver. These type of methods are nothing but classical (or new) numerical schemes discretizing the Vlasov equation on a (fixed) mesh of the phase space. Among them, the most widely used are finite volumes  \cite{Fijal99, filbet} and semi-lagrangian methods \cite{filbet2,Bes04,mcp0,CV07,cr-son00,sl1,sl2,sl3}. Finite volumes (FV) are a simple and inexpensive  option, but in general are low order if one wants to retain  the basic conservative properties of the scheme. \\
Semi-lagrangian schemes (sometimes consider {\it in-between} Eulerian and Lagrangian solvers) have become a popular option, since they can achieve high order allowing at the same time for time integration with large time steps.  However, special care in needed to compute the origin of the characteristics with high order interpolation without spoiling the local character of the reconstruction. A nice numerical study comparing some of the different methods use in plasma simulations is presented in  \cite{comparison}.

In this paper we present a computational study with discontinuous Galerkin (DG) methods. 
DG methods are finite element methods that use discontinuous polynomials. Their local construction endow the methods with good local conservation properties without sacrificing the order of accuracy. This is one of the main motivations for their use in plasma simulations. But it also provides the methods a built-in parallelism  which allows for parallelization of the resulting algorithms. The methods have also many other attractive features: they are extremely flexible in handling $hp$-adaptivity,  the boundary conditions are imposed weakly and the DG mass matrices are block-diagonal which results in very efficient time-stepping algorithms in the context of time-dependent problems, as it is the case here.
In spite that nowdays, DG methods are consider for approximating problems of almost any kind, their  use  for kinetic equations, and more particularly for simulation of plasmas has only been contemplated  very recently.  Among the computational works, we mention the use of DG in a
multi-waterbag approximation of the VP system in \cite{BBBB09};  a piecewise constant DG solver  for VP  in~\cite{gamba0} (which require extremely fine meshes) and semi-lagrangian schemes combined with  high order DG interpolation are presented in \cite{seal,qiuCW}. In both works, the authors also use the positivity preserving limiter introduced in \cite{zs-posit0}. 

A theoretical work has been presented in \cite{acs0,acs1}, where the authors have introduced and analyzed a family of semi-discrete DG schemes for the VP system with periodic boundary conditions, for the one and multi-dimensional cases, respectively. The authors show optimal error estimates for both the distribution function and the electrostatic field, and they study the conservation properties of the proposed schemes. Due to the local construction of the DG schemes,
total mass (or charge) conservation is shown to hold easily. This property is essential in the numerical approximation to VP, since it is required for guaranteeing the well-posedness of the related Poisson problem. The authors also propose a novel DG scheme that is shown to preserve the total energy of the VP system. Their proof however requires the assumption that the DG finite element spaces contain at least all quadratic polynomials.

In this work, we undertake the issues of verification and validation of these family of DG schemes, for the one-dimensional VP system. To accomplish both tasks, we first discuss how the schemes can be efficiently implemented in practice, even in parallel.  For the space discretization, the methods introduced in \cite{acs0}  are based on the coupling of a DG discretization for the Vlasov equation (transport equation) together with a mixed finite element (possibly discontinuous)  approximation to the Poisson problem. Here, however, we present two slight variations of the DG approximation for the Vlasov equation, to allow for feasible computations. The modifications are done in the definition of the numerical flux involving the coupling with the approximate electrostatic field (hence the nonlinearity). The definition in \cite{acs0} would require the computation, at each time step of the zeros of the  approximate electrostatic field, which would increase substantially the cost, taking into account that we use high order approximations. Nevertheless, as we show here, the slight variation in the schemes  does not  affect the optimal accuracy of the methods. Furthermore, since the new definition of the flux is still consistent,  the mass and energy conservation can still be guaranteed (even at the theoretical level). Also,  here we demonstrate numerically that for the energy preserving scheme given in \cite{acs0}, it is indeed necessary (and not a technical restriction due to the proof) to use finite element spaces spaces containing all quadratic polynomials. 

For the time discretization we stick to  a simple fourth order explicit Runge Kutta (RK) method,  the so-called RK4 or classic Runge-Kutta \cite{hairer0}. The reason for not using {\it total variation diminishing} (TVD) RK integrator  is twofold. On  the one hand, in our simulations we have observed no numerical evidence of any essential benefit of the TVD integrator
 over the standard RK method  (probably due to the smoothness  of the solution). On the other hand,   since we focus on high order methods (for the space discretization), the time integration should be accomplished also with some high order time integration scheme. As is well known \cite{tvd0}, a fourth (or higher)  order TVD RK, would require for the computation of the internal stages, the evaluation of the operator and its adjoint, due to the presence of some negative coefficients in the corresponding TVD-RK tableau. This would substantially increase the cost (and storage) of the overall procedure, without any  significant  benefit.
With the fully discrete schemes, we verify numerically the theory developed in \cite{acs0}; both the error analysis together with the conservation properties. 

The second goal of the paper is to validate the methods by studying their performance in approximating some of the classical benchmark problems in plasma physics. Here we consider the linear and nonlinear Landau-damping together with two benchmarks related to the two stream instability problems. We compare our numerical results with those available in literature, getting always at least the same outcomes. In particular, we show the benefit of using the energy preserving high order DG method for the numerical simulations (since no extremely refined meshes are needed and the code can be parallelized).\\

In the last part of the paper, we consider the application of the schemes for the boundary value problem of the VP system studied in \cite{filbet_shu}.
This problem models the evolution of a collisionless electron gas under the influence of a electrostatic field $E$ in an interval $[0,1]$, with electrons  emitted at one end and absorbed at the other end  of the interval. Due to the absorbing boundary condition, it has been proved theoretically   the distribution function $f$ might become discontinuous in finite time, depending on the sign and magnitude of the electrostatic field at the boundary.
Although the DG methods we consider in this paper were not originally designed to approximate such problem, we study here the ability of the methods to capture the discontinuity. 
The results however, are not completely satisfactory, since we do not always (at all the times) capture the behaviour of the solution predicted by the theory developed in \cite{filbet_shu}. A possible reason is the weak nature of the singularity, but it might also happen that as the time evolves the full discretization is adding too much artificial viscosity, which does not allow the methods to capture completely the singularity. This issue together with the tuning of the schemes to capture correctly the singularities (at all times) will be the subject of future study. \\

The outline of the paper is as follows. In Section \ref{sec:2} we
describe the main properties of the continuous problem and introduce the basic notations related to the discrete DG methods. We then introduce the numerical methods we consider discussing also their main properties in Section \ref{sec:3}.  In Section \ref{sec:4} we consider the full discretization and deal with the implementation issues related to the schemes.
Section \ref{sec:6} is devoted to the validation and convergence study of the schemes. We present extensive numerical tests and consider the application of the methods for the simulation of Landau damping and two different tests related to the nonlinear two stream instability. In section \ref{sec:7} we examine the application of the considered DG methods for approximating a Vlasov-Poisson boundary value problem (no periodic boundary conditions). Finally, we derive some conclusion in section \ref{sec:fin} 



\noindent {\bf Notation:} Throughout this paper, we use the standard notation for Sobolev spaces (see \cite{Adams75}). For a bounded domain $B\subset \mathbb{R}^{2}$,  we denote by $H^{m}(B)$ the standard Sobolev space of order $m \geq 0$ and by  $\|\cdot \|_{m,B}$ and $| \cdot |_{m,B}$ the usual Sobolev norm and seminorm, respectively. For $m = 0$, we write $L^{2}(B)$ instead of $H^{0}(B)$. We shall denote by $H^{m}(\mathcal{I})/\re$ the quotient space consisting of equivalence classes of elements of $H^{m}(\mathcal{I})$ differing by constants; for $m=0$ it is denoted by $L^{2}(\mathcal{I})/\re$. We shall indicate by $L^{2}_{0}(B)$ the space of $L^{2}(B)$ functions having zero average over $B$.
\section{The Vlasov-Poisson system and basic notation}\label{sec:2}
In this section we introduce the Vlasov Poisson system and recall some of its properties.
In the last part of the section, we also introduce the basic notation required for describing the numerical methods we consider.
\subsection{Continuous problem: the Vlasov-Poisson system}\label{sec2:a}
We consider a noncollisional plasma of charged particles (electrons and ions). For simplicity, we assume that the properties of the plasma are one dimensional and   we take into account only the electrostatic forces, thus neglecting the electromagnetic effects. We denote by $f=f(x,v,t)$ the electron distribution function and by $E(x,t)=\Phi_{x}(x,t)$  the electrostatic field. The Vlasov-Poisson equations of the plasma in dimensionless variables can be rewritten as,
 \begin{align}
\frac{\partial f}{\partial t}+v\frac{\partial f}{\partial x}-\Phi_{x} \frac{\partial f}{\partial v} 
&=0\quad && (x,v,t)\,\,\in \,\,  \Omega_x\times \re \times [0,\tf],\label{mod01a}\\
-\Phi_{xx}&= \rho(x,t)-1  \quad && (x,t)\,\,\in \,\,  \Omega_x \times [0,\tf], \label{mod01b}
\end{align}
where $v$ denotes the velocity of the charged particles and $\rho(x,t)$ is the charge density defined by
\begin{equation*}
\rho(x,t)=\dyle{\int_{\re} f(x,v,t)dv} \quad \forall\,\, (x,t)\in \Omega_x\times[0,\tf].
\end{equation*}
Let $f_{0}$ denote a given initial distribution $f(x,v,0)=f_{0}(x,v)$ in $(x,v)\in [0,1]\times \re$. We impose periodic boundary conditions on $x$ for the transport equation \eqref{mod01a}, 
\begin{equation*}
f(0,v,t)=f(1,v,t) \quad \forall\, \, (v,t)\in \re\times [0,\tf],
\end{equation*}
and also for the Poisson equation \eqref{mod01b}; i.e.,
\begin{equation}\label{bceA0}
\Phi(0,t)=\Phi(1,t),\quad \forall\,\, t\in [0,\tf].
\end{equation}
To ensure the well-posedness of the Poisson problem we add the compatibility (or normalizing) condition
\begin{equation}\label{bc00}
\dyle{\int_{0}^{1}\rho(x,t) dx =\int_{0}^{1}\int_{\re} f(x,v,t)dvdx =1}, \quad \forall \,\, t\in[0,\tf],
\end{equation}
which is the condition for total charge neutrality. To guarantee the uniqueness of its solution $\Phi$ (otherwise is determined only up to a constant), we fix the value of $\Phi$ at a point. We set
\begin{equation}\label{phi00}
\Phi(0,t)=0 \quad \forall \,\, t\in [0,\tf].
\end{equation}
Notice that \eqref{bc00} express that the total charge of the system is preserved in time.

Through the paper we are only concerned with compactly supported solutions $f$ of problem \eqref{mod01}-\eqref{mod01b}. We assume that a bounded set  $\Omega_v \subset \mathbb{R}$ such that 
\begin{equation*}
\mbox{supp}(f(x,v,0)) \cup\mbox{supp}(f(x,v,t))  \subseteq \Omega_x\times \Omega_v \;, \qquad \forall\, t\in [0,\tf]\;, 
\end{equation*}
and so the Vlasov equation \eqref{mod01a}, can be (and will be) regarded in $\Omega_x\times \Omega_v\times [0,\tf]$:
 \begin{equation}\label{mod01}
\frac{\partial f}{\partial t}+v\frac{\partial f}{\partial x}-\Phi_{x} \frac{\partial f}{\partial v}=0
\quad  (x,v,t)\,\,\in \,\,  \Omega_x\times \Omega_v \times [0,\tf]\;.
\end{equation} 
The charge density is accordingly defined by
\begin{equation}\label{rho}
\rho(x,t)=\dyle{\int_{\Omega_v} f(x,v,t)dv} \quad \forall\,\, (x,t)\in \Omega_x\times[0,\tf].
\end{equation}
We define the total energy of the system as
\begin{equation}\label{ene:def0}
\mathcal{E}(t)= \int_{\O} f(x,v,t)\frac{|v|^{2}}{2}dxdv +\int_{\O_x} \frac{1}{2} |\Phi_x(x,t)|^{2} dx  \quad \forall\, t\in [0,\tf].
\end{equation}
The first term in the above definition represents the kinetic energy; the second, the potential energy of the system.
\subsubsection{Properties}
The Vlasov-Poisson system preserves in time many physical observables. We now briefly revise some:
\begin{itemize}
\item  {\it Mass conservation:} as already mentioned, the total charge of the system is preserved:
\begin{equation}\label{mas0}
\frac{d}{dt} \int_{\O} f(x,v,t)\, dx\,dv =0 \qquad \forall\, t\in [0,\tf]\;.
\end{equation}
\item {\it $L^{p}$-conservation:} noting that $\div_{x,v}\left( [v,-\Phi_x(x,t)]\right)\equiv 0$ one can deduce straightaway  the conservation of all $L^{p}$-norms of the distribution function:
\begin{equation}\label{lp0}
\frac{d}{dt} \int_{\O} \|f(x,v,t)\|^{p}\, dx\,dv =0 \qquad \forall\, t\in [0,\tf]\;.
\end{equation}
We will be particularly concerned with $p=1,2$.
\item {\it Total Energy:} Following \cite{dolbeault}, one can also show the following energy a-priori estimate:
\begin{equation}\label{ene:00}
\frac{d}{dt}\mathcal{E}(t)=\frac{d}{dt}\left( \int_{\O} f(x,v,t)\frac{|v|^{2}}{2}\, dx\,dv +
\int_{\O_x} \frac{1}{2} |E(x,t)|^{2} dx \right)=0 \qquad \forall\, t\in [0,\tf]\;,
\end{equation}
where we have already used the definition of the electrostatic field $E(x,t)=\Phi_x(x,t)$ (compare with \eqref{ene:def0}).
\end{itemize}
\subsection{Basic notation and preliminaries for the numerical methods}\label{sec2:2}
Let $\{\mathcal{T}_{h}\}$ be a family of partitions of our
computational/physical domain $\Omega=\Omega_x\times
\Omega_v=\Omega_x\times [-L,L]$, which we assume to be regular
\cite{ciar2} and made of rectangles. Each cartesian mesh
$\mathcal{T}_{h}$ is defined as 
\begin{equation*}
\mathcal{T}_{h}:=\left\{
\K_{ij}=I_{i}\times J_{j}  ,\quad 1\leq i\leq N_{x},\,\, 1\leq
j\leq N_{v}\,\right\},
\end{equation*}
 where
\begin{equation*}
I_{i}=[x_{i-1/2},x_{i+1/2}] \quad \forall\, i=1,\ldots ,
N_{x};\qquad J_{j}=[v_{j-1/2},v_{j+1/2}] \quad \forall\,
j=1,\ldots , N_{v}\;.
\end{equation*}
The mesh sizes $h_{x}$ and $h_{v}$ relative to the partition
are defined as
\begin{equation*}
0< h_{x} =\max_{1\leq i\leq N_{x}} h_{i}^{x}:=x_{i+1/2}-x_{i-1/2},
\quad 0< h_{v} =\max_{1\leq j\leq N_{v}}
h_{j}^{v}:=v_{i+1/2}-v_{i-1/2}\;,
\end{equation*}
with $h_{i}^{x}$ and $h_{j}^{v}$ denoting the cell lengths of $I_{i}$
and $J_{j}$, respectively. The mesh size of the partition is
defined as $h=\max{(h_{x},h_{v})}$. The shape regularity assumption implies that $\exists \, c_1,\, c_2 >0$ constants independent of $h$ such that $c_1h_v\leq \ h_{x} \leq c_2\,h_v$.\\
We assume that $v=0$ corresponds to a node of the partition along the $v$-axis, i.e.,
$v_{j-1/2}=0$ for some $j$ in the partition of $\Omega_v=[-L,L]$. We denote by $\{\mathcal{I}_{h}\}$ the family of partitions of the interval $\Omega_x$: $ \mathcal{I}_{h}:=\left\{\,\, I_{i}\, : \,\, 1\leq i\leq N_x\,\right\}$.\\
For $ k\geq 1$, let $\mathbb{P}^{k}(I_{i})$ be the space of polynomials of degree up to $k$, and let $\Q^{k}(\K_{ij})$ be  the space
of polynomials of degree at most $k$ in each variable ($(x,v)$). We define the finite element spaces: 
\begin{eqnarray}
V_{h}^{k}&=&\left\{  \psi\in L^{2}(\mathcal{I})\,\,: \quad \psi \in \mathbb{P}^{k}(I_{i}), \,\,\,\forall\,I_{i} \, , \, i=1,\ldots N_{x},\right\},  \label{fep}\\
\calZ_{h}^{k}&:=&\left\{  \xi\in L^{2}(\O)\,\,: \quad \xi \in \Q^{k}(\K_{ij}), \,\, \, \forall\, \K_{ij}= I_{i}\times J_{j},  \,\, \forall i\,, j \, \right\},  \label{feq}\\
W_{h}^{k}&=&\left\{  \chi \in \mathcal{C}^{0}(\mathcal{I}) \,\,: \quad \chi \in \mathbb{P}^{k}(I_{i}), \,\,\,\forall\, I_{i} \, , \, i=1,\ldots N_{x},\right\}\cap L^{2}(\mathcal{I})/\re\;.  \label{fepC} 
\end{eqnarray}
As is usual in the DG methods, we now introduce the the trace operators.
 We denote by
$(\varphi_{h})_{i+1/2,v}^{+}$ and  $(\varphi_{h})_{i+1/2,v}^{-}$
the values of $\varphi_h$ at $(x_{i+1/2},v)$ from the right cell
$I_{i+1}\times J_{j}$  and from the left cell $I_{i}\times J_{j}$,
respectively;
\begin{equation*}
(\varphi_{h})_{i+1/2,v}^{\pm}=\dyle{\lim_{\varepsilon \downarrow
0} {\varphi_h(x_{i+1/2}\pm \varepsilon,v)}} \;, \quad
(\varphi_{h})_{x,j+1/2}^{\pm}=\dyle{\lim_{\varepsilon \downarrow
0} {\varphi_h(x, v_{j+1/2}\pm \varepsilon)}} \;,
\end{equation*}
for all $(x,v)\in \mathcal{I}\times\mathcal{J}$ or in short-hand
notation
\begin{equation}\label{notai0}
(\varphi_{h})_{i+1/2,v}^{\pm}=\varphi_h(x^{\pm}_{i+1/2},v)\;,
\qquad   (\varphi_{h})_{x,j+1/2}^{\pm}=\varphi_h(x,
v^{\pm}_{j+1/2}) \;,
\end{equation}
for all $(x,v)\in I_{i}\times J_{j}$. The jump $\jump{\cdot}$ and
average $\av{\cdot}$ trace operators of $\varphi_{h}$ at
$(x_{i+1/2},v), \,\, \forall\, v\in J_{j}$ are defined by
\begin{equation*}
\begin{aligned}
\jump{\varphi_{h}}_{i+1/2,v}&:=(\varphi_{h})_{i+1/2,v}^{+}-(\varphi_{h})_{i+1/2,v}^{-}  \quad &&\forall \varphi_h \in \calZ_{h}^{k}\;, \\
\av{\varphi_{h}}_{i+1/2,v}&:=\dyle\frac{1}{2}\left[(\varphi_{h})_{i+1/2,v}^{+}+(\varphi_{h})_{i+1/2,v}^{-}\right]
\quad &&\forall \varphi_h \in \calZ_{h}^{k}\;.
\end{aligned}
\end{equation*}
For $k\geq 0$, let 
$P^{k}:L^{2}(\mathcal{I})\lor V_{h}^{k}$ be the standard
$L^{2}$- orthogonal projection onto the finite element space $V_{h}^{k}$
defined locally, i.e., for each $1\leq i\leq N_{x}$,
\begin{equation}\label{defpk00}
\int_{I_{i}} \left(P^{k}(w) -w\right) q_{h}\,dx= 0 \qquad \forall
q_{h}\in \mathbb{P}^{k}(I_{i})\;.
\end{equation}
By definition the projection is stable in $L^{2}(\mathcal{I})$ 
\begin{equation}\label{stabPk}
\| P^{k} (w)\|_{L^{2}(\mathcal{I}_h)}  \leq  \|w\|_{L^{2}
(\mathcal{I})}\qquad \forall\, w\in L^{2}(\mathcal{I}).
\end{equation}
We denote by $\calP_h:L^{2}(\O)\lor \calZ_h^{k}$  the corresponding two dimensional 
 $L^{2}$-orthogonal projection; defined
by $\calP_{h}(w)=(P^{k}_{x}\otimes P^{k}_{v} )(w)$; i.e., for all
$i$ and $j$,
\begin{equation}\label{defPL2}
\int_{I_{i}} \int_{J_{j}}\left(\calP_h(w(x,v)) -w(x,v)\right)
\varphi_{h}(x,v)\,dv\,dx= 0 \quad \forall \varphi_{h}\in
\mathbb{P}^{k}(I_{i})\otimes \mathbb{P}^{k}(J_{j})\;.
\end{equation}
Also from its definition, its $L^{2}$-stability follows immediately.
\section{Discontinuous Galerkin methods for the Vlasov-Poisson system: semi-discrete methods}\label{sec:3}
In this section, we introduce the DG methods we consider for approximating the Vlasov-Poisson system. The methods are  those proposed in  \cite{acs0,acs1}, but with some slight variation required for  practical computations. 
Following \cite{acs0,acs1} we first describe the schemes for the Vlasov equation, proposing several options to modify the methods in \cite{acs0} so that they allow for a feasible implementation.
We  then discuss the approximation of the Poisson problem (again following closely  \cite{acs0}).
We close the section by discussing the main properties of the introduced schemes. Throughout the whole section we focus on the space discretization. 
\subsection{Discontinuous Galerkin approximation to the Vlasov equation}\label{sec3:1}
We now des\-cri\-be the DG methods we consider to approximate the Vlasov equation \eqref{mod01}. 
For the time being, we assume we are given a finite element (conforming or nonconforming) approximation of degree $r$ to the electrostatic  field $E(x,t)=\Phi_x(x,t)$,  which we denote by $E_h \in \mathcal{W}_{h}$.  By $E_{h}^{i}$  we refer to its restriction to $I_{i}$. The properties and characterization of $E_h$ are discussed in next subsection. \\
We denote by $f_{h}(0)=\mathcal{P}_h (f(x,v,0))$ the approximation to the initial data $f(x,v,0)$ computed  using the orthogonal $L^{2}$-projection onto the space $\calZ_h^{k}$.  Since the Vlasov equation is a transport equation, we construct the DG method in the usual way:  given $E_h\in
\mathcal{W}_{h}$ find $f_h:[0,\tf]\lor \calZ_{h}^{k}$ such that
\begin{equation}\label{prob-disc}
\dyle\sum_{i=1}^{N_{x}}\dyle\sum_{j=1}^{N_{v}}\calB^{h}_{ij}(E_h;f_h,\varphi_h)=0
\quad \forall\varphi_h\in\calZ_{h}^{k}\;,
\end{equation}
where the bilinear form $\calB^{h}_{ij}(E_h;f_h,\varphi_h)$ is
defined for each $i$, $j$ and $\varphi_h\in\calZ_{h}^{k}$ as:
\begin{align}
\calB_{ij}(E_h;f_h,\varphi_h)=&\dyle{  \int_{\K_{ij}}
\frac{\partial f_h}{\partial t}\varphi_h \,dv\,dx
- \int_{\K_{ij}} v f_h \frac{\partial \varphi_h}{\partial x} \,dv\,dx + \int_{\K_{ij}} E_{h}^{i} f_{h} \frac{\partial \varphi_h}{\partial v}\,dv\,dx} &&\nonumber \\
&\dyle{+ \int_{J_{j}}\left[ (\widehat{(vf_h)} \varphi^{-}_h)_{i+1/2,v}-(\widehat{ (vf_h)}\varphi^{+}_h)_{i-1/2,v}  \right] dv } &&\label{method0}\\
&\dyle{- \int_{I_{i}}\left[
\left(\widehat{\left(E^{i}_{h}f_h\right)}
\varphi^{-}_h\right)_{x,j+1/2}-\left(\widehat{
\left(E^{i}_{h}f_h\right)}\varphi^{+}_h\right)_{x,j-1/2}  \right].
dx },&& \nonumber
\end{align}
In \eqref{method0} we have used the short hand notation given in
\eqref{notai0}. The numerical fluxes are defined using the upwind flux:
\begin{align}
\widehat{vf_h}&=\left\{\begin{array}{cc}
v\,  f_{h}^{-} & \mbox{  if   } v\geq 0\\
v\,  f_{h}^{+} & \mbox{  if   } v<0
\end{array}\right.  \qquad && \widehat{vf_h}=\av{vf_h}-\frac{|v|}{2}\jump{f_h}\;,\label{flux:v} \\
\widehat{E_{h}^{i}f_{h}} &=\left\{\begin{array}{cc}
E_h^{i}\, f_{h}^{+}& \mbox{  if   } \calP^{0}(E_h^{i})\geq 0\\
E_h^{i}\, f_{h}^{-} & \mbox{  if   } \calP^{0}(E_h^{i})<0
\end{array}\right. \qquad && \widehat{E_{h}^{i}f_{h}}= \av{E_{h}^{i}f_{h}}+\mbox{sign}\left(\calP^{0}(E_h^{i})\right)\cdot\frac{E_h^{i}}{2} \jump{f_h}\;. \label{a1}
\end{align}
At the boundary $\partial\O$,  the numerical fluxes are taken as
\begin{equation*}
(\widehat{vf_h})_{1/2,v}=(\widehat{vf_h})_{N_x+1/2,v} ,\quad
(\widehat{E_{h}^{i}f_{h}})_{x,1/2}=(\widehat{E_{h}^{i}f_{h}})_{x,N_v+1/2}=0
, \, \forall \,(x,v)\in\mathcal{I}\times \mathcal{J},
\end{equation*}
so that the periodicity in $x$ and the compactness in $v$ are
reflected. 
Note that the numerical fluxes as defined in \eqref{flux:v} and \eqref{a1} are consistent.\\
Observe that, unlike in \cite{acs0,acs1} the definition \eqref{a1} of the upwind flux  $\widehat{E_{h}^{i}f_{h}}$ involves a condition on the sign$(\mathcal{P}^{0}(E_h^{i}))$, rather than on sign$(E_h^{i})$. Obviously if $E_h^{i}$ does not vanish inside $I_i$ (and so it does not change sign inside $I_i$), sign$(E_h^{i})=\mbox{sign}\left(\calP^{0}(E_h^{i})\right)$ and the classical definition of the upwind flux is recovered.  However, since $E_h$ is a piecewise polynomial of degree $k+1$  approximation to the electrostatic field, it will in general change sign in some elements $I_i$ of the partition $\mathcal{I}_h$. The classical definition would require to construct (at each time step) a partition of $\Omega_x$ that adapts to the changes of sign of $E_h$ by locating the zeros of $E_h$ at nodes of the desired partition. Such process, although feasible in one dimension, might become too expensive and complicate unnecessarily  the whole solution method for the Valsov-Poisson system (specially in higher dimensions).
The definition \eqref{a1} is considered for computational purposes. It allows to avoid  
computing the zeros of $E_h$  and  re-meshing, at each time step, the partition $\mathcal{I}_h$.

For our computations of the one-dimensional problem we have also examined  two other  variants of the numerical flux  $\widehat{E_{h}^{i}f_{h}}$ defined in \eqref{a1}, that do not require re-meshing, although they require a control on the sign of $E_h$. We also show how this control on the sign$(E_h)$ can be done efficiently.  The first variant we consider is given by 
\begin{equation}\label{a2}
\widehat{E_{h}^{i}f_{h}} =\left\{\begin{aligned}
E_h^{i}\, f_{h}^{+}& \qquad\mbox{  if   } E_h^{i}> 0&&\\
E_h^{i}\, f_{h}^{-} &\qquad \mbox{  if   } E_h^{i}<0&&\\
\av{E_h^{i}\, f_h}_{\omega} & \qquad\mbox{  if   }\exists\, x^{\ast}  \in I_i \quad \mbox{such that} \quad E_h^{i}(x^{\ast})=0\;&&
\end{aligned}\right. 
\end{equation}
where we have used the weighed average
\begin{equation}
\av{E_h^{i}\, f_h}_{\omega} =\omega^{+} E_h^{i}f_h^{+}+ \omega^{-} E_h^{i}f_h^{-} \qquad \omega^{+}\;, \, \omega^{-} \in [0,1]\quad\omega^{+}+\omega^{-}=1.
\end{equation}
The parameter $\omega$ should be chosen so that the amount of upwind is tuned. Although based on heuristics, in our computations we have found that a good choice is given by
\begin{equation}\label{def:w}
		\omega^+ = \frac{| \max_{I_i} E_h |}{ | \max_{I_i} E_h | + | \min_{I_i} E_h |} \;,
		\qquad
		\omega^- = \frac{| \min_{I_i} E_h |}{ | \max_{I_i} E_h | + | \min_{I_i} E_h |}\;.	
		\end{equation}		
The definition of the numerical flux  \eqref{a2} can be rewritten in the compact form:
\begin{equation}\label{a2:b}
\widehat{E_{h}^{i}f_{h}} =\left\{\begin{aligned}
 \av{E_{h}^{i}f_{h}}+\frac{|E_h^{i}|}{2} \jump{f_h} &\qquad \mbox{  if   }\nexists\, x^{\ast}  \in I_i \quad \mbox{such that} \quad E_h^{i}(x^{\ast})=0\; &&\\
\av{E_{h}^{i}f_{h}}+E_h^{i} (\omega^{+}-\frac{1}{2}) \jump{f_h} & \qquad\mbox{  if   }\exists\, x^{\ast}  \in I_i \quad \mbox{such that} \quad E_h^{i}(x^{\ast})=0\;&&
\end{aligned}\right. 
\end{equation}
Observe that this definition of the numerical flux is also consistent.\\
The last variant we consider is defined by:
\begin{equation}\label{aa00}
\widehat{E_{h}^{i}f_{h}} =\left\{\begin{array}{cc}
E_h^{i}\, f_{h}^{+}& \mbox{  if   } E_h^{i}> 0\\
E_h^{i}\, f_{h}^{-} & \mbox{  if   } E_h^{i}<0\\
\calP^{0}(E_h^{i}) f_{h}^{+}& \mbox{  if   } \calP^{0}(E_h^{i})> 0 \mbox{ and  } \exists\, x^{\ast}  \in I_i \quad \mbox{such that} \quad E_h^{i}(x^{\ast})=0 \\
\calP^{0}(E_h^{i}) f_{h}^{-}& \mbox{  if   } \calP^{0}(E_h^{i})< 0 \mbox{ and  } \exists\, x^{\ast}  \in I_i \quad \mbox{such that} \quad E_h^{i}(x^{\ast})=0. \\
\end{array}\right. 
\end{equation}
Note that the numerical flux defined above in \eqref{aa00} is not consistent (it fails to be consistent in those elements where $E_h^{i}$ changes sign, where we commit an error of order $O(h)$). 


Notice that both the weighted average approach \eqref{a2} and the last definition \eqref{aa00} require the knowledge of those elements of the partition $\mathcal{I}_h$ where $E_h$ vanishes. This information can be obtained very easily (at least in one dimension), by checking the sign of coefficients of  $E_h$ expanded in a basis with  Bernstein polynomials\footnote{Bernstein polynomials are non-negative at everypoint of their domain}. See the Appendix \ref{ap0} for further details on Bernstein polynomials and how the detection of change of sign is implemented. 

Mimicking \eqref{rho}, we define the discrete density, $\rho_h(x,t)$:
\begin{equation}\label{rho:disc}
\rho_{h}(x,t)= \int_{\mathcal{J}} f_h(x,v,t)\,dv
=\dyle{\sum_{j}\int_{J_{j}} f_h(x,v,t)\,dv} \quad \forall\,\, x\,
\in \, \mathcal{I},\quad  \forall\,\, t\, \in \, \,[0,\tf].
\end{equation}
One of the nice properties of DG schemes, is that the conservation of the total mass is satisfied by construction.  Next Lemma guarantees that the DG scheme \eqref{prob-disc}-\eqref{method0} with fluxes \eqref{flux:v} and either \eqref{a1} or \eqref{a2} preserves the total charge:
\begin{lemma}\label{massC}
{\bf Particle or Mass Conservation:}  Let $f_h \in
\mathcal{C}^{1}([0,\tf]; \calZ_h^{k})$, with $k\geq 0$, be the DG
approximation to $f$, satisfying
\eqref{prob-disc}-\eqref{method0}, with numerical fluxes defined as in \eqref{flux:v} and either \eqref{a1} or \eqref{a2} or \eqref{aa00}. Then, for all $t\in[0,\tf]$,
\begin{equation}\label{massC0}
\dyle\sum_{i,j}\int_{\K_{ij}} f_h(t)\,dv\,dx
=\dyle\sum_{i,j}\int_{\K_{ij}}f_h(0)\,dv\,dx=\dyle\sum_{i,j}\int_{\K_{ij}}
f_0 \,dv\,dx=1.
\end{equation}
\end{lemma}
Although standard,  we provide here the proof of the above Lemma for the sake of completeness.
\begin{proof}
({\it Proof of Lemma \ref{massC}}). Since $f_h(0)=\calP_h(f_0)$ it follows from the mass conservation of the continuous VP system  \eqref{bc00} and the definition of the $L^{2}$-projection \eqref{defPL2} that
\begin{equation}\label{masf0}
\dyle\sum_{i,j}\int_{\K_{ij}}f_h(0)\,dv\,dx=\dyle\sum_{i,j}\int_{\K_{ij}}\calP_h(f_0)\,dv\,dx=\dyle\sum_{i,j}\int_{\K_{ij}}
f_0 \,dv\,dx=1.
\end{equation}
Now, let $\K_{ij}$ be any arbitrary but fixed element in $\Th$. By setting in \eqref{method0} $\varphi_h=1$ in $\K_{ij}$ and $\varphi_h=0$ elsewhere we find,
\begin{align*}
\calB_{ij}(E_h{\rosso ; } f_h,1)=&\,\frac{d}{dt} \int_{\K_{ij}} f_{h} \,dv\,dx +\int_{J_{j}} [\widehat{(vf_h)}_{i+1/2,v}-\widehat{(vf_h)}_{i-1/2,v}]\,dv &&\\
&- \int_{I_{i}} [\widehat{(E^{i}_h
f_h)}_{x,j+1/2}-\widehat{(E^{i}_hf_h)}_{x,j-1/2}]\,dx\;, \qquad &&
\end{align*}
where we have already used that such $\varphi_h$ obviously satisfies $(\varphi_h)_{i+1/2,v}^{-}=(\varphi_h)_{i-1/2,v}^{+}=1$ at the boundaries of $\K_{ij}$.
Now, the above equation obviously holds for any $i,j$, since the choice of $\K_{ij}$ was arbitrary.
Therefore summing over all $i$ and $j$ the above equation, the flux terms telescope and
there is no boundary term left because of the periodic (for $i$)
and compactly supported (for $j$) boundary conditions. Substitution now in
 \eqref{prob-disc} gives
\begin{equation*}
0=\dyle\sum_{i,j} \calB_{ij}(E_h;f_h,1)=\frac{d}{dt}
\dyle\sum_{i,j}\int_{\K_{ij}} f_{h} \,dv\,dx=0,
\end{equation*}
and so integrating in time and using \eqref{masf0} we reach \eqref{massC0}.
\end{proof}
\subsection{Finite element approximation of the electrostatic field $E$}\label{sec:e2}
We now describe  the methods we consider for approximating the electrostatic field $E(x,t)=\Phi_x(x,t)$. The discrete Poisson problem reads:
\begin{equation}\label{poiss-disc0}
(\Phi_{h})_{xx}=1-\rho_h \quad  x \in \Omega_x, \qquad
\Phi_{h}(1,t)=\Phi_{h}(0,t).
\end{equation}
The well posedness of the above discrete problem is guaranteed by
\eqref{massC0} from Lemma \ref{massC} which in particular implies
\begin{equation}\label{pois2}
(\Phi_{h})_{x}(1,t)=(\Phi_{h})_{x}(0,t).
\end{equation}
To ensure the uniqueness of the solution we also set $\Phi_{h}(0,t)=0$.\\
Since in the Vlasov equation, the transport depends on $E$, to approximate \eqref{poiss-disc0} we consider a mixed finite element approach. For that purpose, we first rewrite problem \eqref{poiss-disc0} as a first order system:
\begin{equation}\label{poiss-disc}
E_h=\frac{\partial \Phi_{h}}{\partial x} \quad  x \in \Omega_x;
\qquad -\frac{\partial E_h}{\partial x}=\rho_h-1 \quad  x \in
\Omega_x\,
\end{equation}
with boundary condition $\Phi_{h}(0,t)=\Phi_{h}(1,t)=0$. We consider the following methods:
\subsubsection{Mixed Finite element approximation:} 
we consider the one-dimensional version of Raviart-Thomas elements,
RT$_{k}\,\, k\geq 1$ \cite{rav-tom0,brezzi-fortin}. In 1D the
mixed finite element spaces turn out to be the $(W_{h}^{k+1},
V_{h}^{k})$-finite element spaces. Note that in particular,
$\frac{d}{dx}(W_h^{k+1})=V_h^{k}$. For $k\geq 0$ the scheme reads:
find $(E_h, \Phi_{h} ) \in W_{h}^{k+1} \times V_{h}^{k}$ such that
\begin{align}
& \int_{\calI} E_h  \,z \,dx  +\int_{\calI} \Phi_{h} \, z_{x} \,dx =0 \quad  &\forall\, z\in W_{h}^{k+1}, \label{rt:a}&&\\
&-\int_{\calI} (E_h)_{x} \, p \,dx   = \int_{\calI} (\rho_h -1) p
\,dx &\quad  \forall\, p\in V_{h}^{k}. \label{rt:b}&&
\end{align}
We also refer to \cite{babuska-narashi}, where the lowest order case was studied for the one-dimensional Poisson problem.
\subsubsection{Local Discontinuous Galerkin (LDG) method:}\label{ldg:0} 
The DG approximation to the first order system \eqref{poiss-disc} reads: find $(E_h, \Phi_{h} ) \in V_{h}^{k+1}
\times V_{h}^{k+1}$ such that for all $i$:
\begin{align}
& \int_{I_{i}} E_h  z \,dx  = -\int_{I_{i}} \Phi_{h}  z_{x} \,dx + [ (\widehat{\Phi_{h}}z^{-})_{i+1/2}-(\widehat{\Phi_{h}}z^{+})_{i-1/2}]   &\forall\, z\in V_{h}^{k+1}, \label{ldg0:a}&&\\
&\int_{I_{i}} E_h  p_{x} \,dx  - \left[
(\widehat{E_h}p^{-})_{i+1/2}-(\widehat{E_h}p^{+})_{i-1/2}\right] =
\int_{I_{i}} (\rho_h-1) p \,dx & \forall\, p\in V_{h}^{k+1}\;.
\label{ldg0:b}&&
\end{align}
The numerical fluxes $(\widehat{\Phi_{h}})_{i-1/2}$ and $(\widehat{E_h})_{i-1/2}$
for the LDG method are defined by:
\begin{equation}\label{flux:DG0}
\left\{\begin{aligned}
(\widehat{\Phi_{h}})_{i-1/2}&=\av{\Phi_{h}}_{i-1/2} -c_{12}\jump{\Phi_{h}}_{i-1/2}\;,&&\\
(\widehat{E_h})_{i-1/2} &=\av{E_h}_{i-1/2}
+c_{12}\jump{E_h}_{i-1/2}+c_{11}\jump{\Phi_{h}}_{i-1/2}\;,&&
\end{aligned}\right.
\end{equation}
where $c_{11}=c\,
(k+1)^{2} h_x^{-1}$ and  $|c_{12}|=1/2$. At
the boundary nodes due to periodicity in $x$ we impose
\begin{equation*}
(\widehat{\Phi_h})_{1/2}=(\widehat{\Phi_h})_{N_x+1/2} ,\quad
(\widehat{E_{h}})_{1/2}=(\widehat{E_{h}})_{N_x+1/2} .
\end{equation*}
The method was first introduced in \cite{cockshu98} for a time dependent convection diffusion problem with $c_{11}=O(1)$. For the Poisson problem it has been considered in \cite{faithcock} in the one-dimensional  case, and in \cite{ilaria0} for higher dimensions. 
\subsubsection{Energy preserving LDG method LDG(v):} We consider the DG approximation as given in \eqref{ldg0:a}- \eqref{ldg0:b},  with numerical fluxes defined by:
\begin{equation}\label{flux:DGv}
\left\{\begin{aligned}
(\widehat{\Phi_{h}})_{i-1/2}&=\av{\Phi_{h}}_{i-1/2} -\frac{\txt{sign}(v)}{2}\jump{\Phi_{h}}_{i-1/2}\;,&&\\
(\widehat{E_h})_{i-1/2} &=\av{E_h}_{i-1/2}
+\frac{\txt{sign}(v)}{2}\jump{E_h}_{i-1/2}+c_{11}\jump{\Phi_{h}}_{i-1/2}\;,&&
\end{aligned}\right.
\end{equation}
with $c_{11}$ chosen as before, i.e., $c_{11}=c\,
(k+1)^{2} h_x^{-1}$. Note that the above method requires the solution of two Poisson problems; one for $v>0$ and one for $v<0$. Being one dimensional, this can be efficiently done without increasing the overall cost of the computation.\\
This choice of numerical fluxes was introduced in \cite{acs0} and extended to multidimensions in \cite{acs1}. In both works it was shown that when combined with the classical standard upwind DG approximation for the Vlasov equation, the resulting semi-discrete DG method conserves the discrete total energy of the system. 
\subsection{Properties of the numerical methods}
We now briefly comment on the properties of the DG methods presented. 
The methods presented contain a small modification of the schemes introduced and analyzed in \cite{acs0,acs1} in a  few elements. More precisely, the numerical flux $\widehat{E_h^{i}f_h}$ in the DG scheme for the Vlasov equation has been redefined through \eqref{a1} and \eqref{a2} (to allow for practical computations), in a few elements. Therefore, we expect that the resulting schemes will show similar stability, conservation and approximation properties.

\noindent {\bf $\bullet$ Mass conservation:} As we already showed in Lemma \ref{massC}, the total charge of the system is conserved. 

\noindent {\bf $\bullet\,\, L^{2}$-stability:}  we now comment on the $L^{2}$-stability of the methods. In \cite{acs0,acs1} the authors prove $L^{2}$-stability for the DG schemes proposed there.  Here, due to the modification of the numerical flux  $\widehat{(E^{i}_hf_h)}$ in those elements where $E_h$ might change sign, one cannot prove $L^{2}$-stability (or at least the usual proof will not go through). Still, since the  flux is modified only  in a few elements,  it is reasonable to expect  the methods to behave as if they were $L^{2}$-stable. This will be verified in the numerical experiments section \ref{sec:6}.

\noindent {\bf $\bullet$ Energy conservation:}  
Finally, we define the discrete total energy:
\begin{equation}\label{def:enh}
\mathcal{E}_h(t) = \int_{\Omega} \frac{|v|^2}{2} f_h(x,v,t) dv\, dx + 
	 \int_{\Omega_x} \frac{1}{2} | E_h(x,t) |^2 dx +
	 \frac{(k+1)^2}{h_x} \sum_{i=0}^{N_x} \jump{\Phi_h}_{i+1/2}^2 \quad \forall\, t\in [0,\tf]\;.
\end{equation}
Next result shows that also for the methods considered here with the modified fluxes \eqref{a1} and \eqref{a2}, when  they are combined with the LDG(v) method for approximating the Poisson problem, the total energy of the Vlasov Poisson system is preserved. 
\begin{theorem}[Energy conservation]\label{teo:ene} Let
$k\geq2 $ and let $((E_{h},\Phi_h), f_{h})$ be the LDG(v)-DG
approximation belonging to $\mathcal{C}^{1}([0,T];(V_{h}^{k}\times
V_h^k)\times \calZ_{h}^{k})$ of the Vlasov-Poisson system
\eqref{mod01a}-\eqref{mod01b}, where $f_h$ is the solution of \eqref{prob-disc},
\eqref{method0}, with numerical fluxes \eqref{flux:v} and either \eqref{a1} or \eqref{a2}, and the approximation $(\Phi_h,E_h)$ solves of \eqref{ldg0:a}- \eqref{ldg0:b}, with 
numerical fluxes \eqref{flux:DGv}. 
Then, the total discrete energy is conserved in time,
\begin{equation}\label{eneId0}
\frac{d}{dt} \left( \dyle{\sum_{i,j} \int_{\K_{ij}}
\left| v \right|^{2}\,f_h(t)\,dv\,dz+\sum_{i}\int_{I_{i}} E_h(t)^{2}\,dx
+c_{11}\sum_{i}\jump{\Phi_{h}(t)}^{2}_{i-1/2}} \right) =0\;.
\end{equation}
\end{theorem}
\begin{proof}
The proof follows exactly the same steps as the proof of \cite[Theorem 5.1]{acs0}. 
All the arguments used there carried over for the methods given here with the numerical flux $\widehat{E^{i}_hf_h}$ modified as in \eqref{a1} and \eqref{a2}. The reason is that the definitions of the fluxes \eqref{a1} and \eqref{a2} are consistent, which is the only property  needed for $\widehat{E^{i}_hf_h}$ to ensure the conservation of the total energy. We omit the details for the sake of conciseness.
\end{proof}
As already noticed in \cite{acs0,acs1}, Theorem \ref{teo:ene} requires the use of polynomial degree $k\geq 2$. We will show in the numerical experiments that this restriction is not technical, but it is indeed required in practice.
\section{Fully discrete method and implementation details}\label{sec:4}
In this section we describe the time integration we consider and we discuss the details on the final solution algorithm. In last part of the section we  also comment on the implementation of the algorithm.
\subsection{Time integration}
The DG methods presented so far are semi-discrete. For the time discretization we consider a simple fourth order explicit Runge Kutta (RK) method,  the so-called RK4 or classic Runge-Kutta \cite{hairer0}. While for conservation laws and other general nonlinear hyperbolic problems,  in order to have good resolution of the shocks and discontinuities a {\it total variation diminishing} (TVD) RK should be used, here we have observed no significant differences (see Fig. \ref{fig:sim2d-rk2tvd}), probably due to the inherent smoothness of the solution. Moreover, since we are concerned with high order methods (in space), the time integration should be accomplished also with some high order time integration scheme. As is well known, a fourth order TVD RK, would require for the computation of the internal stages, the evaluation of the operator and its adjoint, due to the presence of some negative coefficients in the corresponding TVD-RK tableau. Therefore, the cost (and storage) of the overall procedure would substantially increase, and from the experiments carried out, we have no numerical evidence of any essential benefit. This issue deserves surely a further theoretical study, that we plan to do in the future. For these reasons, although in general, is much safer to use a TVD Runge-Kutta method for solving hyperbolic problems, we have stick to  the classical fourth order RK for the simulation of the Vlasov-Poisson system.

We now describe the actual implementation. For the  Runge-Kutta integrator, we take a uniform partition of the time interval $[0,\tf]$, 
$\{0=t_0<t_1<\ldots <t_{n} <\ldots t_{N_t}=\tf\}$ with time step 
$\Delta t:=t_{n+1}-t_{n}$. The RK4 solver updates the current approximate solution, from $f_h(t_n)$ to $f_h(t_{n+1})$, in four internal stages. 
The method is fourth order accurate in time, so we expect the errors coming from the time discretization does not affect much the ones coming from the space discretization. After choosing basis functions in \eqref{prob-disc}-\eqref{method0} we arrive to a system of ODE's.
\begin{equation}\label{eq:finalODE}
	M\frac{d \mathbf{f}_h }{dt} = 
	\mathcal{L}( \mathbf{f}_h, \mathbf{E}_h,t),
\end{equation}
where $M$ denotes the mass matrix, and  $\mathbf{f}_h$ and $\mathbf{E}_h$ are the vector representation of the unknowns $f_h$ and $E_h$, respectively, in the chosen basis.The vector of unknowns, $\mathbf{f}_{h}$, is arranged so that those degrees of freedom corresponding to the same element, say $T_{ij}$, are in the same block. Hence,  the structure of $\mathbf{f}_{h}$  looks as follows:
\begin{equation*}
	\mathbf{f}_{h} := 
	\left[\left( \mathbf{f}_{h} \right)_{T_{11}},\left( \mathbf{f}_{h} \right)_{T_{12}}, \hdots, \left( \mathbf{f}_{h} \right)_{T_{21}}, \hdots, \left( \mathbf{f}_{h} \right)_{T_{N_{x},N_{y}}}\right]^T.
\end{equation*}
The mass matrix, $M$, is then block diagonal (in Lagrange basis) with  each block of  size $(k+1)\cdot(k+1)$, corresponding to the degrees of freedom in each element $T_{ij}$. 
Therefore, the application of  the RK4 solver to \ref{eq:finalODE} can be done elementwise since it is only  needed  the inversion of the local mass matrices of size $(k+1)\cdot(k+1)$ where $k$ is a moderate integer, i.e. $k \leq 12$.  This property allows performing the time marching from $t^n$ to $t^{n+1}$ in parallel as we describe later. The local matrix inversion is done before starting the time integration and it is stored and saved for re-use in the whole computation.\\
To advance in time from  $(t^{n}, \mathbf{f}_h^n, \mathbf{E}_h^{n})$ to  $(t^{n+1}, \mathbf{f}_h^{n+1}, \mathbf{E}_h^{n+1})$ the RK4 method proceeds in $4$-stages:  
\begin{equation*}
\begin{array}{l}
	\mathbf{k}_1 := \Delta t \, M^{-1} \mathcal{L}( \mathbf{f}_h^n, \mathbf{E}_h^n, t_n) 
	\bigskip \\ 
	\mathbf{k}_2 := \Delta t \, M^{-1} \mathcal{L}
	( \mathbf{f}_h^n + {\mathbf{k}_1}/2, \mathbf{E}_h^{n+1/4}, t_n + \Delta t/2) 
	\bigskip \\ 
	\mathbf{k}_3 := \Delta t \, M^{-1} \mathcal{L}
	( \mathbf{f}_h^n + {\mathbf{k}_2}/2, \mathbf{E}_h^{n+2/4}, t_n + \Delta t/2) 
	\bigskip \\ 
	\mathbf{k}_4 := \Delta t \, M^{-1} \mathcal{L}
	( \mathbf{f}_h^n + {\mathbf{k}_3}, \mathbf{E}_h^{n+3/4}, t_n + \Delta t) 	
	\bigskip \\
	\mathbf{f}_h^{n+1} := \mathbf{f}_h^{n} +  \frac{1}{6}
	(\mathbf{k}_1 + 2 \mathbf{k}_2 + 2 \mathbf{k}_3 + \mathbf{k}_4)
\end{array} ,
\end{equation*}
where
\begin{equation*}
	\mathbf{E}_h^{n + i/4} := 
	\mathtt{SolvePoissonUsing}[ \rho^{n+i/4} ] \quad \forall i=1,...3,
\end{equation*}
and $\rho^{n+i/4}$ is computed using the corresponding approximate solution at that stage.\\
Since Vlasov equation is a 2-dimensional transport problem, we choose $\Delta t$ such that
\begin{equation*}
	\Delta t \propto 
	\min_{\forall i,j}\left( h_i^x\, , \, h_j^v \right) / \max{ \{\|E_h\|_{L^{\infty}(\Omega_x)}, 2L\}}\;, 
\end{equation*}
In the actual implementation we have tuned the $\Delta t$ to ensure that the errors coming from the time discretization do not pollute the errors from the space discretization (that we want to observe). In  the experiments we have also used the TVD RK2 time integrator \cite{tvd0,tvd12}, to compare the results obtained with the two time integrators, and to see if there is possible advantage. (See the experiments on the Nonlinear Landau-Damping in Section \ref{sec:6}). It is the following second order RK scheme:
\begin{equation}\label{rk2:TVD}
\begin{array}{l}
	\mathbf{k}_1 := \mathbf{f}_h^n+ \Delta t \, M^{-1} \mathcal{L}( \mathbf{f}_h^n, \mathbf{E}_h^n, t_n) 
	\bigskip \\ 
	\mathbf{f}_h^{n+1} := \dfrac{1}{2} \mathbf{f}_h^n  +  \frac{1}{2}
	\mathbf{k}_1 + \frac{1}{2}\Delta t \, 	M^{-1} \mathcal{L}
	(  {\mathbf{k}_1}, \mathbf{E}_h^{n+1/2}, t_n + \Delta t) 
	\end{array} 
\end{equation}
where
\begin{equation}
	\mathbf{E}_h^{n + 1/2} := 
	\mathtt{SolvePoissonUsing}[ \rho^{n+1/2} ],
\end{equation}
with the same expression for $\Delta t$. We have written it in the above form  to  evidence that the time advancing is done by means of linear combination of intermediate Euler steps (see \cite{tvd0} for further details).
\subsection{Solution algorithm}
We now provide the pseudo-algorithm used in our computations to solve for $\{(f_h^{n},E_h^{n})\}_{n>0}$ given  $f_h^{0}=\calP_h(f(x,v,0))$.
\begin{enumerate}
	\item Given $t_f$, $\Delta t$, $f(x,v,0)$, $k$ (degree of polynomial), $N_x$ and $N_v$
	\begin{enumerate}
		\item Build a 2D mesh suitable for DG scheme using $N_x, N_v, k$
		(Vlasov part)
		\item Build a 1D mesh suitable for DG scheme using $N_x, (k+1)$
		(Poisson part)
		\item Initialize the $f_h^{0}=\calP_h(f(x,v,0))$
	\end{enumerate}
	\item at each time: $t_n, \forall n=0,...,N_t$
	\begin{enumerate}
		\item at each stage of RK4: $s=1,...,4$
		\begin{enumerate}
			\item Solve the Poisson problem and computing $E_h^{n+(s-1)/4}$ 
			\item compute $k_{s}$ using 
			$\mathcal{L} \left( f_h^{n+(s-1)/4}, E_h^{n+(s-1)/4}, t_{s} \right) $
		\end{enumerate}
		\item compute the $f_h^{n+1}$ using $\{ k_{s} \}$.
	\end{enumerate}
	\item Give the approximate solution at $t_f: f_h(t_f)$
\end{enumerate}
Since the Poisson problem is one dimensional, its solution  (Step 2.a (i)) is done using an exact solver. This is pre-computed at the initial time step, and then at each Step 2.a (i), only two matrix vector multiplication are required. For higher dimensions, such solution process should be done iteratively and with an appropriate preconditioner \cite{guido-jay,paola-blanca1}.
We should mention during actual implementation and verification of the code, we have observed that in order to guarantee accuracy and to ensure the conservation of mass and the total energy, it is essential  in every time step to  compute $E_h$ at each stage  of the RK method. If on the contrary, one uses $E^n_h$ at the all stages involved in the evolution from $t_{n}$ to $t_{n+1}$, the high order accuracy and conservation properties of the methods are lost.
\subsection{Comments on the implementation}
Now we describe a few issues related to the practical implementation of the scheme. 
As mentioned before, the fact of using an explicit ODE solver for \eqref{eq:finalODE} together with the block diagonal structure of the mass matrix $M$ allows to perform time marching in a parallel manner. Here we briefly demonstrate the structure of parallelization algorithm for time marching which is very useful when the size of the semi-discrete form becomes very big. Denoting the number of computational threads by $N_{\tt CPU}$, we partition the elements in $\mathcal{T}_{h}$ into $\{ \mathcal{T}_{h,m} \}$ for $m=1,..., N_{\tt CPU}$ where 
\begin{equation*}
	\begin{array}{l}
		\mathcal{T}_{h} = \cup_{m=1}^{N_{\tt CPU}} \mathcal{T}_{h,m},
		\\
		\emptyset = \cap_{m=1}^{N_{\tt CPU}} \mathcal{T}_{h,m}.
	\end{array}
\end{equation*}
Moreover in the same way we decompose the vector of degrees of freedom $\bf f_h$ into $\{ \bf f_{h,m} \}$ and $\bf k_s$ into $\{ \bf k_{s,m} \}$ for $m=1,...,N_{\tt CPU}$. Then for the parallel time marching we have
\begin{enumerate}
	\item Given $N_{\tt CPU}$ (number of threads), $\{ \mathcal{T}_{h,m} \}$ (partition of $\mathcal{T}_{h}$)
	\item at each time: $t_n, \forall n=0,...,N_t$
	\begin{enumerate}
		\item at each stage of RK4: $s=1,...,4$
		\begin{enumerate}
			\item Solve the Poisson problem and compute $E_h^{n+(s-1)/4}$ 
			\item on each partition $\mathcal{T}_{h,m} $: $m=1,...,N_{\tt CPU}$
			\begin{enumerate}
				\item compute $k_{s,m}$ using $\mathcal{L} \left( f_h^{n+(s-1)/4}, E_h^{n+(s-1)/4}, t_{s} \right) $		
			\end{enumerate}
			\item construct $k_{s} = \cup_{m}^{N_{\tt CPU}} k_{s,m}$ .
		\end{enumerate}
		\item compute the $f_h^{n+1}$ using $\{ k_{s} \}$.
	\end{enumerate}
	\item Give the approximate solution at $t_f: f_h(t_f)$
\end{enumerate}
Note that Step 2.a.ii.~is the parallel part of the algorithm. The structure of the above algorithm is well-suited for \texttt{OpenMP} library for parallelization. We have run the algorithm on maximum 8 thread using \texttt{OpenMP} when the size of system becomes very big. In order to boost the number of threads one should modify the algorithm structure to run under \texttt{MPI} library which was not needed during our numerical experiments.
\section{Numerical experiments}\label{sec:6}
In this section we present several numerical test to assess the performance of the introduced methods and to validate their properties. We start with a convergence study using a forced VP system for which the exact solution can be explicitly computed. Then, we present the results obtained for simulations of some of the classical benchmark test for the VP system with periodic boundary conditions. More precisely the tests we consider are:
\begin{enumerate}
\item Convergence study: forced Vlasov-Poisson system,
\item 1D weak Landau damping,
\item Nonlinear (strong) Landau damping,
\item Two stream instability I,
\item Two stream instability II.
\end{enumerate}
To study and asses the ability of the proposed DG schemes to conserve the physical invariants of the continuous Vlasov-Poisson system, we have computed for the different tests, the time evolution of the deviations from their initial values of the quantities that are
conserved by the continuous Vlasov-Poisson system (see Section \ref{sec2:a}). The value of these quantities at time $t=0$ can be always computed using the discrete initial data (which is the $L^{2}$-projection of the continuous initial data as defined in \eqref{defPL2}).
We will study  the following deviations (from conservation):
\begin{align}
&\mbox{\bf Total energy}\qquad &&\frac{ \left| \mathcal{E}_{h}(t) - \mathcal{E}_h(0) \right| }{ \mathcal{E}_h(0) } \label{dev:ener}\\
&\mbox{\bf Total Mass}\qquad&&	\frac{ \int_{\O_x}\rho_h(t,s)ds - \int_{\O}f_h(x,v,0)dxdv }{  \int_{\O}f_h(x,v,0)dxdv}
	\label{dev:mass} \\
&\mbox{\bf $L^{1}$-norm}\qquad &&	\frac{ \Vert f_h(t) \Vert_{L^1} - \Vert f_h(0) \Vert_{L^1}  }{ \Vert f_h(0) \Vert_{L^1} } \label{dev:l1} \\
&\mbox{\bf $L^{2}$-norm}\qquad &&	\frac{ \Vert f_h(t) \Vert_{L^2} - \Vert f_h(0) \Vert_{L^2} }{ \Vert f_h(0) \Vert_{L^2} } \label{dev:l2} &&
\end{align} 
Most of the computations are carried out with the RK4 as described in section \ref{sec:4} and the energy preserving DG-LDG(v) method, with numerical fluxes \eqref{flux:v}-\eqref{a2} and \eqref{def:w} for the Vlasov discretization and \eqref{flux:DGv} for the LDG(v) discretization of Poisson. 
When using the other described methods (in space and time) we will explicitly say.

\subsection{Test 1: convergence for a forced Vlasov-Poisson system}\label{sec:conv-vp}
\begin{figure}[!htb]
		\centering
		\includegraphics[scale=.655]{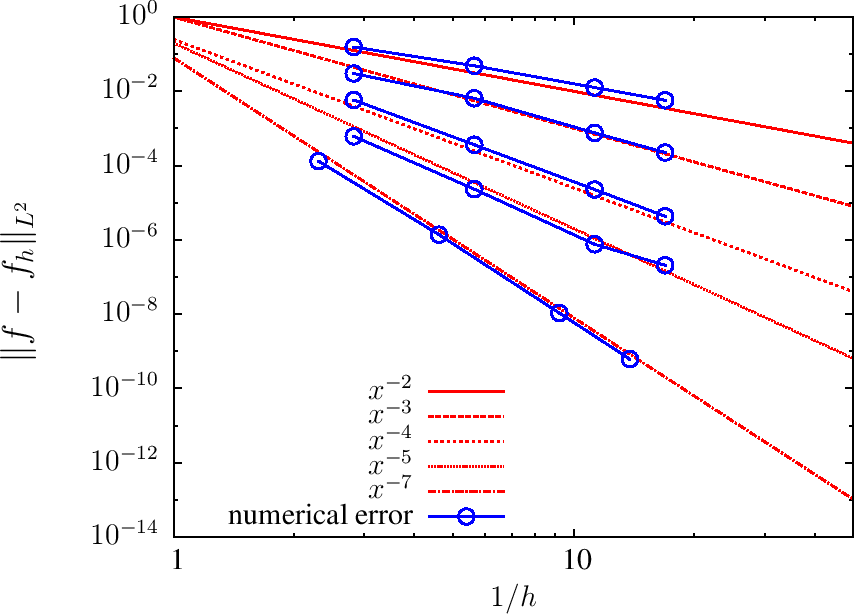} \hfill
		\includegraphics[scale=0.655]{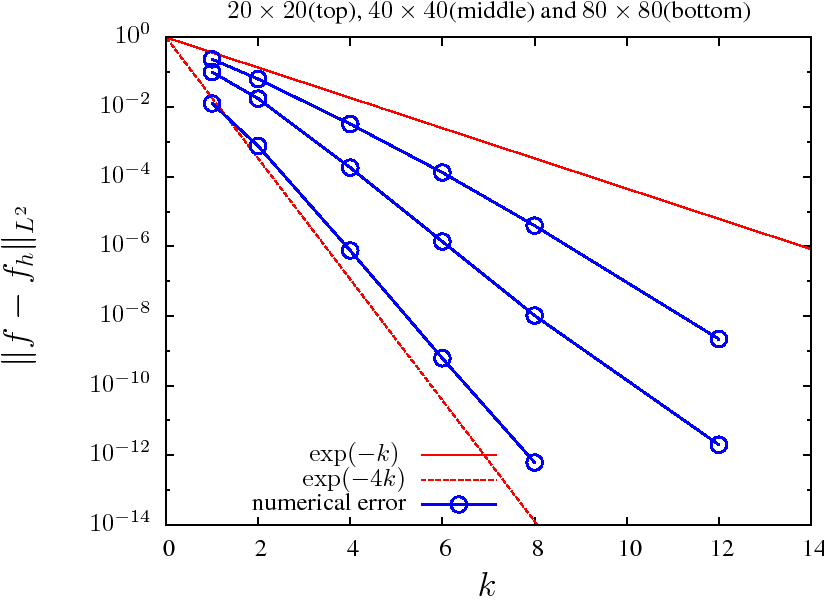} 	
		\caption{ {\bf Forced VP:} Convergence diagram  for the $L^{2}$-error in the approximation of the distribution function.
		$h$-convergence diagram (left) and $k$-convergence diagram (right).}
		\label{fig:sim2d-vp}
	\end{figure}
The aim of this test is to validate and assess the convergence properties of the proposed DG schemes. We set $\Omega_x=[-\pi,\pi]$ and $\Omega_v=[-4,4]$ and consider the following VP system in $\Omega=\Omega_x\times \Omega_v$:

\begin{equation}\label{test1:0}
	\left\{
	\begin{aligned}
		f_t + v f_x - E(x,t) f_v &= \psi(x,v,t) \qquad &&(x,v,t)\in \Omega \times \re^+, \\
		-\frac{\partial}{\partial x}E(x,t)& = \rho(x,t) - \sqrt{\pi}\qquad && x\in [-\pi,\pi],\\
	\end{aligned}
	\right.
\end{equation}
where the right hand side $\psi(x,v,t)$ is chosen so that the exact solution $(f,E)$ of \eqref{test1:0} (complemented with periodic bc in $x$ and compact support in $v$) is given by
\begin{align*}
		f(x,v,t) &= \left\lbrace 2 - \cos(2 x- 2\pi t) \right\rbrace e^{-\frac{1}{4}(4v-1)^2 } 
		\quad &&(x,v,t)\in \Omega \times \re^+ , \\
		E(x,t) &= \frac{\sqrt{\pi}}{4} \sin(2 x- 2\pi t) \quad && (x,t) \in \Omega_x \times \re^+\;.
	\end{align*}	
The forcing term $\psi(x,v,t)$ in \eqref{test1:0} is defined by
\begin{eqnarray*}
	\psi(x,v,t) = e^{-\frac{1}{4}(4v-1)^2 } 
				&&
				\left( 
				\left\lbrace (4 \sqrt{\pi}+2)v - (2 \pi + \sqrt{\pi}) \right\rbrace \sin( 2x- 2\pi t ) 
				\right.
				\\		
				&& +
				\left.				
				\sqrt{\pi} \left( 1/4  - v \right) \sin( 4x- 4\pi t ) \right).
\end{eqnarray*}
Observe that the solution $(f,E)$ is periodic in time with period $1$: 
	\begin{equation*}
		f(x,v,1) = f(x,v,0) \quad \forall\, (x,v)\in \Omega\;.
	\end{equation*}
	Therefore, we have performed the  the computations  up to time $\tf=1$. 
	\begin{figure}[!htb]
		\centering
		\includegraphics[scale=0.655]{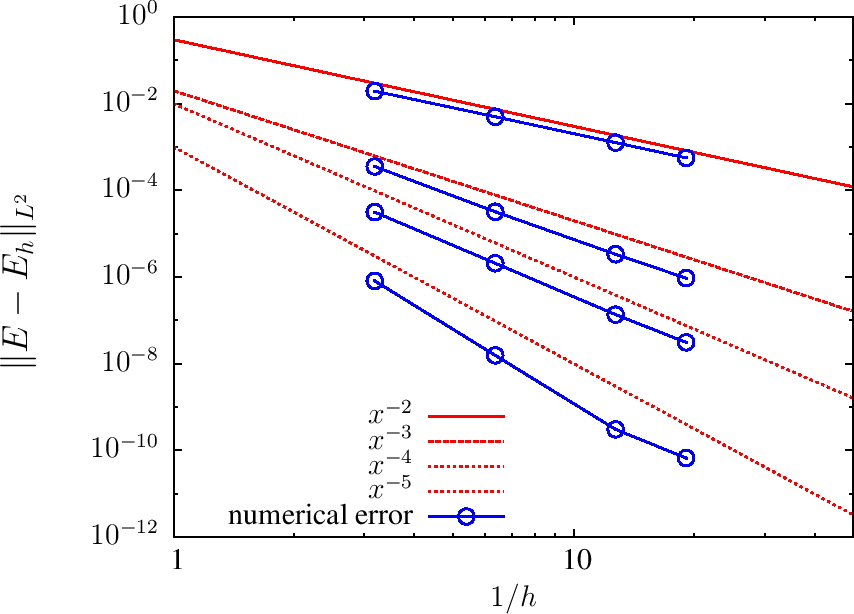} 
		\caption{{\bf Forced VP: }Convergence diagram  for the error in the approximation of the electrostatic field $E_h$.}
		\label{fig:sim2d-vpE}
	\end{figure}
In Figure \ref{fig:sim2d-vp} are given the convergence diagrams for $L^{2}$ error of the distribution function approximated by the energy preserving DG method \eqref{ldg0:a}-\eqref{ldg0:b}-\eqref{flux:DGv}. 
The $h$-convergence diagram is depicted on the left figure for several polynomial degrees $k=1,2,\ldots 6$. In agreement with the theory in \cite{acs0}, optimal order of convergence $k+1$ is achieved when using the method with approximation degree polynomial $k$. On the right figure we have also represented the $k$-convergence diagram (error versus polynomial degree) varying the polynomial degree $k=1,2,4,8,12$ for different meshes $20 \times 20$, $40 \times 40$ and $80 \times 80$. Although not proved theoretically, the results in the figure seem to indicate an exponential rate of convergence in this case.\\
 The corresponding convergence diagram for the error in the electrostatic field is given in Figure \ref{fig:sim2d-vpE}. Also in this case the theory in \cite{acs0} is verified. Note that as predicted in \cite{acs0} the same convergence rates are obtained for the approximations of $f_h$ and $E_h$, even if the approximation to $E_h$ is done using polynomials one degree higher. Observe though that the error in $E_h$ is much smaller. 
 \begin{table}
\small
	\centering
	\begin{tabular}{|l|c|c||c|c||c|c|} 
		\hline \hline
		$k=2$ &  \multicolumn{2}{|c|}{$RT_2$} &\multicolumn{2}{|c|}{LDG$(v)$} &\multicolumn{2}{|c|}{LDG  }\\
		\cline{1-7}
		h & $L^2$ error  & order & $L^2$ error & order & $L^2$ error & order \\
		\hline $1/20$ & $3.0154 \times 10^{-2}$ & - & $3.0134 \times 10^{-2}$ & - & $3.0134 \times 10^{-2}$ & - \\
		\hline $1/40$ & $6.4640 \times 10^{-3}$ & $2.221873$ & $6.4623 \times 10^{-3}$ & $2.2213132$ &
		$6.4623 \times 10^{-3}$ & $2.2213157$
		\\		
		\hline $1/80$ & $7.5804 \times 10^{-4}$ & $3.092085$ & $7.5775 \times 10^{-4}$ & $3.0922407$ &
		$7.5775 \times 10^{-4}$ & $3.0922409$
		\\			
		\hline \hline
	\end{tabular}	
	\vspace{0.5cm} 
	\begin{tabular}{|l|c|c||c|c||c|c|} 
		\hline \hline
		$k=3$ &  \multicolumn{2}{|c|}{$RT_3$} &\multicolumn{2}{|c|}{LDG$(v)$} &\multicolumn{2}{|c|}{LDG  }\\
		\cline{1-7}
	h & $L^2$ error & order & $L^2$ error  & order & $L^2$ error  & order \\
		\hline $1/20$ & $5.8300 \times 10^{-3}$ & - & $5.8295 \times 10^{-3}$ & - & $5.8295 \times 10^{-3}$ & - \\
		\hline $1/40$ & $3.6364 \times 10^{-4}$ & $4.0029199$ & $3.6361 \times 10^{-4}$ &
		 $4.0029180$ & $3.6361 \times 10^{-4}$ & $4.0029180$
		\\		
		\hline $1/80$ & $2.2582 \times 10^{-5}$ & $4.0092237$ &$2.2580 \times 10^{-5}$ &
		 $4.0092260$ &	$2.2580 \times 10^{-5}$ & $4.0092260$
		\\			
		\hline \hline
	\end{tabular}
	\caption{{\bf Forced VP:}  Convergence rates and errors $\|f(\tf)-f_h(\tf)\|_{0,\Th}$ for  $k=2$ (top) $k=3$ (bottom)  using different Poisson solvers (with $k+1$ polynomial spaces).}
	\label{tab:convergence}
\end{table} 
We now study the effect (in accuracy) of  using the different  Poisson solvers in the DG method for  the Vlasov Poisson system \eqref{test1:0}.Together with the LDG$(v)$ which gives the energy preserving scheme, we have also run the computations using the classical LDG (with numerical fluxes defined in \eqref{flux:DG0} and finite element spaces $(V_h^{k+1},V_h^{k+1})$ and the mixed finite element approximation \eqref{rt:a}-\eqref{rt:b} with the $RT_{k}$ (Raviart-Thomas) finite element spaces $(W^{k+1},V_h^{k})$. In Table \ref{tab:convergence} are given the $L^{2}$-errors  $\|f-f_h\|_{0,\Th}$ and convergence rates  for different mesh sizes and polynomial degrees $k=2$ and $k=3$. As it  can be observed, all the methods considered seem to yield approximations with the same accuracy and convergence properties. 
 	\begin{figure}[!htb]
		\centering
		\includegraphics[scale=0.6]{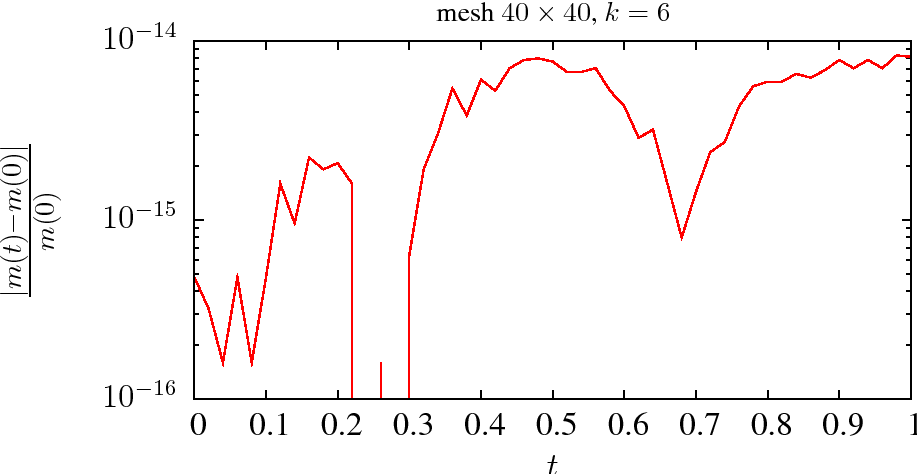} \hfill
				\includegraphics[scale=0.6]{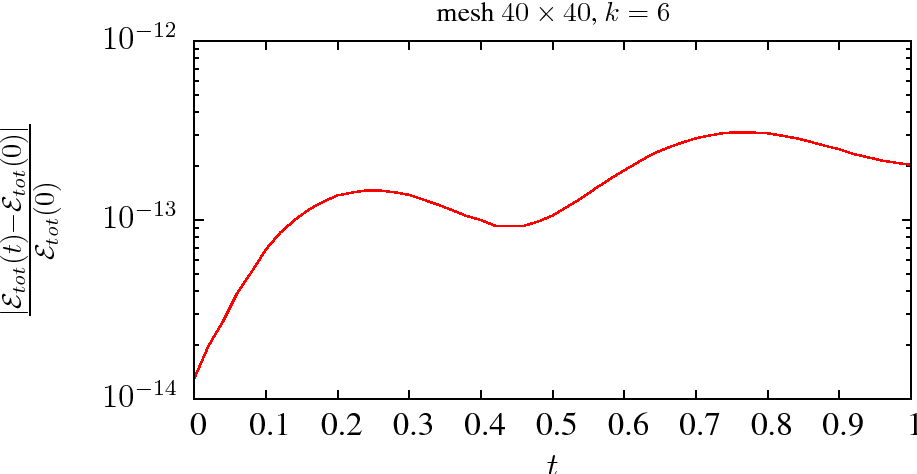} 			
		\caption{{\bf Forced VP: } Time evolution of the deviation from conservation of mass and total discrete energy}
		\label{fig:sim2d-vp1}
	\end{figure}
	We now study how the discrete mass and discrete energy are preserved in time. 
	In \autoref{fig:sim2d-vp1}, are depicted the time evolution of the relative error of the mass conservation (left), and energy conservation (right).  The corresponding diagram for \eqref{dev:l2} is given in Figure \ref{fig:sim2d-vp1b}. The results are obtained on a uniform mesh $40 \times 40$, using the DG energy preserving method with polynomial degree $k=6$. Notice that  the errors in the graphics are close to machine precision, which indicates the  ability of the method to {\it conserve}  the properties of the system.
	\begin{figure}[!htb]
		\centering
		\includegraphics[scale=0.6]{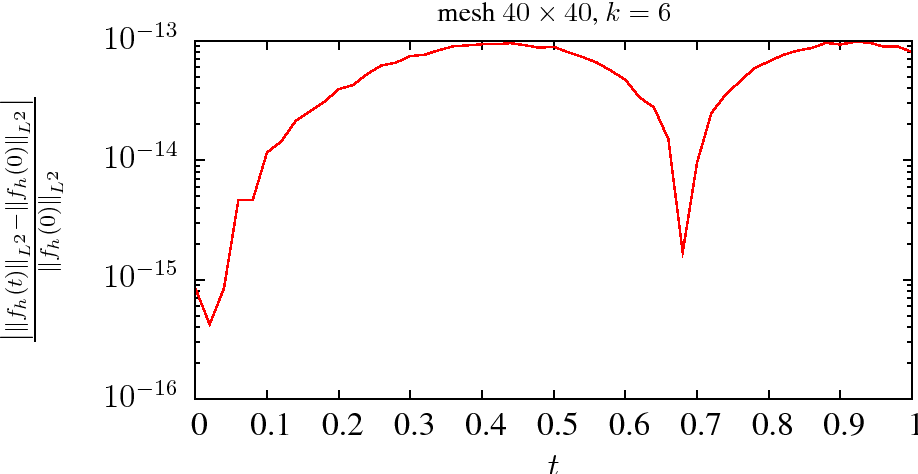} 	
		\caption{{\bf Forced VP:} Deviation from conservation of the $L^{2}$-norm of $f_h$.}
		\label{fig:sim2d-vp1b}
	\end{figure}	
\subsection{1D weak Landau damping}
\begin{figure}[!htb]
		\centering
	
		\includegraphics[scale=0.85]{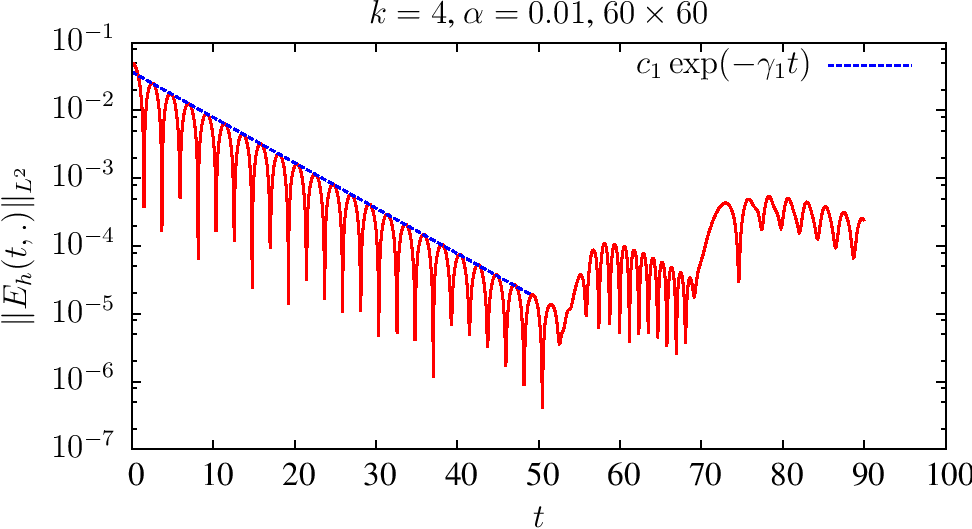} 	
		
		\caption{{\bf Weak Landau Damping:} Time evolution of the amplitude of the electrostatic field.}
		\label{fig:sim2d-vp-wnl1}
	\end{figure}
Typically, in most works in literature concerned with the simulation of  Landau damping (linear, weak or nonlinear) for the VP system, the  computational domains for the phase space $\Omega=\Omega_x\times \Omega_v$, are set to $\Omega_x=[0,4\pi]$ and $\Omega_v=[-5,5]$. However in our computations, we have found that using such $\Omega$, after some $t>0$, the approximate distribution function would not be of compact support in $v$. More precisely  we found that $\left. f_h(x,v,t) \right|_{v=\partial \Omega_v} \approx 10^{-5}$ for large time which is far from $0$, specially for high order accurate approximations as those considered in this paper. Therefore, to ensure the compact support in $v$ of $f_h$ we have set $\Omega_v=[-10,10]$ for our computations. 
In this case, $\left. f_h(x,v,t) \right|_{v=\partial \Omega_v} \approx 10^{-22}$ which can be  obviously  regarded as $0$.
 We take as initial data
\begin{equation}\label{landau:0}
	f(x,v,0) = \frac{1}{\sqrt{2 \pi}} \left( 1 + \alpha \cos(K \, x) \right) e^{- \frac{v^2}{2} } \qquad x\in \Omega_x \,\, v \in \Omega_v\;,
\end{equation}
where $\Omega_x=[0,4\pi]$ and $\Omega_v=[-10,10]$. In \eqref{landau:0},  $\alpha$ is the size of the perturbation and $K$ refers to the basic mode of the electrostatic field. Here, we have set $K=0.5$ and $\alpha=0.01$ so that the perturbation is small and therefore the linear theory can be used. We have computed the approximate solution $(E_h,f_h)$  over a mesh $60 \times 60$, using the energy preserving DG method with polynomials of degree $k=4$.
	\begin{figure}[!htb]
		\centering
		\includegraphics[scale=0.65]{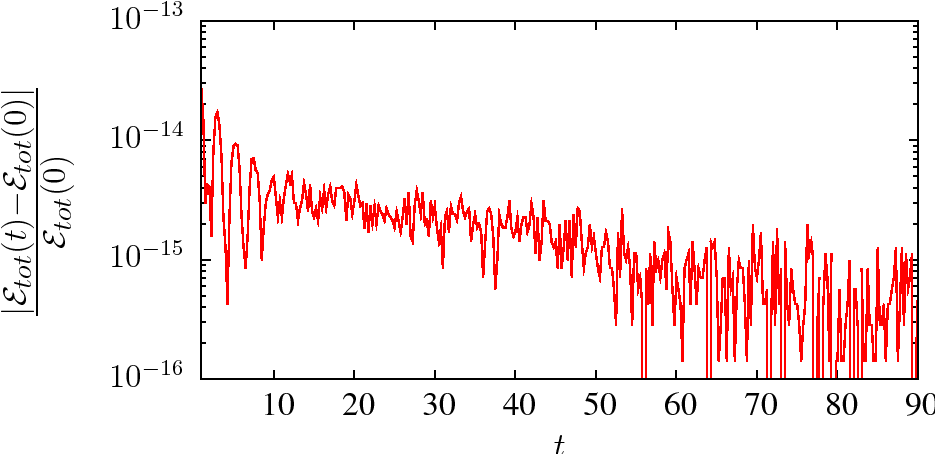}
		\hfill
		\includegraphics[scale=0.65]{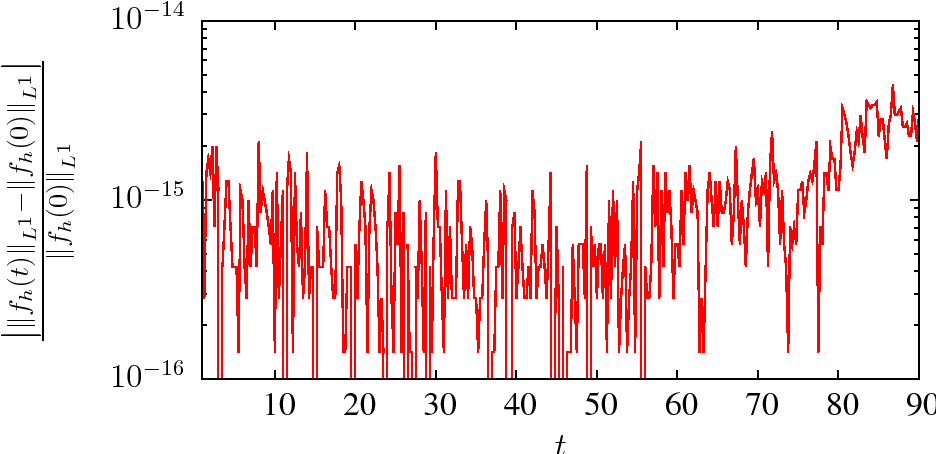}
		\caption{{\bf Weak Landau Damping:} Time evolution of the relative error (deviation) of the total energy  \eqref{dev:ener}(right) and  $L^1$ norm \eqref{dev:l1} (right) for the weak  Landau damping.}
		\label{fig:sim2d-vp-wnl1:b}
	\end{figure}
	
In Figure \ref{fig:sim2d-vp-wnl1} we plot (in a semi-log diagram) the time evolution of the $L^{2}$-norm of the electrostatic field $E_h(t)$. As it can be observed from the graphic, the amplitude of  the electrostatic field decreases exponentially in time up to some {\it recurrence time} $T_R$, after which it oscillates, in agreement with Landau linear theory. We have fitted the line (in log-scale) $c\, \exp(-\gamma t)$ at the local maximums $\|E_h(t)\|_{0}$.
The obtained damping rate $\gamma$ of the oscillations is $\gamma=-0.153272$, which is in good agreement with those results found in literature (compare to $-0.1533$ in \cite{seal}).\\
\begin{figure}[!htb]
		\centering
		\includegraphics[scale=0.65]{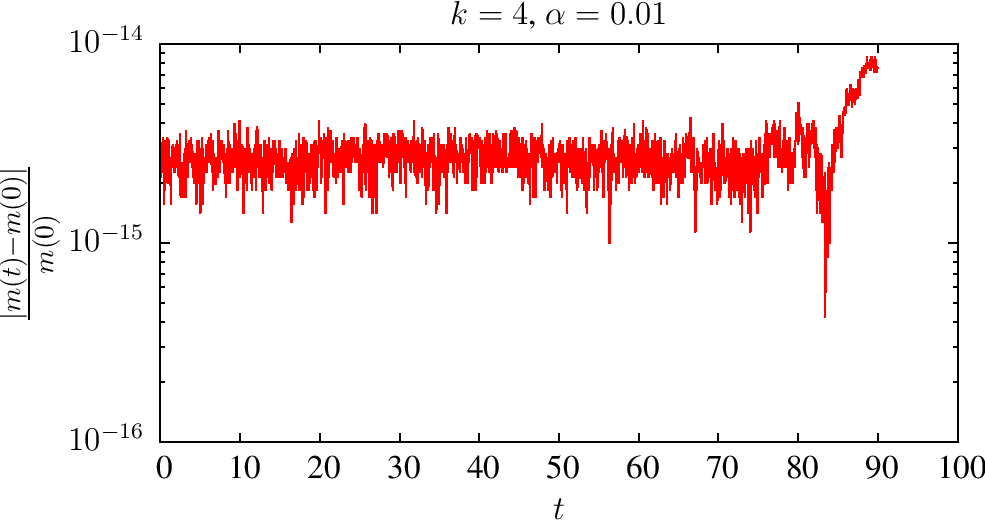}
		\caption{{\bf Weak Landau damping:} Mass conservation}
		\label{fig:sim2d-vp-wnl1:c}
\end{figure}
 
 \begin{figure}[!htb]
	\centering
	\includegraphics[scale=0.65]{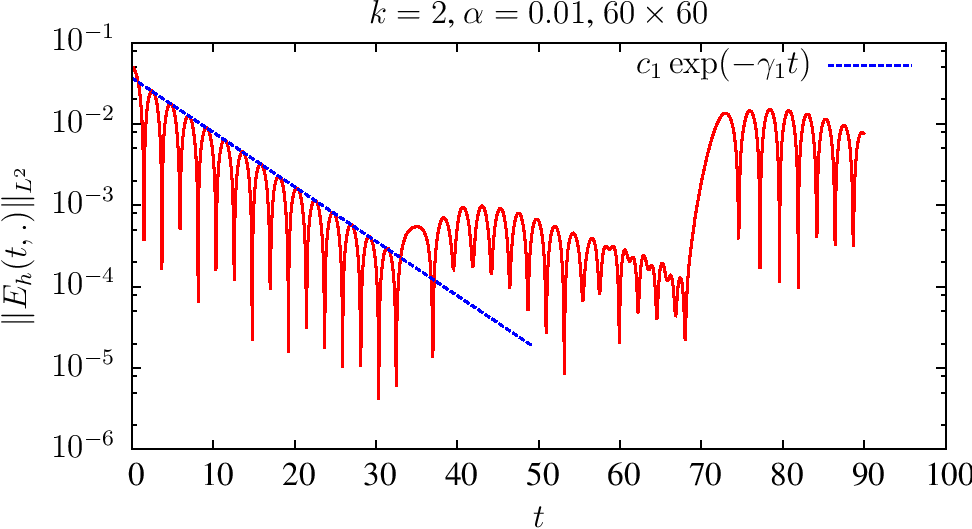}
	\hfill
	\includegraphics[scale=0.65]{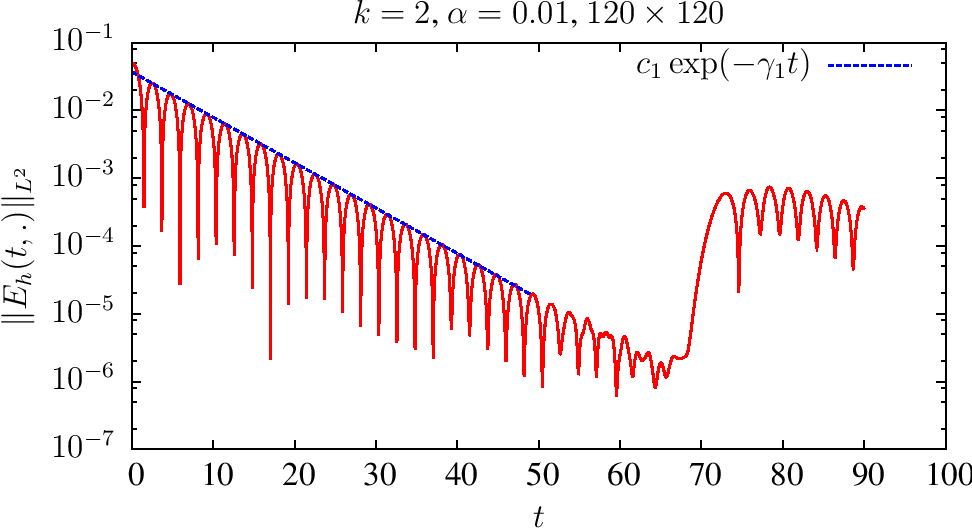}
	\\	
	\includegraphics[scale=0.65]{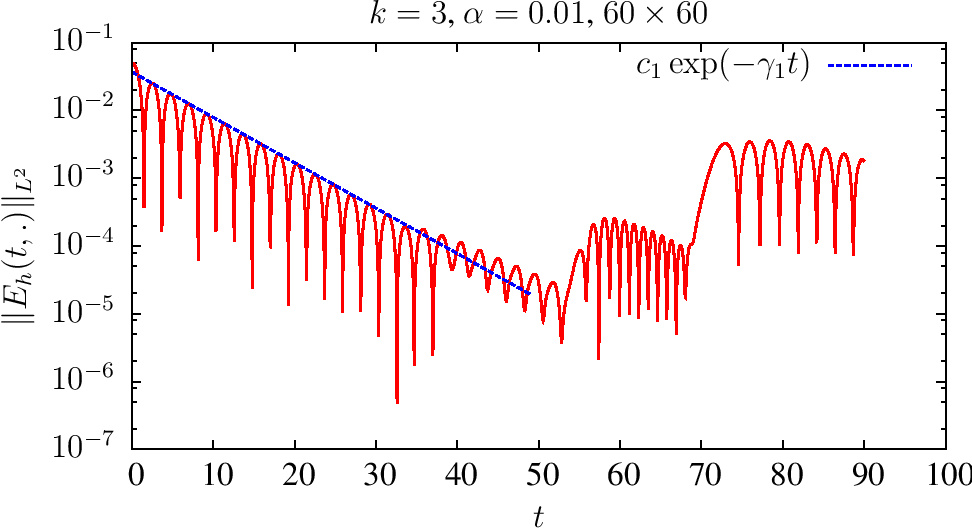}
	\hfill
	\includegraphics[scale=0.65]{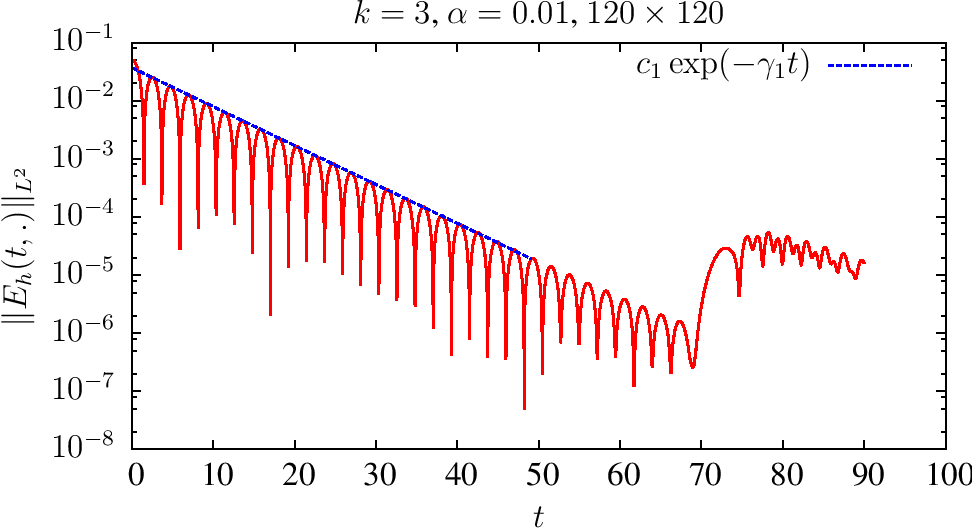}
	\\	
	\includegraphics[scale=0.65]{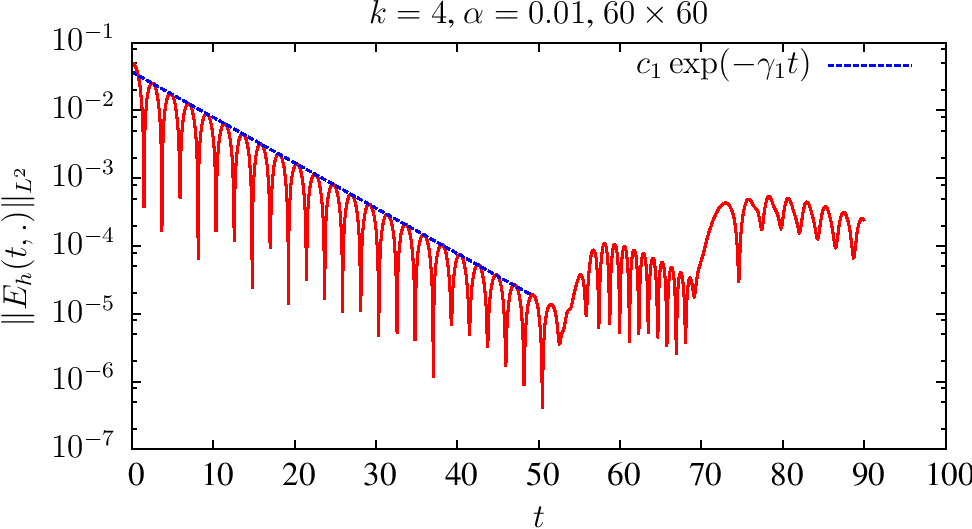}
	\hfill
	\includegraphics[scale=0.65]{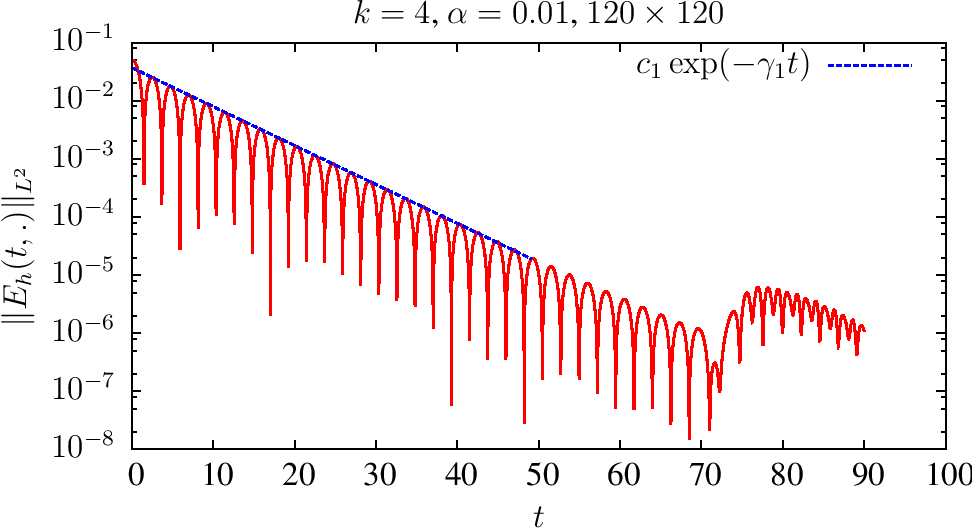}
	\caption{{\bf Weak Landau Damping}}
		\label{fig:sim2d-vp-wnl-order}
	\end{figure} 
	In Figure \ref{fig:sim2d-vp-wnl1:b} we have depicted the time evolution of the deviation from conservation of  the discrete total energy $\mathcal{E}_h(t)$ and the $L^{1}$-norm of the approximate distribution function $f_h$.  The corresponding diagram for the mass conservation is depicted in Fig. \ref{fig:sim2d-vp-wnl1:c}.
 Observe that the deviation from conservation for all these quantities is close to machine precision, which assess the good conservation properties of the scheme.
 
In \cite{filbet} the authors study the performance of eulerian solvers based on finite volumes and finite differences , and in particular they compare the estimated $T_R$  with the theoretical time predicted from the free streaming case given by the formula $T_R=\frac{2\pi}{k h_v}$. We have run some computations to see whether we could find a similar relation for the DG methods and so possibly  depending on the polynomial degree $k$. The results are given in  Fig. \ref{fig:sim2d-vp-wnl-order}. However, from these results it seems to us difficult to provide a closed formula or relation for the DG methods.
\subsection{Nonlinear (strong) Landau damping }\label{sec:nlp}
Nonlinear Landau damping is regularly used to assess the performance and properties of Vlasov-Poisson solvers (see \cite{Fijal99,filbet2, cr-son00,mcp0,BBBB09,seal,gamba0}. \\
We take as initial data $f(x,v,0)$ the function in \eqref{landau:0}, but we now set a larger amplitude of the initial perturbation of the density  $\alpha=0.5$ and take $K=0.5$. The computational domain is taken as for the weak case; $\Omega=\Omega_x\times \Omega_v$, with  $\Omega_x=[0,4\pi]$ and $\Omega_v=[-10,10]$. In this case, $\left. f_h(x,v,t) \right|_{v=\partial \Omega_v} \approx 10^{-22}$ for large $t$, and so the compact support of $f_h$ in $v$ is guaranteed.
 \begin{figure}[!htb]
		\centering
		\includegraphics[scale=0.85]{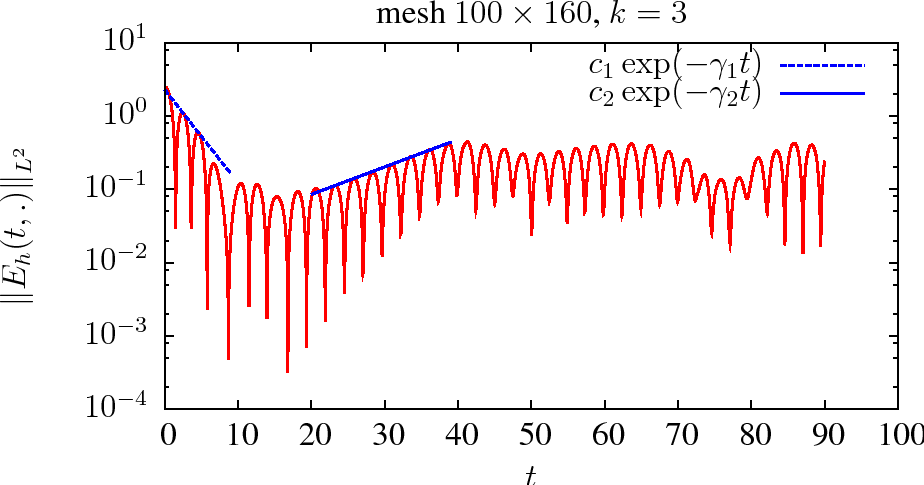} 
		\caption{{\bf Non-linear Landau Damping:} Time evolution of the amplitude of the electrostatic field. Estimated coefficients $c_1=2.279673$, $\gamma_1=-0.292285$, $c_2=0.015228$ and $\gamma_2=0.086126$.}
		\label{fig:sim2d-vp-nl0}
	\end{figure}
	
For this test, the Landau Linear theory cannot be applied since now the non-linear effects become important. However we will compare with other results obtained numerically in literature.
\begin{table}
	\centering
	\begin{tabular}{|l||c|c||c|c|} 
	\hline mesh & $\gamma_{\tt decay}$ & $c_{\tt decay}$ & $\gamma_{\tt growth}$ & $c_{\tt growth}$\\
	\hline $50 \times 80$ & -0.292286 & 2.279682 & 0.085114 & 0.015669 \\
	\hline $100 \times 160$ & -0.292285 & 2.279673 & 0.086126 & 0.015228 \\		
	\hline $150 \times 240$ & -0.292285 & 2.279673 & 0.086116 & 0.015232 \\
	\hline
	\end{tabular}
	\caption{{\bf Non-linear Landau Damping:} Estimated values for coefficients of the fitted functions $c \exp( -\gamma t)$; different mesh size.}
	\label{tab:coef}
\end{table} 

In Figure \autoref{fig:sim2d-vp-nl0} we plot the time evolution of the (log of the) $\|E_h(t)\|_{0}$, computed with the DG-DG(v) method over a mesh $100 \times 160$ using polynomials of degree $k=3$.  On the right diagram in the same figure, we have represented the corresponding time evolution of the deviation from conservation of the discrete total energy of the system \eqref{dev:ener}. Observe that the amplitude of the electrostatic field decreases  exponentially initially $t\in [0,10]$ and then  increases exponentially  (for $t \in [20,40]$) and after that it oscillates periodically. We have also depicted in the same diagram, the {\it lines } (in log-scale) $c \exp(- \gamma t)$ fitted at the local maximums  of $\Vert E(t,.) \Vert_{L^2}$ for $t \in [0,10]$ (initial decay) and for $t \in [20,40]$ (growth). The estimated coefficients and damping rates are $c_1=2.383814$, $\gamma_1=-0.305920$, $c_2=0.015360$ and $\gamma_2=0.085241$. They  are in good agreement with numerical simulations presented in the literature: the estimated $\gamma_{\tt decay}$ is same as the one reported in \cite{seal} $-0.292$;  and close to the one obtained in \cite{cheng}: $-0.281$. The growth rates are also in good agreement with those reported in literature (cf. \cite{seal}).
\begin{table}
	\centering
	mesh $100 \times 160$ \\
	\begin{tabular}{|l||c|c||c|c|} 
	\hline mesh & $\gamma_{\tt decay}$ & $c_{\tt decay}$ & $\gamma_{\tt growth}$ & $c_{\tt growth}$\\
	\hline $k=1$ & -0.280475 & 2.177524 & 0.087209 & 0.014650 \\
	\hline $k=2$ & -0.292584 & 2.244969 & 0.085946 & 0.015390 \\
	\hline $k=3$ & -0.292285 & 2.279673 & 0.086126 & 0.015228 \\		
	\hline
	\end{tabular}
	\caption{ {\bf  Non-linear Landau Damping:} Estimated values for coefficients of the fitted functions $c \exp( -\gamma t)$; different $k$.}
	\label{tab:coef-order}
\end{table} 	
We have also run this test using different meshes to study the effect of the mesh refinement  and of increasing the polynomial degree $k$ on the estimated values for the damping and increment rates. The estimated coefficients  of the fitted functions $c \exp(-\gamma t)$ are given in Table~\ref{tab:coef} (for the mesh refinement) and in Table~\ref{tab:coef-order} (for the increase in the polynomial degree).As expected,  by refining the mesh or by increasing the polynomial degree, the estimated coefficients have less error.\\

\begin{figure}[!htb]
		\centering
			\includegraphics[scale=0.65]{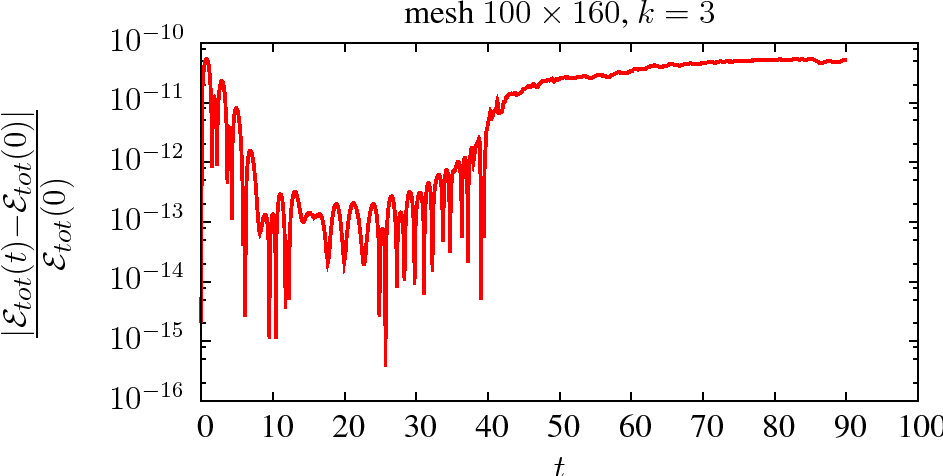} 	
			\includegraphics[scale=0.65]{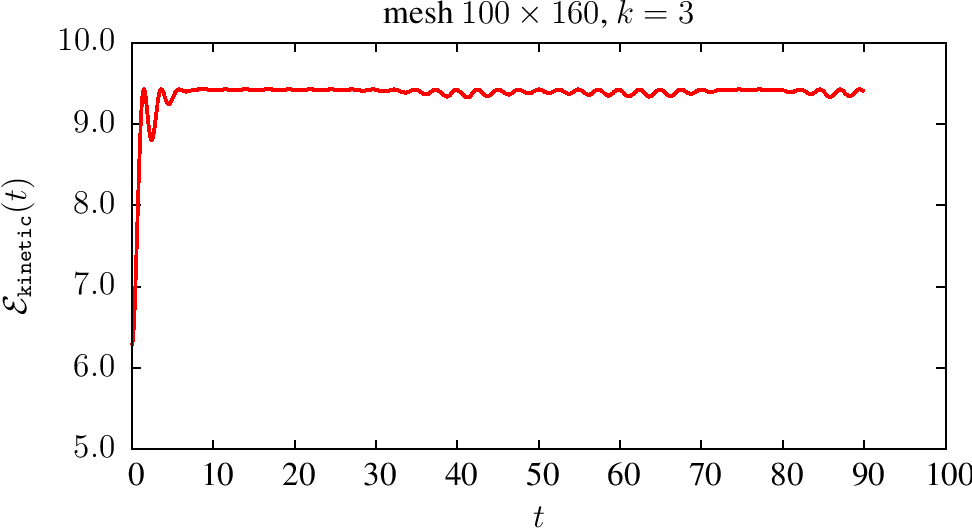} 		
		\caption{{\bf Non-linear Landau Damping: }Time evolution of the deviation from conservation of the total energy of the system.}
		\label{fig:sim2d-vp-nl0:b}
	\end{figure}
In Figure  \ref{fig:sim2d-vp-nl0:b}, we have represented the corresponding  time evolution of the deviation from conservation of the discrete total energy of the system \eqref{dev:ener}. 
 Note that up to $t \approx 10$ the deviation from conservation of total energy (quantity \eqref{dev:ener}) decreases to $10^{-12}$ and after that due to the process of filamentation there is a slowly  increment  until $t \approx 40$ when strong oscillations occur in $v$-direction. Therefore, the discrete total energy of the system is conserved with a relative error of order $10^{-10}$, which to our knowledge has not been obtained  before in literature. This property indicates that our scheme gives an accurate description of macroscopic values (physical quantities defined by the moments of the distribution function with respect to $v$).  \\
 \begin{figure}[!htb]
		\centering
		\includegraphics[scale=0.65]{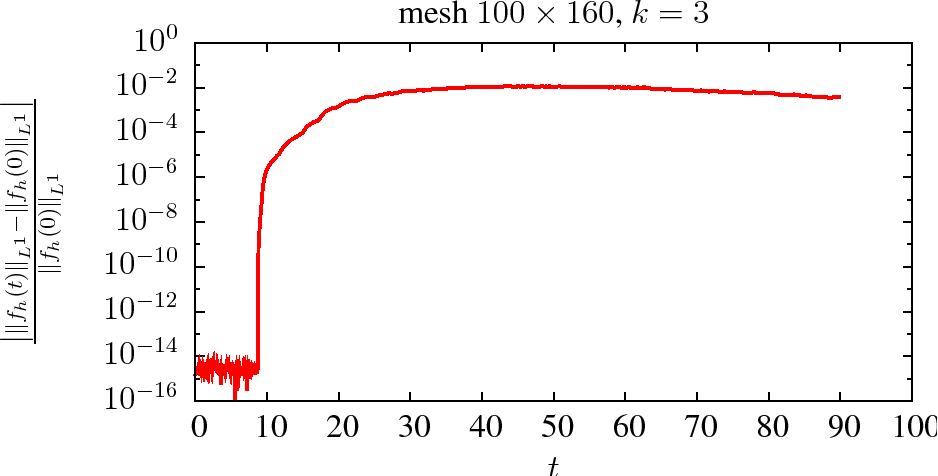} 
		\includegraphics[scale=0.65]{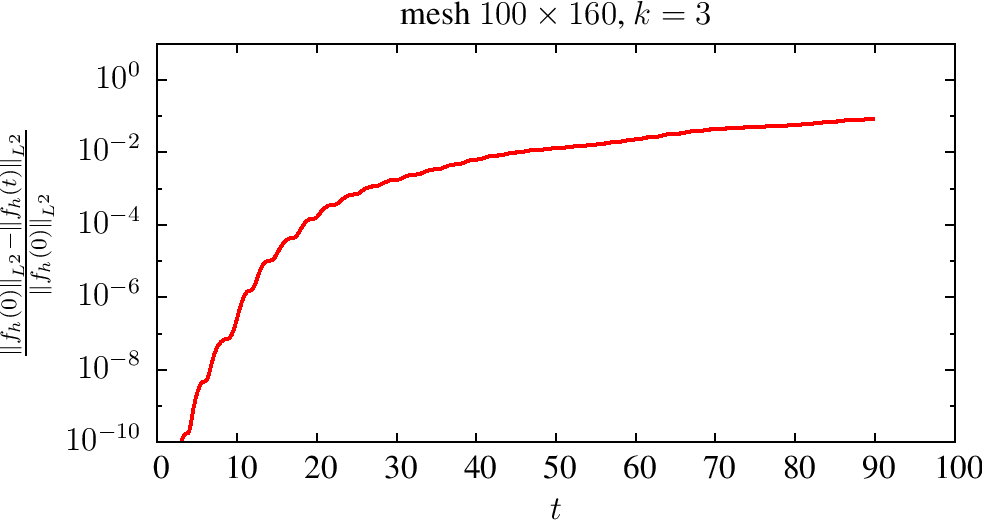} 
		\caption{ {\bf Non-linear Landau Damping: } The evolution of $\Vert f_h \Vert_{L^1}$ and $\Vert f_h \Vert_{L^2}$ in a semi-log scale where using a mesh $100 \times 160$.}
		\label{fig:sim2d-vp-nl1}
	\end{figure}
 The evolution in time of \eqref{dev:l1} and \eqref{dev:l2} is depicted in \autoref{fig:sim2d-vp-nl1}.\\
  \begin{figure}[!htb]
		\centering
		  \subfloat[$t=15$]{
		  \includegraphics[scale=0.32, angle=-90, clip=true, viewport=1.5in 1.5in 7.5in 10.5in]{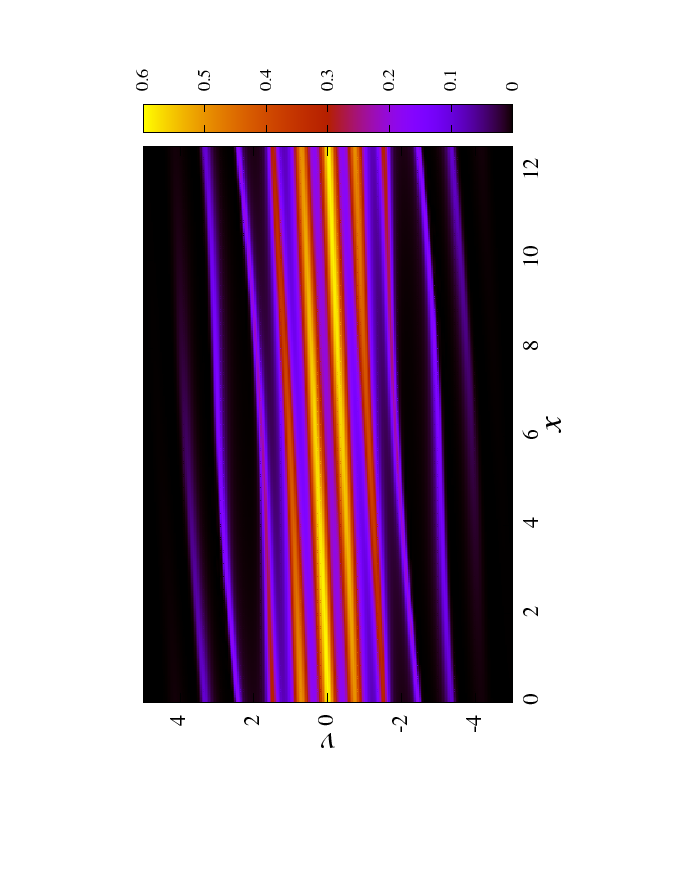}}                
		  \subfloat[$t=35$]{
		  \includegraphics[scale=0.32, angle=-90, clip=true, viewport=1.5in 1.5in 7.5in 10.5in]{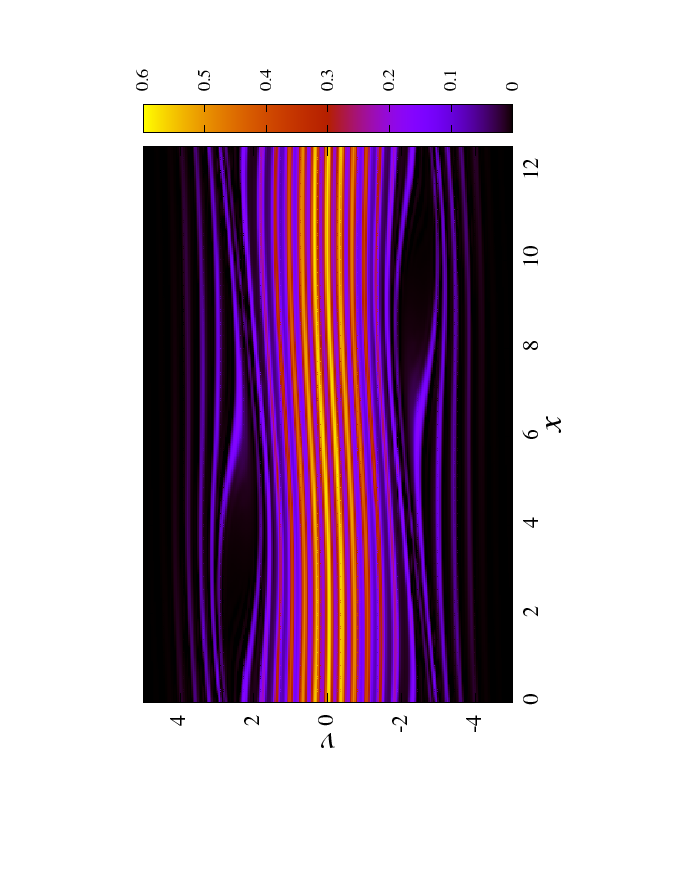}} \\                 
		  \subfloat[$t=45$]{
		  \includegraphics[scale=0.32, angle=-90, clip=true, viewport=1.5in 1.5in 7.5in 10.5in]{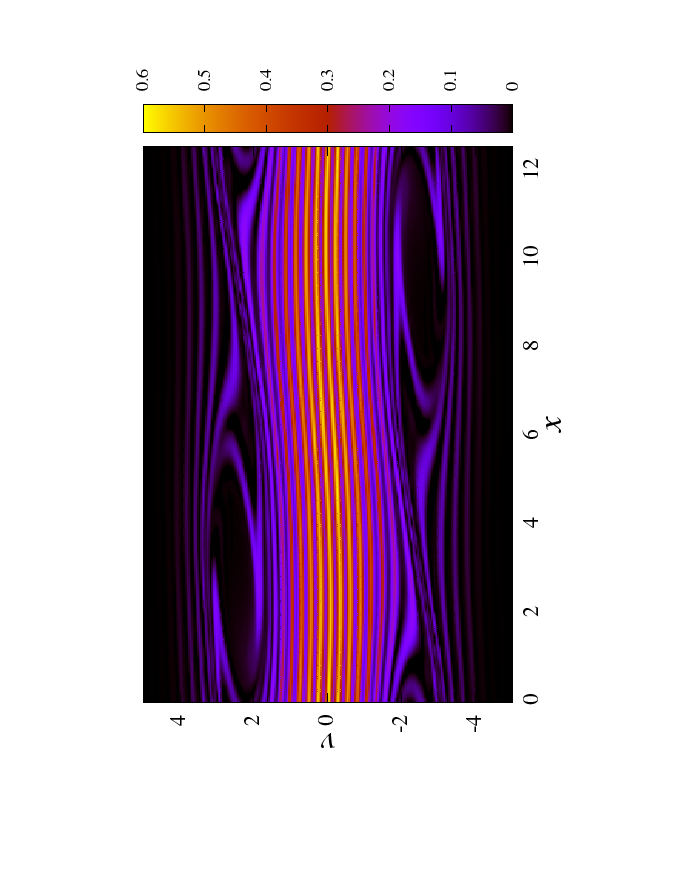}} 
		  \subfloat[$t=65$]{
		  \includegraphics[scale=0.32, angle=-90, clip=true, viewport=1.5in 1.5in 7.5in 10.5in]{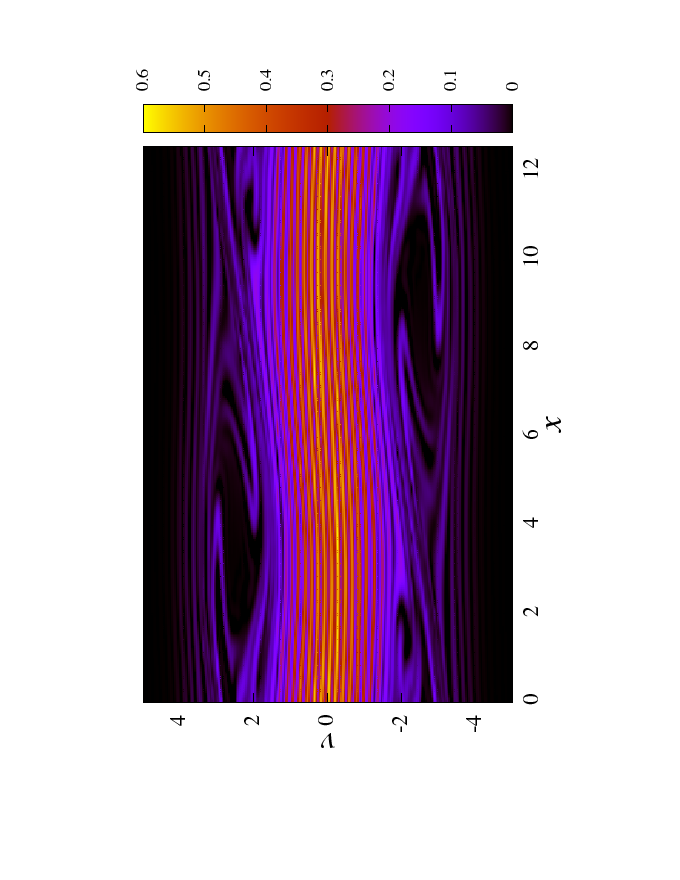}}                		
		\caption{ {\bf Non-linear Landau Damping:} approximate distribution function at different times, computed over a mesh $100 \times 160$ with the DG-LDG(v) method with $k=3$.}
		\label{fig:sim-vp-nl3}
	\end{figure}
In Figure \ref{fig:sim-vp-nl3} are represented in phase space the approximate distribution function $f_h$ at different times, obtained with the DG-DG(v) method over a mesh $100 \times 160$ using polynomials of degree $k=3$. From the figures it can be observed the detailed structure of the solution captured by the energy preserving DG method.
\begin{figure}[!htb]
		\centering
		  \subfloat[$t=15$]{
		  \includegraphics[scale=0.65, angle=0]{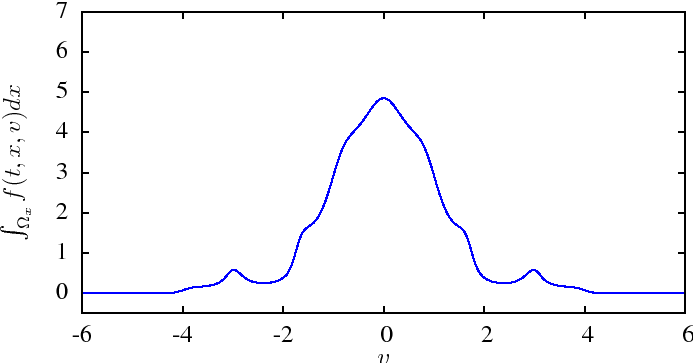}}                  
		  \subfloat[$t=45$]{
		  \includegraphics[scale=0.65, angle=0]{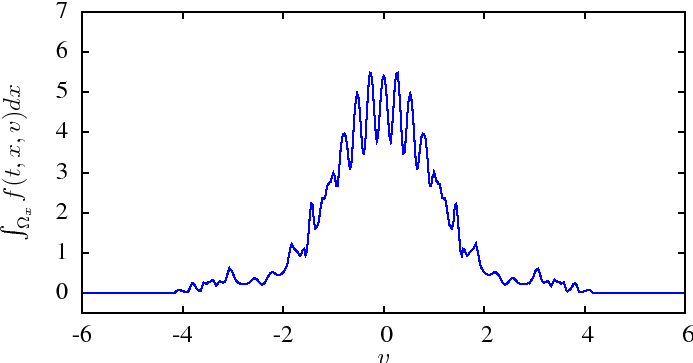}}                  \\
		  \subfloat[$t=65$]{
		  \includegraphics[scale=0.65, angle=0]{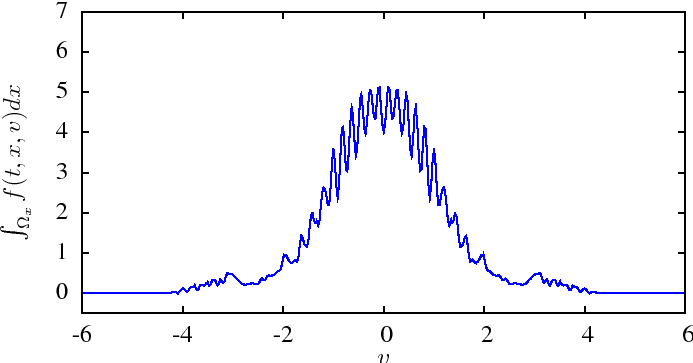}}                  
		  \subfloat[$t=90$]{
		  \includegraphics[scale=0.65, angle=0]{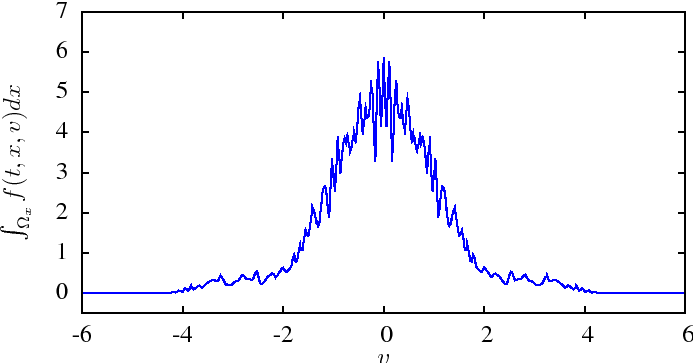}}                  		 	  
		\caption{{\bf Non-linear Landau Damping: } Profiles in $v$ of the approximate distribution function $f_h(t)$ at different times.}
		\label{fig:sim-vp-nl4}
	\end{figure}		
To asses the ability of the scheme to capture the possible strong oscillations in $v$-direction, we have also plotted, in Figure \ref{fig:sim-vp-nl4},  the corresponding profiles in $v$ of each $f_h(t)$, defined by:
\begin{equation*}
	\int_{\Omega_x} f_h(x,v,t) dx\;.
\end{equation*}
 We have also performed the simulations with the RK2-TVD time integrator, to study how the use of such integrator affects the accuracy and conservation properties of the DG-DG(v) method.  In Figure \ref{fig:sim2d-rk2tvd}, we have depicted the time evolution of the amplitude of the electrostatic field (left diagram) and the time evolution of \eqref{dev:ener} (right diagram), when using  $k=3$ over a mesh $100 \times 160$.  These graphics should be compared with Figures \ref{fig:sim2d-vp-nl0} and \ref{fig:sim2d-vp-nl0:b} (left), respectively. 
 \begin{figure}[!htb]
		\centering
		\includegraphics[scale=0.6]{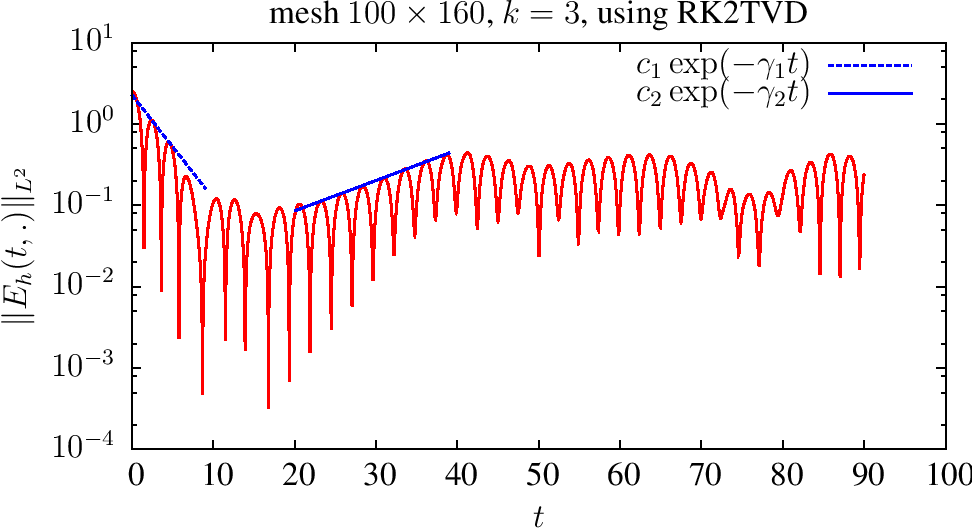} 	
				\hfill
			\includegraphics[scale=0.6]{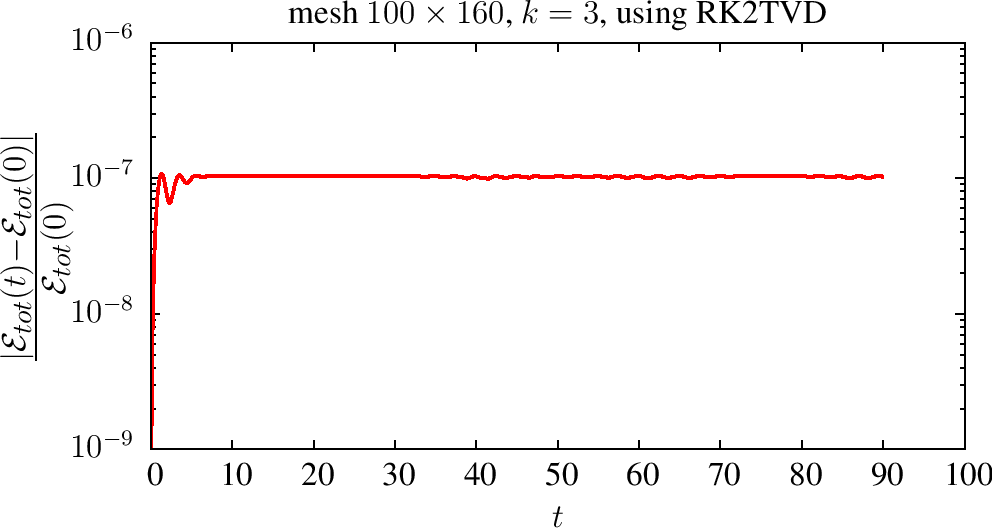} 	
		\caption{{\bf Non-linear Landau Damping:} results with RK2-TVD time integrator}
		\label{fig:sim2d-rk2tvd}
	\end{figure}	
As regards the energy conservation, notice that the errors are higher than those obtained with the RK4, although this can be an effect of the lower order of  the time integration used. In any case, the use of TVD integration does not seem to show any special advantage for this problem.

We now study and compare the  conservation properties of the DG method when using the two approaches  \eqref{a1} and \eqref{a2} for defining the numerical flux $\widehat{E_hf_h}$. In Figure \ref{fig:sim2d-vp-nl-p0} are depicted the time evolution of the error in the total energy, obtained with both approaches for  $k=3$ (upper diagrams). The result obtained with the flux defined as in \eqref{a1} are represented on the left diagram. The one obtained with the weighted average modification given in \eqref{a2} are given on the right diagrams. On the bottom diagram of the same figure is given the result corresponding to the use of the non-consistent definition \eqref{aa00}. As can be appreciated from the results, the use of either \eqref{a1} or \eqref{a2}  leads to the conservation of the total energy, while the definition \eqref{aa00} does not. These results  confirm Theorem \ref{teo:ene}.
 Although not reported here, the same set of experiments was run with $k=1,2$.  For the lowest order $k=1$, similar results were obtained for all the approaches, with no significant differences among them.
	\begin{figure}[!htb]
		\centering
		\includegraphics[scale=0.65]{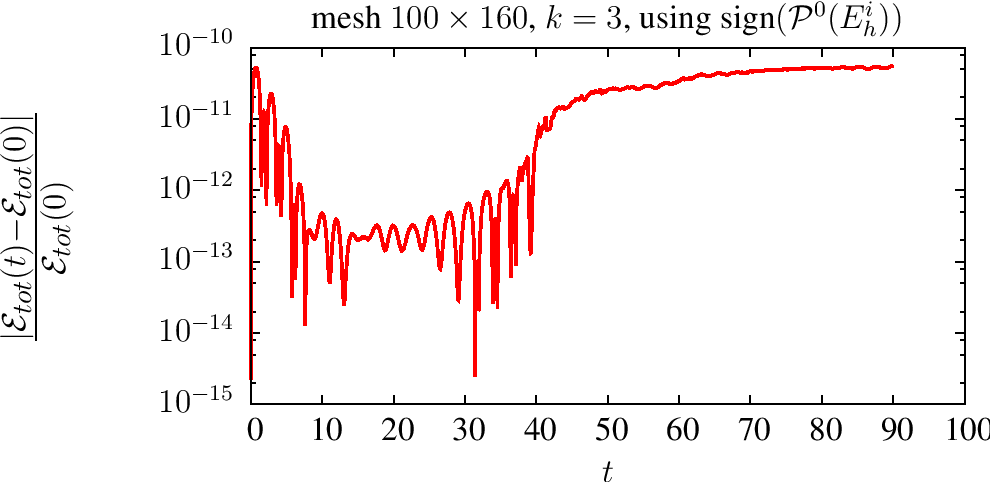}	
		\hfill
		\includegraphics[scale=0.65]{figs/2d-vp-nl/cenerg.png} 			
		\\ \vspace{0.5cm}
		\includegraphics[scale=0.65]{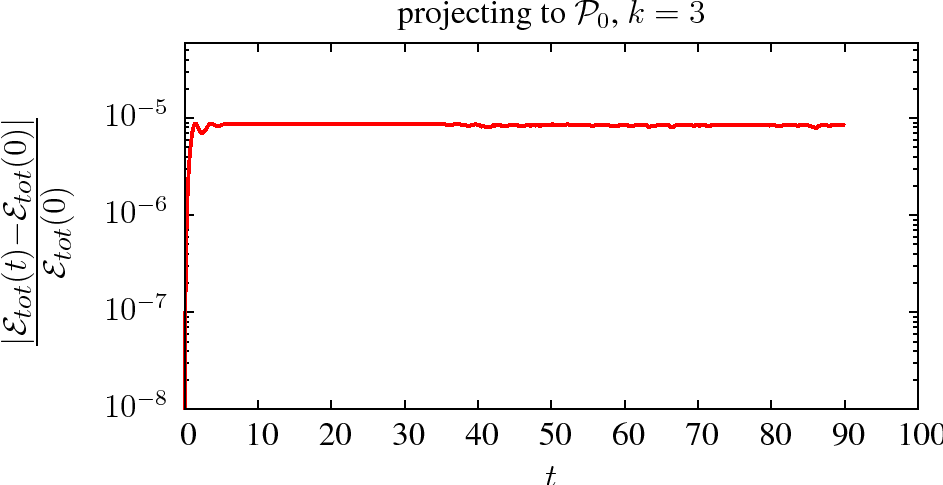}
		\caption{{\bf Non-linear Landau damping:} time evolution of the deviation \eqref{dev:ener} of the discrete total energy computed over a mesh $100 \times 160$ with $k=3$ with the different implementations of the numerical flux $\widehat{E_h^{i}f_h}$. Definition \eqref{a1} (left  upper diagram); definition \eqref{a2} (right upper diagram), non-consistent definition \eqref{aa00} (bottom diagram)}
		\label{fig:sim2d-vp-nl-p0}
	\end{figure}
\subsection{Two stream instability}\label{sec:twi}
This is a standard benchmark for checking the reliability of the schemes to face the strong oscillations. We have set $\Omega_x = [0,4 \pi]$ and $\Omega_v = [-10,10]$ (as discussed in \ref{sec:nlp}, to ensure the approximation is  compactly supported in $\Omega_v$ at all times of the computation) and take as initial data for the VP system \eqref{mod01}-\eqref{rho},
\begin{equation}
	f(x,v,0) = \frac{v^2}{\sqrt{8 \pi}} \left\lbrace 2 - \cos(K (x-2\pi)) \right\rbrace e^{ {- \frac{v^2}{2}} } \quad (x,v)\in \Omega\;,
\end{equation}
with  $K=0.5$. In Figure \ref{fig:sim-vp-2s0} are represented the approximate solutions (in phase space) obtained  at different times $t=15,30,45$ and $60$ with the energy preserving DG method using polynomial degree $k=3$ over a mesh $150 \times 150$.  The evolution in time of the deviation from conservation of the discrete total energy \eqref{dev:ener} and  the $L^{2}$-norm of the solution \eqref{dev:l1} are depicted respectively, on the left and right diagrams in Figure  \ref{fig:sim-vp-2s1}. Notice that also in this case, the total energy of the system is preserved up to machine precision. 

		\begin{figure}[!htb]
		\centering
		  \subfloat[$t=15$]{
		  \includegraphics[scale=0.32, angle=-90, clip=true, viewport=1.5in 1.5in 7.5in 10.5in]{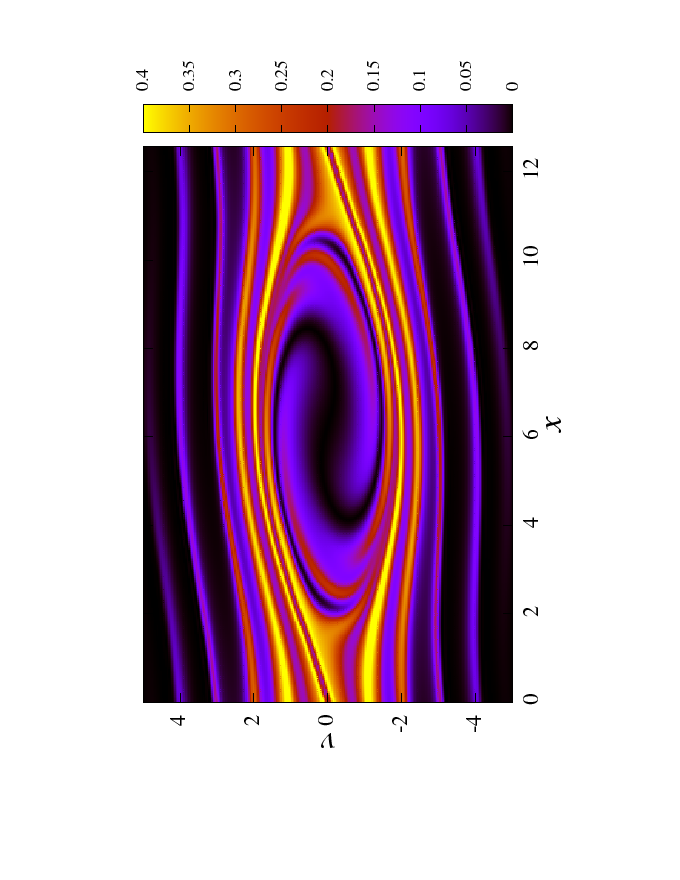}}                
		  \subfloat[$t=30$]{
		  \includegraphics[scale=0.32, angle=-90, clip=true, viewport=1.5in 1.5in 7.5in 10.5in]{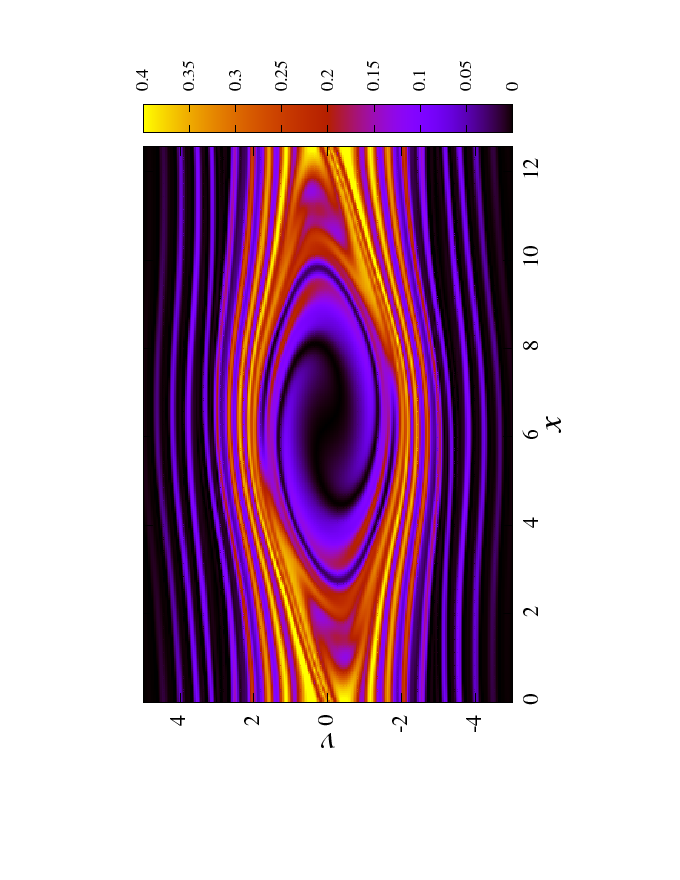}} \\ 	  
		  \subfloat[$t=45$]{
		  \includegraphics[scale=0.32, angle=-90, clip=true, viewport=1.5in 1.5in 7.5in 10.5in]{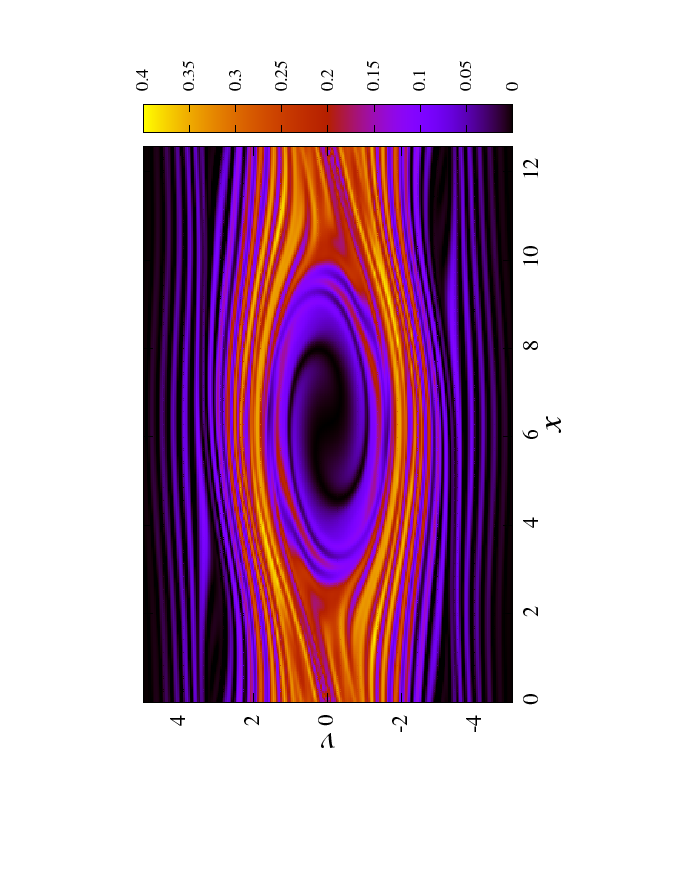}}                
		  \subfloat[$t=60$]{
		  \includegraphics[scale=0.32, angle=-90, clip=true, viewport=1.5in 1.5in 7.5in 10.5in]{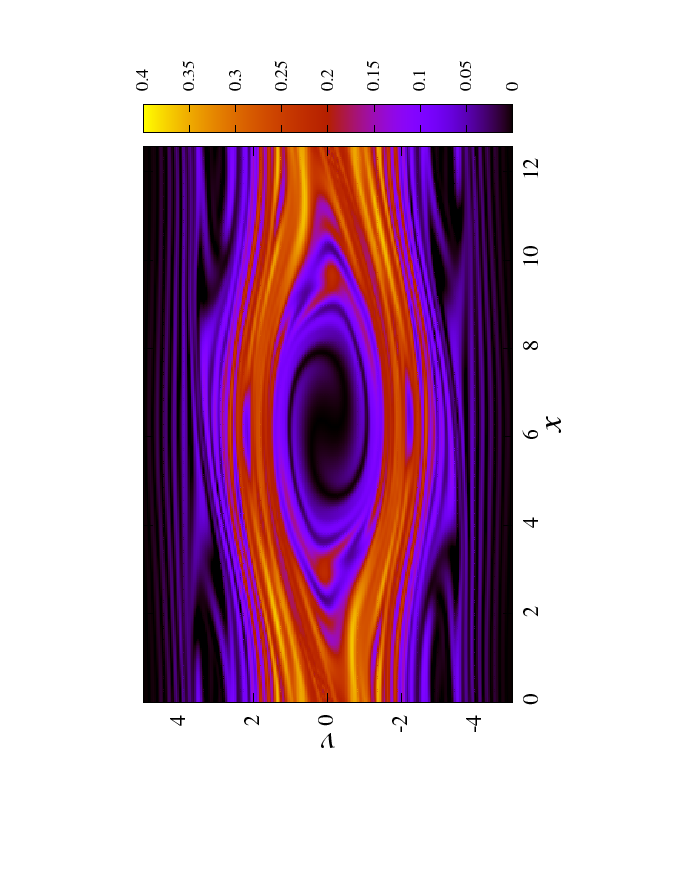}} 	  
		\caption{{\bf Two stream instability:} solution of the VP system for two stream instability at different times using mesh $100 \times 160$ and $k=3$.}
		\label{fig:sim-vp-2s0}
	\end{figure}
			\begin{figure}[!htb]
		\centering
		  \includegraphics[scale=0.65, angle=0]{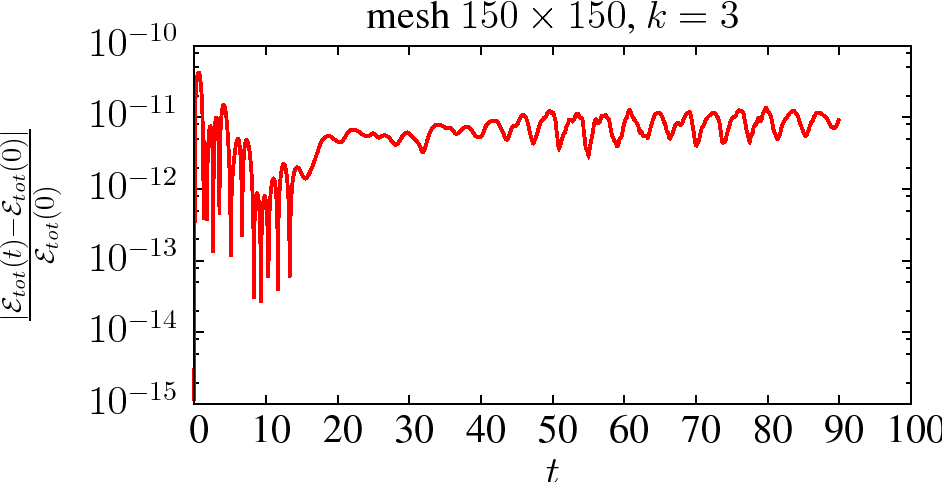}                  
		  \includegraphics[scale=0.65, angle=0]{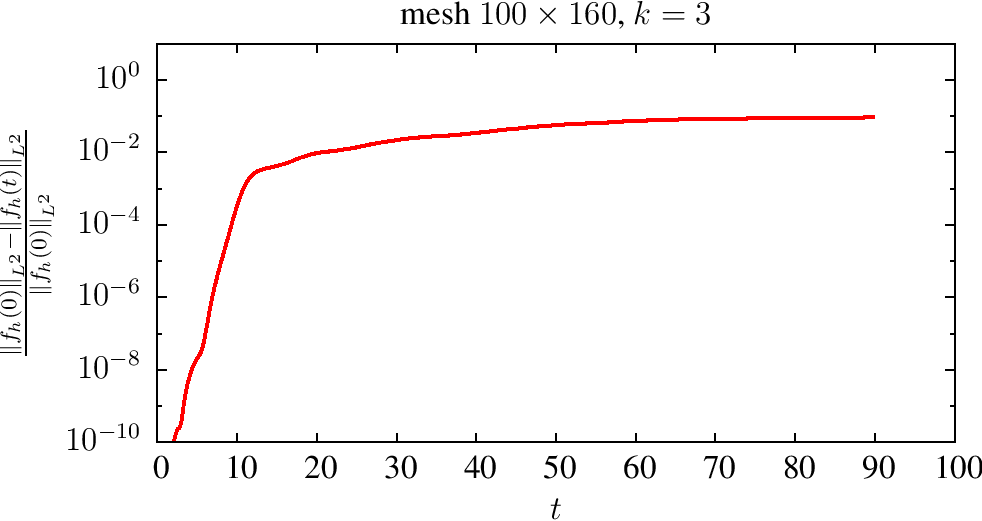}		  
		\caption{{\bf Two stream instability:} evolution of the relative error in $L^2$ and total energy in a semi-log scale}
		\label{fig:sim-vp-2s1}
	\end{figure}

We now compare the effect of using different Poisson solvers in the conservation properties of the final DG scheme for  the Vlasov Poisson system. In Figure \ref{fig:sim-vp-2s-comp-pois} are depicted the time evolution of the deviation from conservation of total energy  \eqref{dev:ener} obtained when using the mixed finite element approximation \eqref{rt:a}-\eqref{rt:b} with the $RT_{k}$ (Raviart-Thomas) finite element spaces $(W^{k+1},V_h^{k})$ ($-\,-\,-$);  the classical LDG (with numerical fluxes defined in \eqref{flux:DG0} and finite element spaces $(V_h^{k+1},V_h^{k+1})$ ($-+-+$) and the LDG$(v)$ method with fluxes \eqref{flux:v} ($-\circ--\circ--$).  On the left graphic are given the results obtained over a fixed mesh $100\times 160$ with $k=1$; on the right graphic those obtained with $k=3$ and over a mesh $40\times 40$. From these graphics it can be appreciated that by using $k=1$  the total energy is not conserved  and the error committed is of  $5$ orders of magnitude higher than for $k=3$ even for a fine mesh. Notice also that for $k=3$  even when using a very coarse mesh ($40\times 40$) the total discrete energy is conserved with error of order $O(10^{-7})$.

	\begin{figure}[!htb]
		\centering
		  \includegraphics[scale=0.65, angle=0]{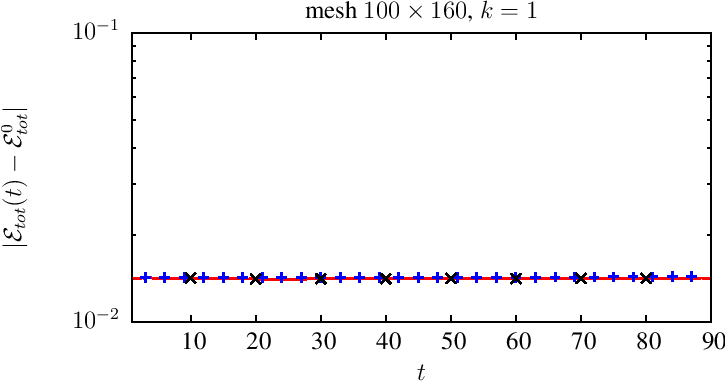}     
		  \hfill             
		  \includegraphics[scale=0.65, angle=0]{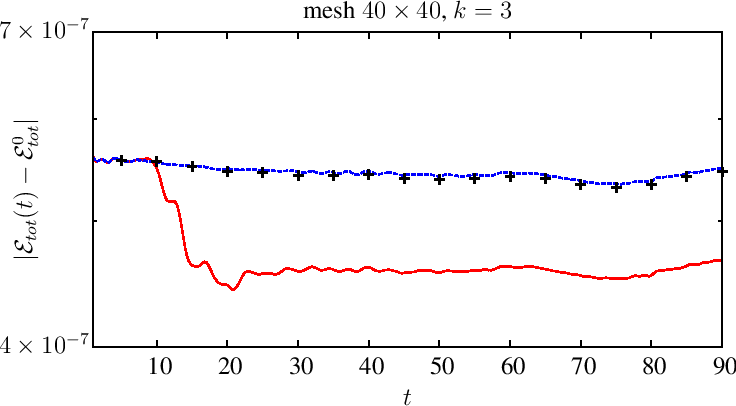}     
		  \hfill              		  
		\caption{{\bf Two stream instability:} time evolution of the deviation \eqref{dev:ener} for different Poisson solvers; the red line is obtained using LDG$(v)$, the cross symbols are $RT_k$ and plus symbols are classical LDG.  $k=1$ (left) and $k=3$ (right).}
		\label{fig:sim-vp-2s-comp-pois}
	\end{figure}

We now study the effect of mesh refinement for the conservation properties of the scheme.
Here, we check the effect of refininging $x$ and $v$ separately. In Figure  \ref{fig:sim-vp-2s2}  are depicted the time evolution of \eqref{dev:ener} (top graphics), \eqref{dev:l1} (center) and \eqref{dev:l2} (bottom graphics) for $k=3$. The graphics on the left show the effect of refinement of the $x$   variable. On the right are given those obtained by refining only the $v$ variable.
From the  graphics it can be observed that the $L^1$ and $L^2$ norms of $f_h$ 
depend more on refinement of the $v$ variable but they seem to be insensitive to the refinement on the $x$ direction.For the total discrete energy, however, the opposite effect can be appreciated from the top graphics. It seems to depend more on the refinement in the $x$ variable (which is expected since the potential energy comes from the Poisson coupling). 
	\begin{figure}[!htb]
		\centering
		  \includegraphics[scale=0.65, angle=0]
		  {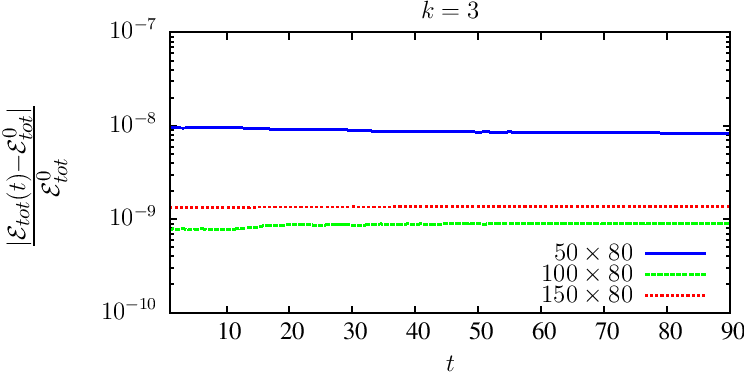}  
		   \hfill 
		  \includegraphics[scale=0.65, angle=0]
		  {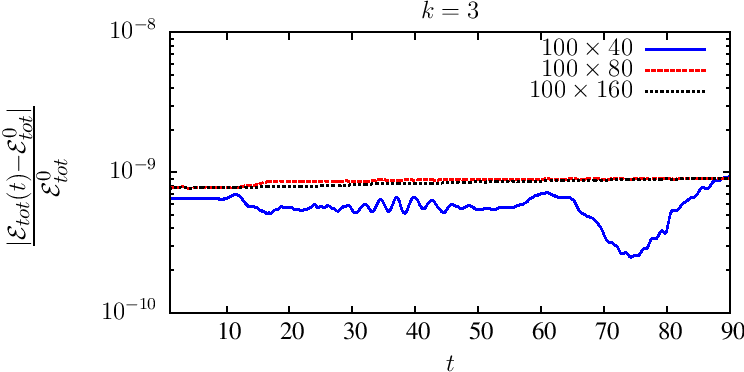}
		  \\
		  \includegraphics[scale=0.65, angle=0]{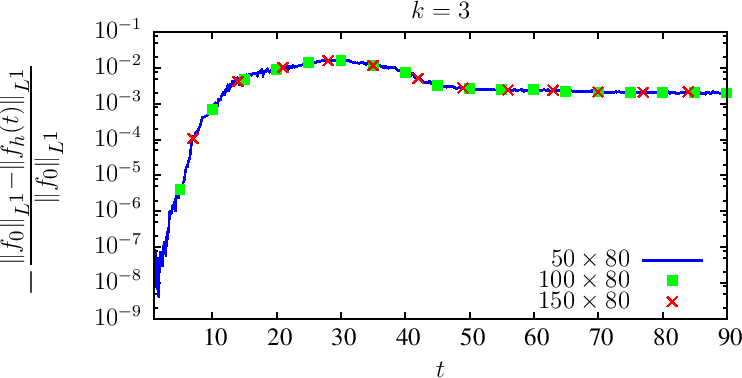}
		  \hfill
		  \includegraphics[scale=0.65, angle=0]{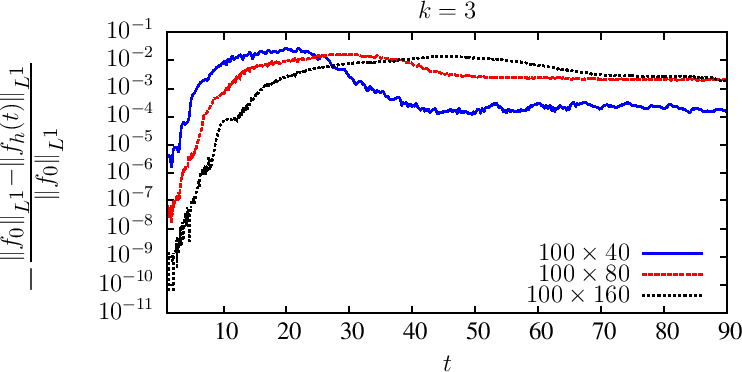}
		  \\
		  \includegraphics[scale=0.65, angle=0]{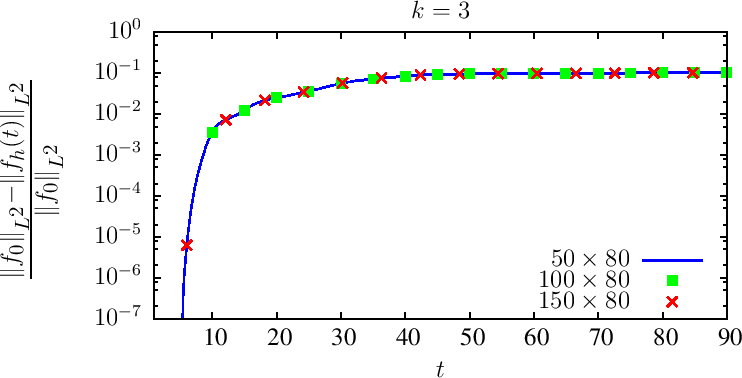}
		  \hfill
		  \includegraphics[scale=0.65, angle=0]{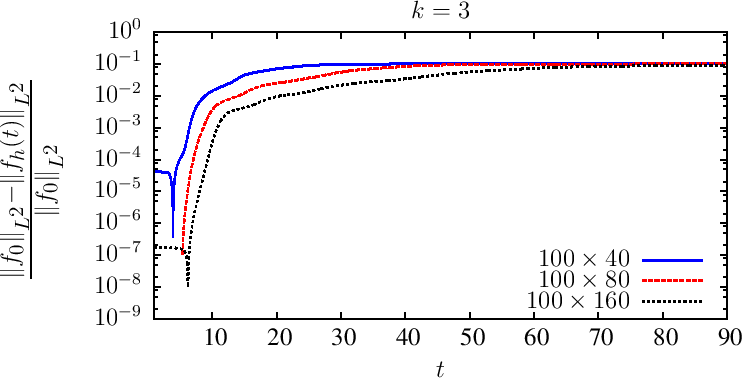}       
		  \caption{{\bf Two stream instability:} time evolution of the deviation from conservation of the total energy \eqref{dev:ener} (top graphics), $L^{1}$-norm of $f_h$ \eqref{dev:l1} (center) and $L^{2}$ norm of $f_h$ (bottom graphic). Diagrams on the left show refinement in $x$; on the right, refinement in $v$.}
		\label{fig:sim-vp-2s2}
	\end{figure}
\subsection{Two stream instability II}
We set now $\Omega_x = [0, 13\pi]$ and $\Omega_v = [-8,8]$ and take the initial data as in \cite{qiuCW}
\begin{equation}
	f(x,v,0) = \frac{( 1 + 0.05 \cos(K\,x) )}{2 v_{th} \sqrt{2 \pi} } 
	\left(
	\mbox{exp}\left(  - \frac{(v - w)^2}{2 v_{th}^2 } \right) +
	 \mbox{exp}\left( - \frac{(v + w)^2}{2 v_{th}^2 } \right)
	 \right) 	 \quad (x,v)\in \Omega \;,
\end{equation}
with  $v_{th} = 0.3$, $w=0.99$ and $K=\frac{2}{13}$. The initial data for this test consists of the two unstable flow moving in the opposite direction of each other.  The approximate distribution functions $f_h$ obtained at time $t=70$, with the DG-LDG(v) using a mesh $256 \times 100$  and $k=1,2,3$  are represented in Figure \ref{fig:sim-vp-2sII2}.

\begin{figure}[!htb]
		\centering
  		  \subfloat[$k=1$]{
		  \includegraphics[scale=0.32, angle=-90, clip=true, viewport=1.5in 1.5in 7.in 10.5in]{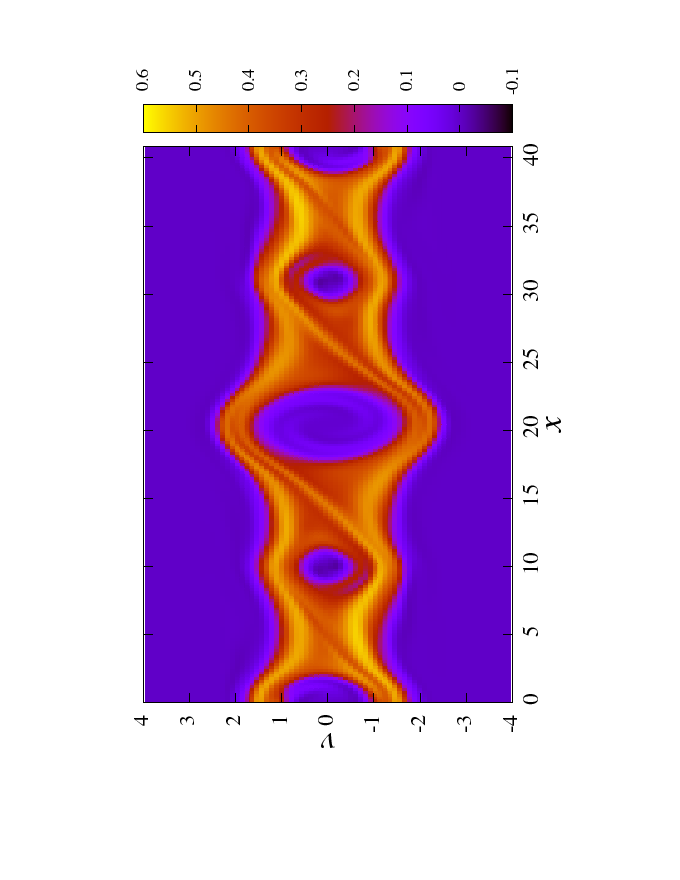}} 
		  \hfill		
		  \subfloat[$k=2$]{
		  \includegraphics[scale=0.32, angle=-90, clip=true, viewport=1.5in 1.5in 7.in 10.5in]{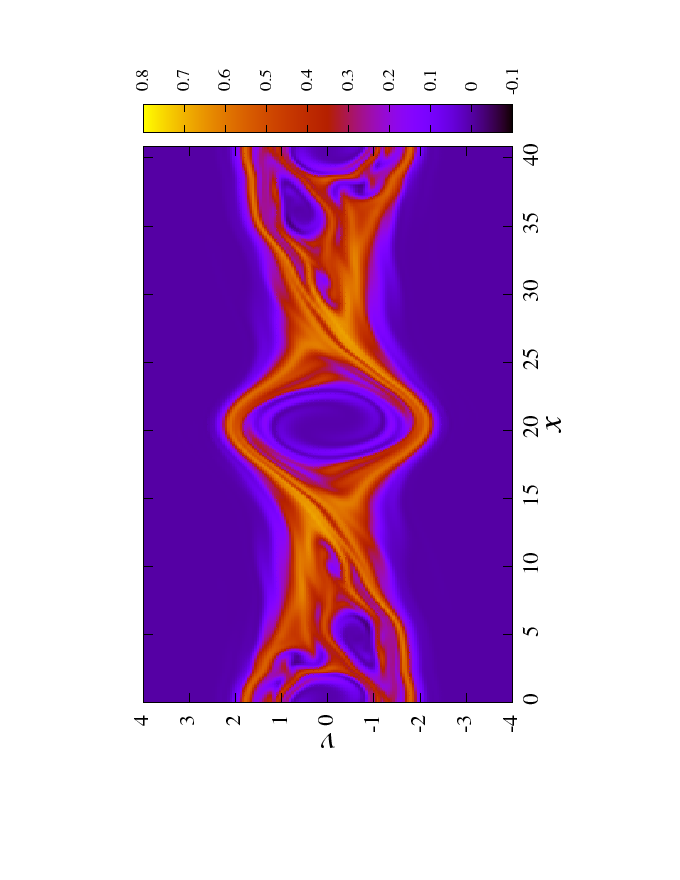}}                
		  \\
		  \subfloat[$k=3$]{
		  \includegraphics[scale=0.32, angle=-90, clip=true, viewport=1.5in 1.5in 7.in 10.5in]{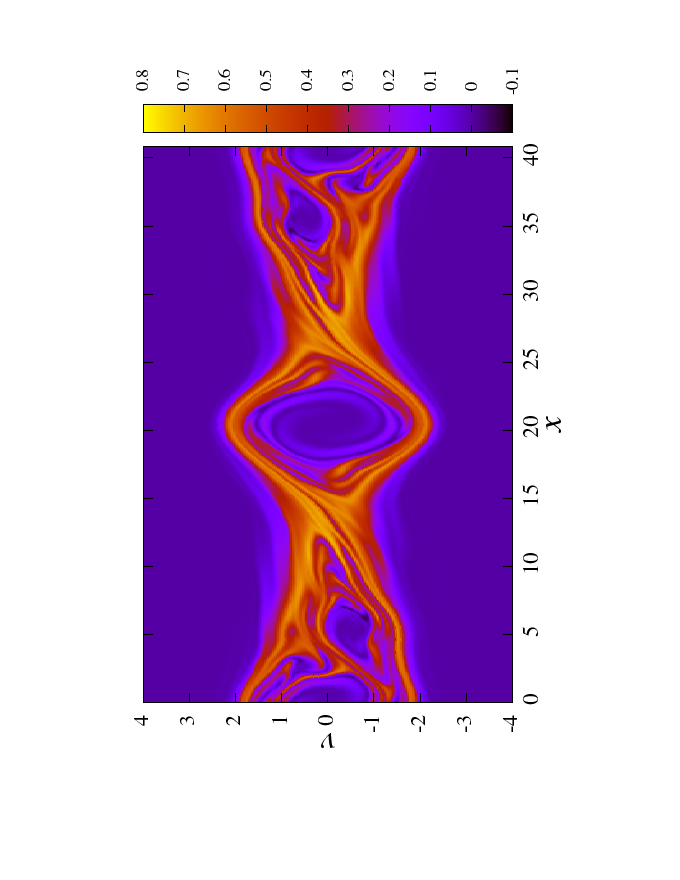}} 
		  \hfill
		  \subfloat[$k=6$]{
		  \includegraphics[scale=0.32, angle=-90, clip=true, viewport=1.5in 1.5in 7.in 10.5in]{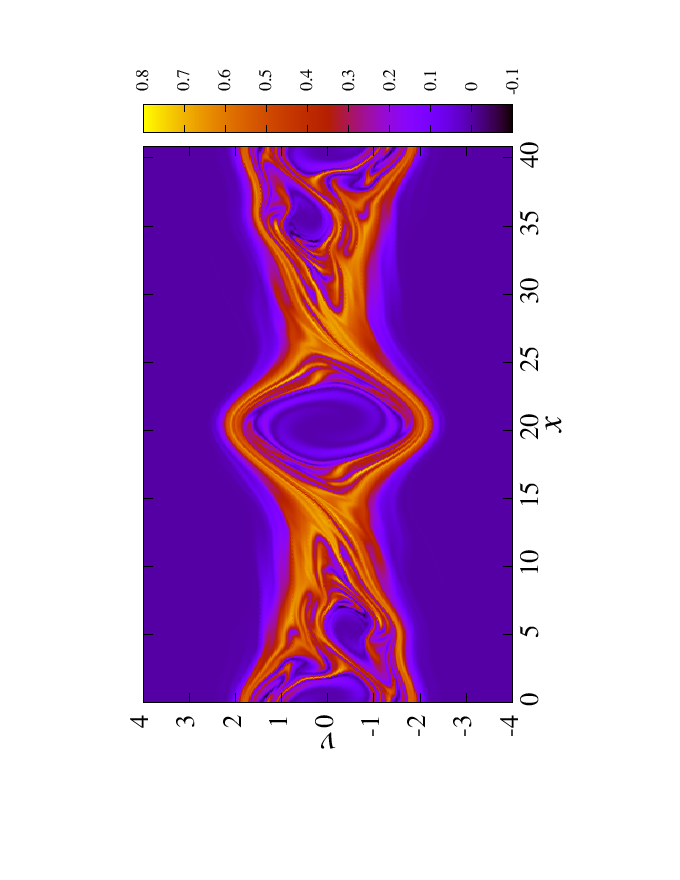}} 		
		  \caption{{\bf Two stream instability II:} solution  at $t=70$ using mesh $256 \times 100$.}
		\label{fig:sim-vp-2sII2}
	\end{figure}

In Figure \ref{fig:sim-vp-2sII1} are depicted the time evolution of the deviation from conservation of the total discrete energy. The left diagram shows the effect of mesh refinement in the quantity \eqref{dev:ener}  for the DG-LDG(v) method with $k=3$.  On the diagram on the right, we have represented the  time evolution of  such deviation for different polynomial degrees  $k=1,2,3$ using a fixed mesh $256\times 100$.  From the right diagram it can be appreciated that $k=1$  does not yield to the conservation of the total energy; in fact the error is almost 10 orders of magnitude higher than for $k\ge 2$, even on a very fine mesh. This result indicates and confirms that the hypothesis $k\geq 2$ in Theorem \ref{teo:ene} is indeed necessary for the DG-LDG(v) scheme to preserve of the total energy of the system.
\begin{figure}[!htb]
		\centering
		  \includegraphics[scale=0.65, angle=0]{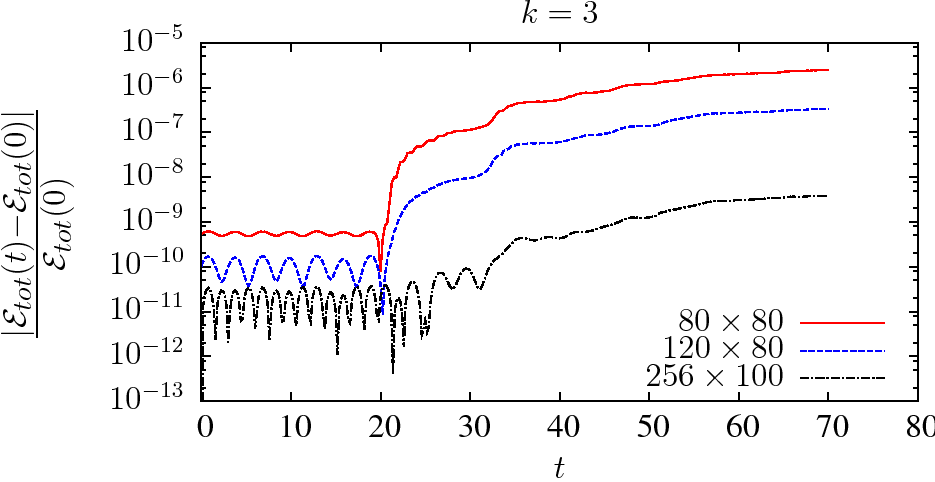}                  
		  \hfill
		  \includegraphics[scale=0.65, angle=0]{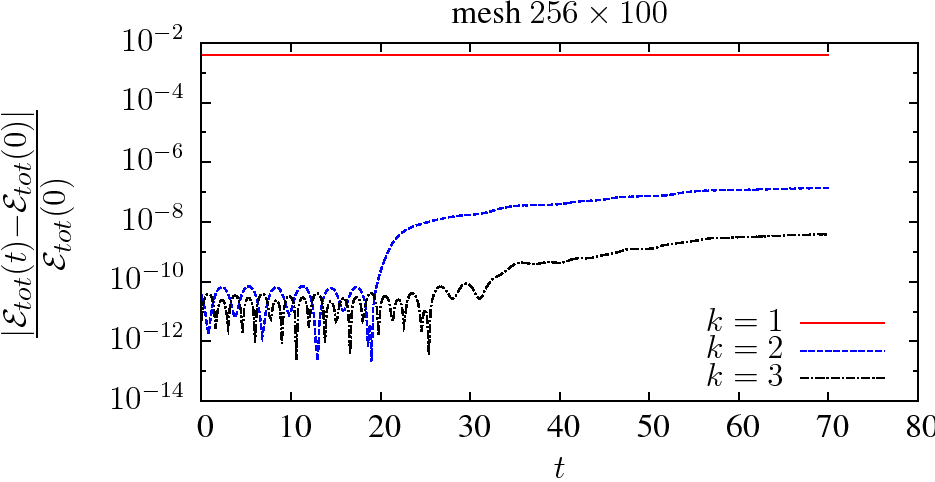}                 		  		  		  
		\caption{{\bf Two stream instability II:}  time evolution of the deviation from its initial value of the total discrete energy.}
		\label{fig:sim-vp-2sII1}
	\end{figure}
\section{Boundary value problem: a nonlinear plane diode}\label{sec:7}
In this section we consider the extension of the proposed DG schemes to approximate a Vlasov-Poisson  boundary value problem.  
\subsection{Model Problem: a nonlinear plane diode}
We first describe the model problem, revising also some key results  from  \cite{filbet_shu} related to its theoretical analysis. 
Denoting by $E:=\Phi_{x}$, the Vlasov-Poisson problem we consider is written as:
\begin{align}
		f_t + v f_x + E f_v &= 0 \qquad  &&(x,v,t) \in   [0,1] \times [-L,L] \times [0,\tf], \label{bv:1}\\
\partial_x E = \rho&=\int_{-L}^{L} fdv 	 \qquad  &&(x,t) \in   [0,1] \times  [0,\tf], \label{bv:0}\\
	\Phi(0,t) = 0, \quad  &\Phi(1,t) = \lambda(t) \qquad &&\forall \, t\in [0,\tf],\label{bv:02}
 \end{align}
 with boundary conditions
 \begin{equation}\label{bv:12}
 \left\{\begin{aligned}
  f(0,v,t) &= g(v,t), \quad &&v>0\quad \qquad \forall \, t\in [0,\tf],\\
 f(1,v,t)&=0,  \quad &&v<0 \qquad \quad \forall \, t\in [0,\tf], 
 \end{aligned}\right.
 \end{equation}
 where $g(v,t)$ is some given function. The system is complemented with an initial data $f(x,v,0)=f_0(x,v)$. Note that the Dirichlet boundary conditions for the Poisson problem imply 
\begin{equation}
 \int_0^1 E(x,t)\, dx = \lambda(t), \qquad \forall\, t\in [0,\tf].
 \end{equation}
 This problem models the evolution of a collisionless electron gas under the influence of the self-consistent electrostatic field $E$ in the interval $[0,1]$, where electrons are emitted at one end   of the interval ($x=0$) and absorbed at the other end ($x=1$).
Due to the absorbing boundary condition, in general  the distribution function $f$ might become discontinuous in finite time.  Existence and uniqueness of solutions of nonlinear boundary value problems of this kind have been studied to a certain extent.
The stationary one dimensional problem was considered in \cite{filbet2}  and the higher dimensional case in  \cite{M13, filbet}. For the time dependent problem \eqref{bv:1}-\eqref{bv:0}-\eqref{bv:02}-\eqref{bv:12}, weak solutions were constructed in \cite{M3,Arsenev73}. \\
In \cite{filbet_shu} the authors carry out an study of the regularity of the distribution function solution of the {\it nonlinear plane diode} \eqref{bv:1}-\eqref{bv:0}-\eqref{bv:02}-\eqref{bv:12}. In particular they show that under rather general assumption on the data $f_0$ and $g$, that the electron distribution is of bounded variation (BV) as a function of $x$ and $v$. This guarantees the uniqueness of the nonlinear plane diode. 
Furthermore, the authors study the influence of (the magnitude of) the external voltage $\lambda$ on the regularity of the solution. They show that if  initially $\lambda(0)$ is large the solution might become discontinuous, while if $\lambda(0)$ vanishes or if it is small and the inflow boundary $g(t,v)$ is non-vanishing, the electron distribution is of class $\mathcal{C}^{1}(\Omega)$ for all time. 

Prior to recall a couple of results from \cite{filbet_shu}, that allow for  quantifying what is mean by ``small'' or ``large''   external voltage,  we introduce some further notation. 
We denote by $\gamma_{S}$  the singular set, 
\begin{equation}
 \gamma_{S}=\{x=0,\,\, v=0\} \cup \{x=1,\,\, v=0\},
\end{equation}
and by $\gamma_0^{+}$ the incoming set at the boundary $x=0$:
\begin{equation}\label{def-gamma0}
\gamma_0^{+}=\{(0,v,t)\quad |\quad v>0, \quad 0\leq t\leq \tf\}\;.
\end{equation}
We also define the norm:
\begin{equation}\label{def:normaAA}
\||[f_0,g]\||:=\|f_0\|_{\mathcal{C}^{1}}+\|g\|_{\mathcal{C}^{1}}+\|v \partial_x f_0\|_{\infty} +\|v \partial_v f_0\|_{\infty} +\|v^{-1} \partial_t g\|_{\infty}+\|v^{-1} \partial_v g\|_{\infty}\;.
\end{equation}

 Next result corresponds to \cite[Lemma 4.3]{filbet_shu}
\begin{lemma}\label{lema4}
(\cite[Lemma 4.3]{filbet_shu}) 
Let  $f_0, g \in \mathcal{C}^{1}(\O)$ with compact support be such that 
\begin{equation}\label{tvp2}
 TV[f_0]+\int_{\gamma_0^{+}} \left( (1+v)|g_v| +|g_t|\right)+\|v^{p}f_0\|_{\infty}+\|v^{p}g\|_{\infty} <\infty\quad \mbox{for some  } p>2,
\end{equation}
where $\gamma_0^{+}$ is the incoming set at the boundary $x=0$ defined in \eqref{def-gamma0}. Assume  they satisfy the following compatibility conditions:
\begin{equation}\label{compat:0}
\left\{\begin{aligned}
&\qquad \qquad f_0(0,v)=g(v,0) &&\quad v>0,\\
& \qquad \qquad  f_0(1,v)=0 &&\quad v<0, \\
&\partial_t g(v,0) + v \partial_x f_0(0,v) + E(0,0) \partial_v f_0(0,v) = 0 &&\quad v>0, \\
&\qquad v \partial_x f_0(1,v) + E(1,0) \partial_v f_0(1,v) = 0 &&\quad v<0,
\end{aligned}\right.
\end{equation}
and that $\|| [f_0,g]\||<\infty$, where $\|| [\cdot]\||$ is the norm defined in \eqref{def:normaAA}.
 Furthermore, let 
\begin{equation}
f_0(x, v) \ne  f_0(0, 0) \quad \mbox{ for  } (x,v)\ne (0,0) \mbox{  and } x, v \mbox{  small.}
\end{equation}
Then, if the external voltage satisfies
\begin{equation}\label{eq:shu-condition}
		\lambda(0) > \int_0^1 \int_{-L}^{L} (1-x) f_0 \, dv\, dx\;,
	\end{equation}
	then $f(x,v,t)$ is not continuous.
\end{lemma}
In absence of external voltage, the solution can be shown to be of class $\mathcal{C}^{1}$.

\begin{theorem}\label{teo4}
(\cite[Theorem 4.4]{filbet_shu}.) 
Let $\lambda(0)=0$.  Assume  $f_0, g \in \mathcal{C}^{1}(\O)$ with compact support be such that $\|| [f_0,g]\||<\infty$ and assume they do satisfy \eqref{tvp2} and \eqref{compat:0}.
Then $f( x, v,t) \in \mathcal{C}^{1}\left( (\Omega\times [0,\tf])\smallsetminus \gamma_{S}\right)  \cap \mathcal{C}^{0}(\Omega\times [0,\tf])$.
\end{theorem}
\subsection{Numerical schemes}
To approximate the boundary value problem, \eqref{bv:1}-\eqref{bv:0}-\eqref{bv:02}-\eqref{bv:12}, we need to modify slightly the definition of the DG schemes given in Section \ref{sec:3}.
The change of sign of the self-consistent electrostatic field in the Vlasov equation, forces the following change in the definition of the method: 
find $(f_h,E_h):[0,\tf]\lor \calZ_{h}^{k}\times V_h^{k}$ such that
\begin{equation*}
\dyle\sum_{i=1}^{N_{x}}\dyle\sum_{j=1}^{N_{v}}\calB^{h}_{ij}(E_h;f_h,\varphi_h)=0
\quad \forall\varphi_h\in\calZ_{h}^{k}\;,
\end{equation*}
where
\begin{align}
\calB_{ij}(E_h;f_h,\varphi_h)=&\dyle{  \int_{\K_{ij}}
\frac{\partial f_h}{\partial t}\varphi_h \,dv\,dx
- \int_{\K_{ij}} v f_h \frac{\partial \varphi_h}{\partial x} \,dv\,dx - \int_{\K_{ij}} E_{h}^{i} f_{h} \frac{\partial \varphi_h}{\partial v}\,dv\,dx} &&\nonumber \\
&\dyle{+ \int_{J_{j}}\left[ (\widehat{(vf_h)} \varphi^{-}_h)_{i+1/2,v}-(\widehat{ (vf_h)}\varphi^{+}_h)_{i-1/2,v}  \right] dv } &&\label{method0000}\\
&\dyle{+ \int_{I_{i}}\left[
\left(\widehat{\left(E^{i}_{h}f_h\right)}
\varphi^{-}_h\right)_{x,j+1/2}-\left(\widehat{
\left(E^{i}_{h}f_h\right)}\varphi^{+}_h\right)_{x,j-1/2}  \right].
dx },&& \nonumber
\end{align}
with the numerical flux $\widehat{vf_h}$ defined as in \eqref{flux:v} and the definition of $\widehat{E_hf}$ is changed  to 
\begin{equation}\label{AA}
\widehat{E_{h}^{i}f_{h}} =\left\{\begin{array}{cc}
E_h^{i}\, f_{h}^{-}& \mbox{  if   } \calP^{0}(E_h^{i})\geq 0\\
E_h^{i}\, f_{h}^{+} & \mbox{  if   } \calP^{0}(E_h^{i})<0
\end{array}\right. \qquad \widehat{E_{h}^{i}f_{h}}= \av{E_{h}^{i}f_{h}}-\mbox{sign}\left(\calP^{0}(E_h^{i})\right)\cdot\frac{E_h^{i}}{2} \jump{f_h}\;. 
\end{equation}
Or, if using the weighted average
\begin{equation}\label{AA2:b}
\widehat{E_{h}^{i}f_{h}} =\left\{\begin{aligned}
 \av{E_{h}^{i}f_{h}}+\frac{|E_h^{i}|}{2} \jump{f_h} &\qquad \mbox{  if   }\nexists\, x^{\ast}  \in I_i \quad \mbox{such that} \quad E_h^{i}(x^{\ast})=0\; &&\\
\av{E_{h}^{i}f_{h}}+E_h^{i} (\omega^{-}-\frac{1}{2}) \jump{f_h} & \qquad\mbox{  if   }\exists\, x^{\ast}  \in I_i \quad \mbox{such that} \quad E_h^{i}(x^{\ast})=0\;&&
\end{aligned}\right. 
\end{equation}
with $\omega^{-}, \omega^{+}=1-\omega^{-}$ defined as in \eqref{def:w}.

At the boundary $\partial\O$,  we still reflect the compactness in $v$ in the numerical flux $\widehat{E_{h}^{i}f_{h}}$;
\begin{equation*}
(\widehat{E_{h}^{i}f_{h}})_{x,1/2}=(\widehat{E_{h}^{i}f_{h}})_{x,N_v+1/2}=0\;,\qquad
  \forall \,(x,v)\in\mathcal{I}\times \mathcal{J},
\end{equation*}
while for the $x$-boundary nodes we  account for the inflow boundary data:
\begin{equation*}
(\widehat{vf_h})_{1/2,v}=v (f_h^{-}-g(v,t)) \quad \forall v>0, \quad v \in \mathcal{J}\;.
\end{equation*}

To approximate the Poisson problem, we consider the LDG(v) method defined in \eqref{ldg0:a}-\eqref{ldg0:b}-\eqref{flux:DGv}.  The only modification required  is on the right hand side of the Poisson problem, since now there is no neutralizing background.At boundary nodes,  the numerical fluxes  are defined to account for the boundary conditions \eqref{bv:02}:
\begin{equation}\label{flux:BC}
\left\{\begin{aligned}
(\widehat{\Phi_{h}})_{1/2}&=0 \qquad  (\widehat{\Phi_{h}})_{N_x+1/2}=\lambda(t)&&\\
(\widehat{E_h})_{1/2} &=\av{E_h}_{1/2}+c_{11}\jump{\Phi_{h}}_{1/2}\;,&&\\
(\widehat{E_h})_{N_x+1/2} &=\av{E_h}_{N_x+1/2}+c_{11}\jump{\Phi_{h} -\lambda(t)}_{N_x+1/2}\;,&&
\end{aligned}\right.
\end{equation}
where $c_{11}=c\,
(k+1)^{2} h_x^{-1}$. 
For the time discretization we use the RK integrator described in Section \ref{sec:4}. 
\subsection{Numerical experiments}
We now present some numerical simulations obtained with the DG-LDG(v) method for the {\it nonlinear plane diode} . The main goal of this section is to verify if the proposed methods are able to detect and capture the singularity of the solution, when the external voltage $\lambda_0=\lambda(0)$ is large and satisfies the condition \eqref{eq:shu-condition}. 
Following \cite{filbet_shu}, we consider the following initial condition  
\begin{eqnarray*}
	f_0(x,v) &=& n_0(x) \frac{1}{\sqrt{2\pi}} v^2 \exp( -v^2 / 2) \\
	n_0(x) &=& \left\lbrace 
	\begin{array}{ll}
		(1+\gamma x)(1- 4x^2)^4 & x \in [0,0.5] \\
		0 & \text{else}
	\end{array}
	\right.
\end{eqnarray*}
and the inflow boundary data:
\begin{eqnarray*}
	g(v,t) = \frac{1}{\sqrt{2 \pi} } v^2 \exp( -v^2 / 2) \quad \forall\, t\in [0,\tf].
\end{eqnarray*}
 \begin{figure}[!htb]
		\centering   		  		  		  
		  \includegraphics[scale=0.65, angle=0]{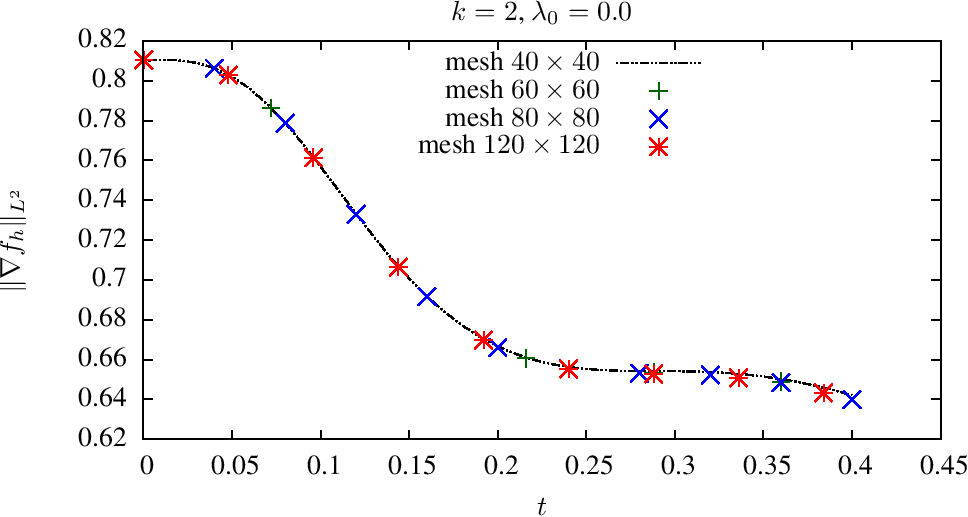}   
		  \hfill               
		  \includegraphics[scale=0.65, angle=0]{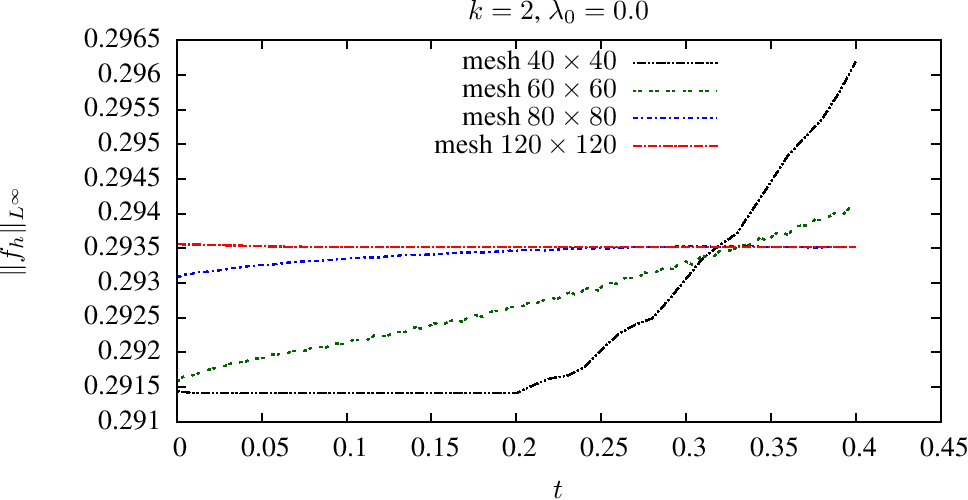}     
		\caption{ {\bf Plane Diode:} $\lambda(0)=0$. Evolution of $\Vert \nabla f_h \Vert_{L^2}$ (left) and 
		$\Vert f_h \Vert_{L^\infty}$ (right) for different mesh size.
		}
		\label{fig:bvp:00}
	\end{figure}
We have taken $L= 10$ to set the domain in velocity, since both $f_0$ and $g$ are of compact support in  $[-L,L]=[-10,10]$.
It can be checked that the data satisfies the smoothness conditions $\|| [f_0,g]\||<\infty$ and  \eqref{tvp2} together with the compatibility conditions given in \eqref{compat:0}. 
We have computed the approximate solution to \eqref{bv:1}-\eqref{bv:0}-\eqref{bv:02}-\eqref{bv:12} for different values of the external voltage $\lambda(0)$.  The main goal is to verify that the presented DG methods are able to detect the smoothness of the approximate solution. That is for $\lambda(0)= 0$, in view of Theorem \ref{teo4}, $f( x, v,t) \in \mathcal{C}^{1}\left( (\Omega\times [0,\tf])\smallsetminus \gamma_{S}\right)  \cap \mathcal{C}^{0}(\Omega\times [0,\tf])$, and we expect the approximate solution to reflect such regularity.\\ 
\begin{figure}[!htb]
		\centering
		  \includegraphics[scale=0.65, angle=0]{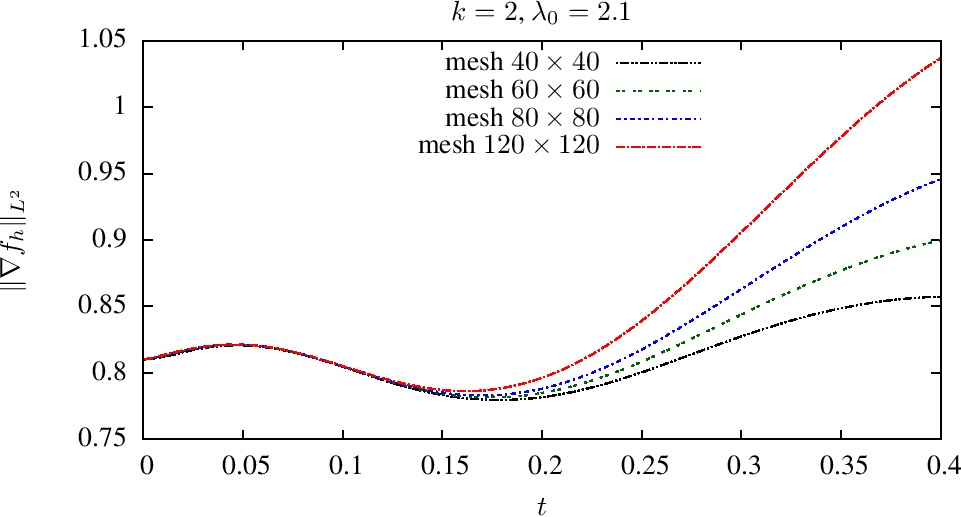}   
		  \hfill               
		   \includegraphics[scale=0.65, angle=0]{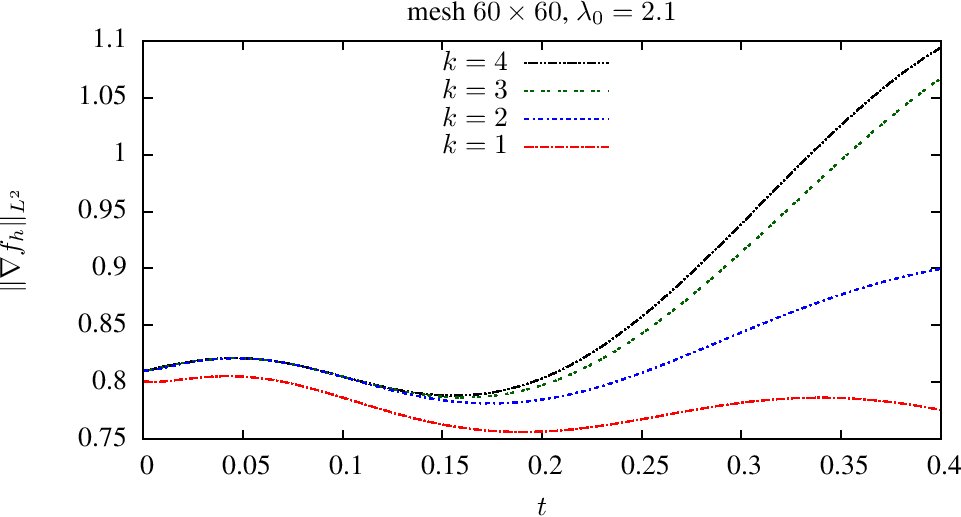} 
		 \caption{ {\bf Plane Diode:} Evolution of $\Vert \nabla f_h \Vert_{L^2}$  for different mesh size (left) and different polynomial degree (right).
		}
		\label{fig:bvp2}
	\end{figure}
	
In Figure \ref{fig:bvp:00} we have depicted the time evolution of $L^{2}$-norm of $\nabla f_h$ and $L^{\infty}$-norm of $ f_h$ for different mesh sizes. As can be observed from the figures, both quantities are finite and bounded and converge toward the same value (finite) as the mesh is refined.

We have run the same experiment for $\lambda(0)= 2.10947$ and $\lambda(0)=10$. For both values the condition \eqref{eq:shu-condition} is verified, and therefore in view of Lemma \ref{lema4} we expect the solution $f$ to become discontinuous. To assess the ability of the DG methods to capture the singularity (or change in smoothness of $f$) we have measured the time evolution of the discrete $L^{2}$ and $L^{\infty}$-norms of $\nabla f_h$ and study how they are affected under mesh refinement and by an increment in the polynomial degree ($k=1,2,3,4$). The results for $\lambda(0)= 2.10947$ are depicted in Figure \ref{fig:bvp2}. \\
\begin{figure}[!htb]
		\centering
		  \subfloat[$t=0.0$]{
		  \includegraphics[scale=0.32, angle=-90, clip=true, viewport=1.5in 1.5in 7.5in 10.5in]{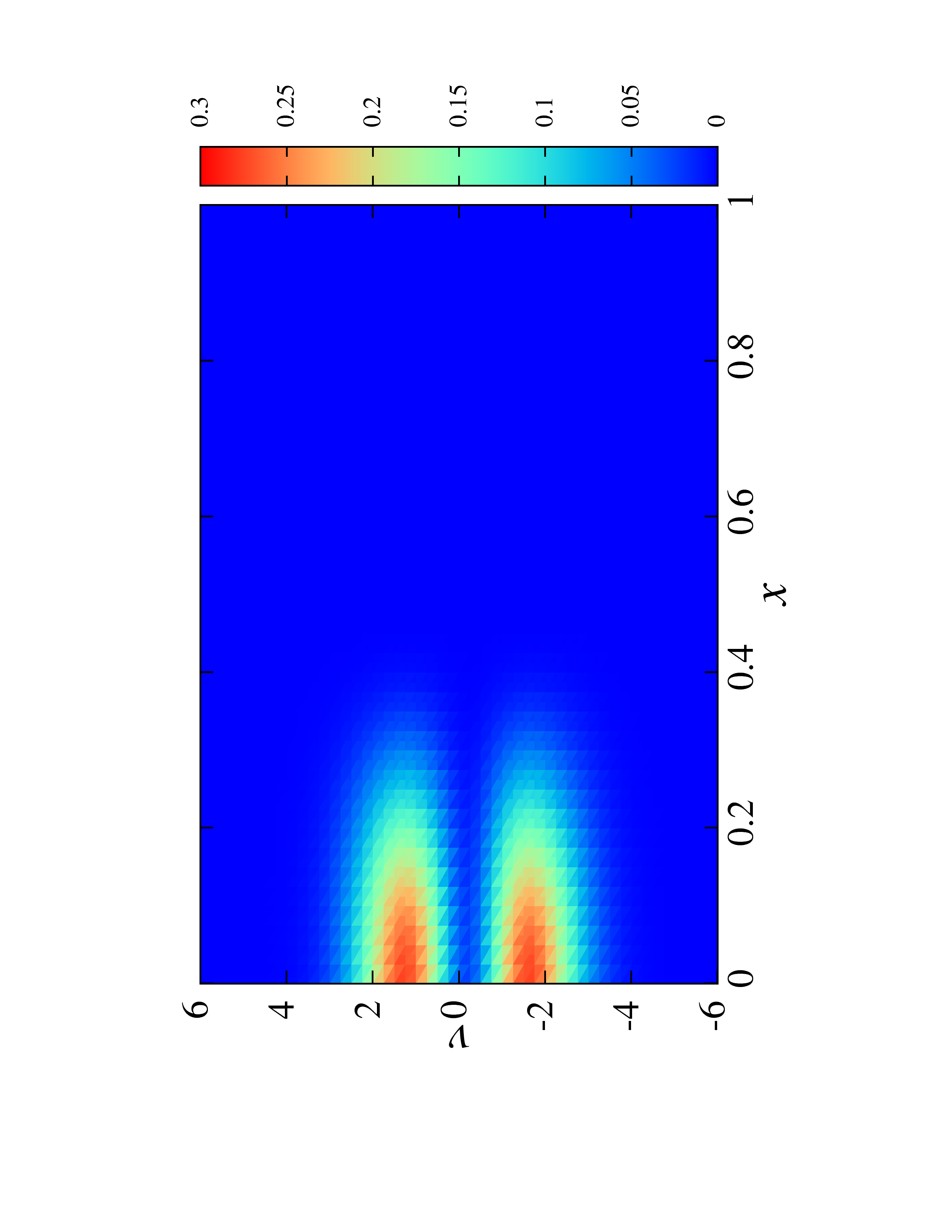}}                
		  \subfloat[$t=0.1$]{
		  \includegraphics[scale=0.32, angle=-90, clip=true, viewport=1.5in 1.5in 7.5in 10.5in]{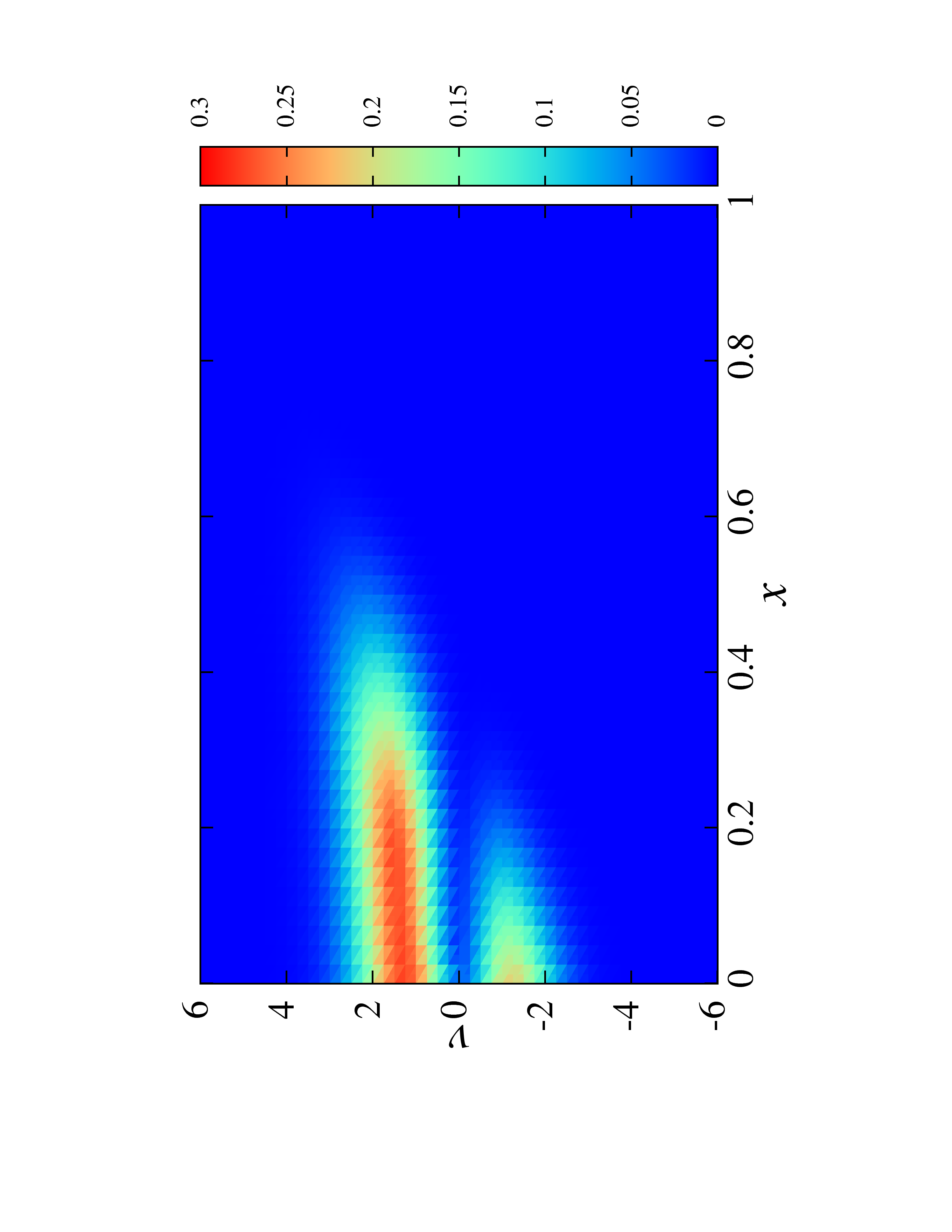}}
		  \\ 	  
		  \subfloat[$t=0.2$]{
		  \includegraphics[scale=0.32, angle=-90, clip=true, viewport=1.5in 1.5in 7.5in 10.5in]{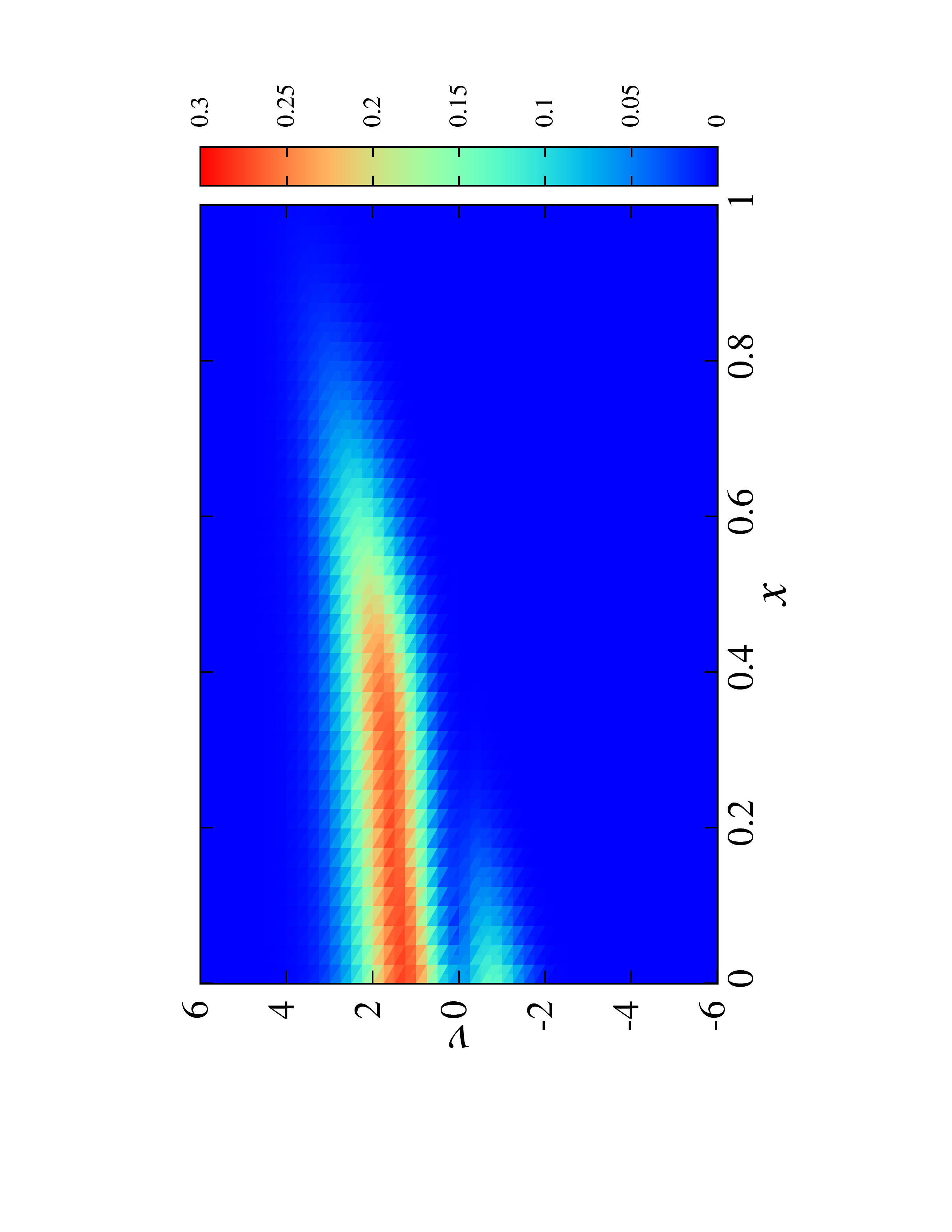}}                
		  \subfloat[$t=0.3$]{
		  \includegraphics[scale=0.32, angle=-90, clip=true, viewport=1.5in 1.5in 7.5in 10.5in]{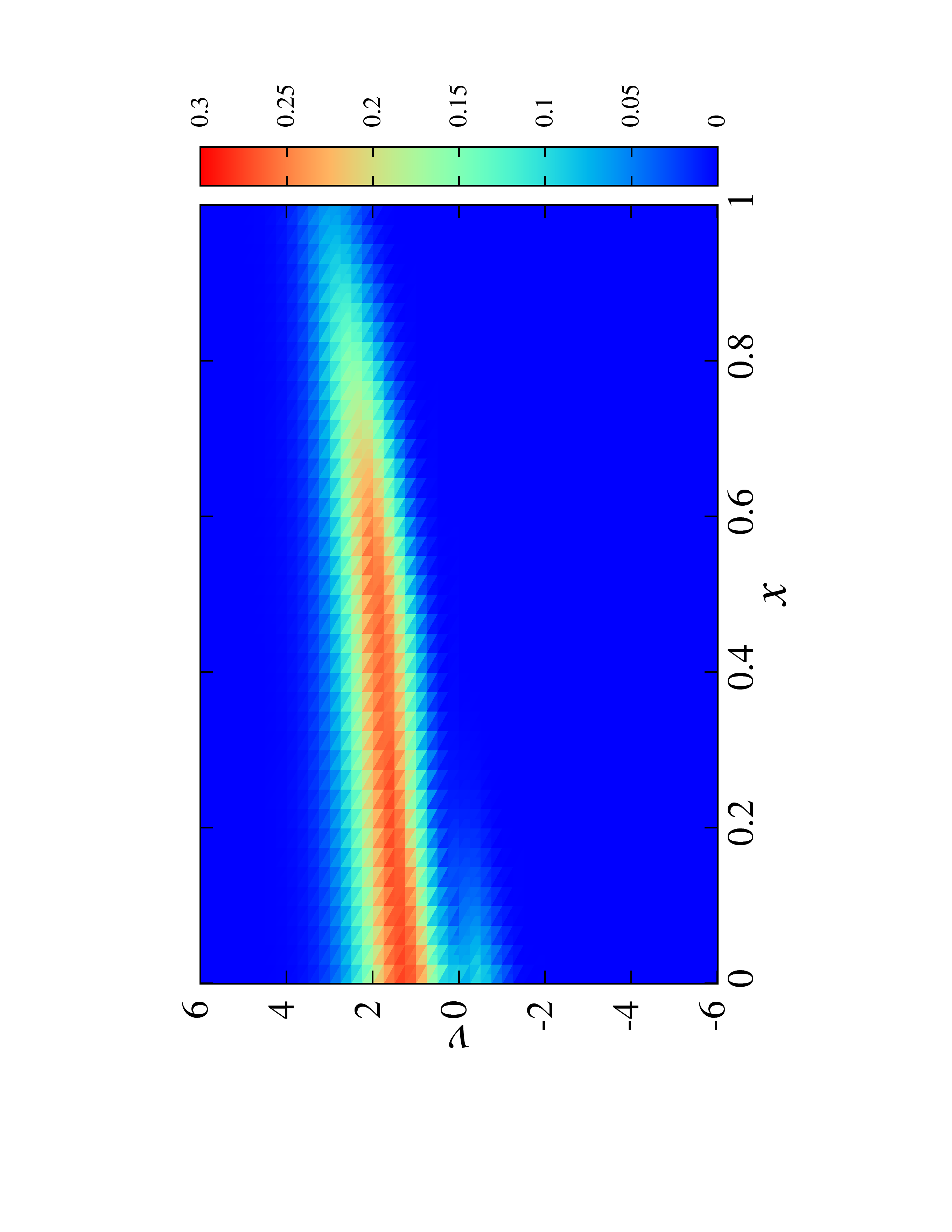}} 	  
		\caption{ {\bf Plane Diode:} Solution of the VP system for non-smooth solution test case at different times using mesh $60 \times 60$ and 
		 $k=2$.}
		\label{fig:sim-non-smooth2}
	\end{figure}

Observe that the behaviour under mesh refinement of the measured quantities is very different for $\lambda(0)=0$ (Figure \ref{fig:bvp:00}) and $\lambda(0)= 2.10947$ (Figure \ref{fig:bvp2}). From the graphics, for the former case we deduce convergence, indicating that the solution is indeed continuous. For $\lambda(0)= 2.10947$, the $\Vert \nabla f_h \Vert_{L^2}$ increases as the mesh is refined, which seems to indicate that  $\Vert \nabla f_h \Vert_{L^2}$  diverges (and blows up in finite time). The same effect is observed by increasing the polynomial degree.\\
\begin{figure}[!htb]
		\centering
		  \includegraphics[scale=0.65, angle=0]{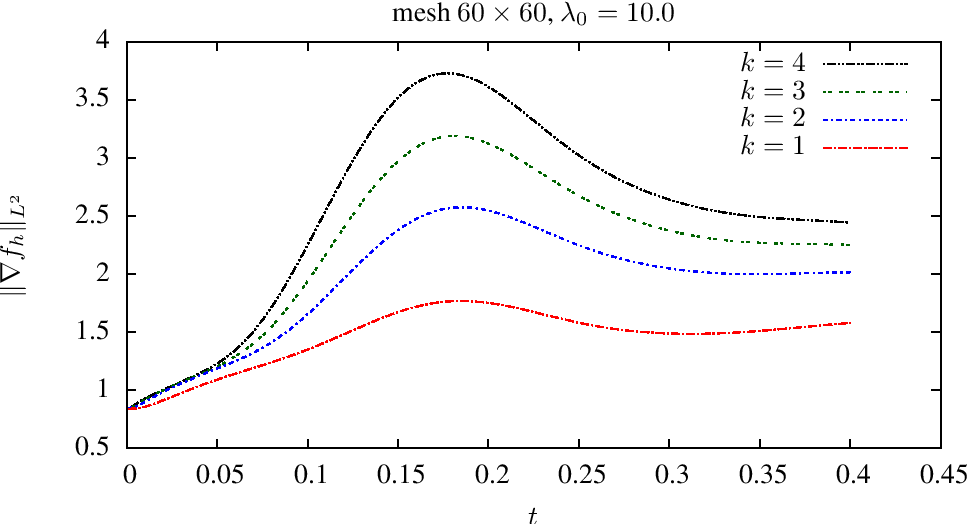}   
		  \hfill               
		  \includegraphics[scale=0.65, angle=0]{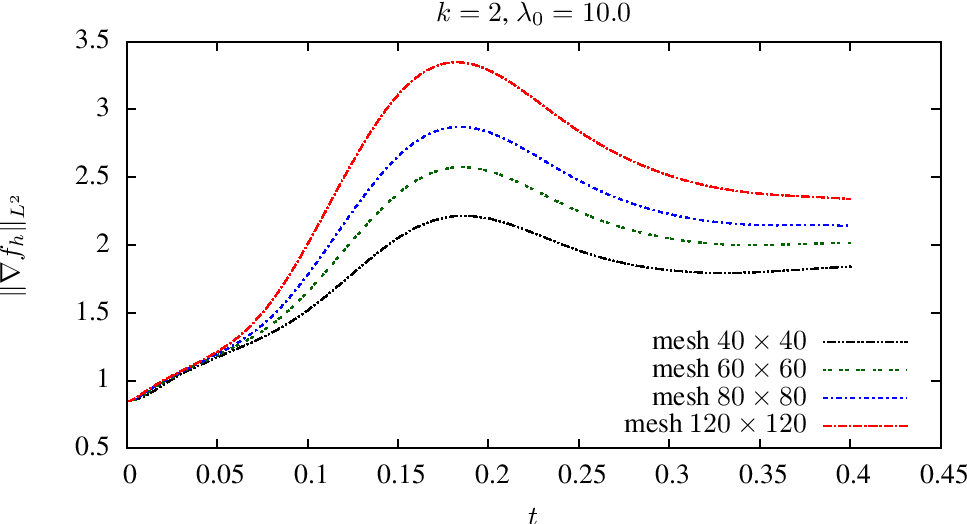} 
		  \\		
		  \includegraphics[scale=0.65, angle=0]{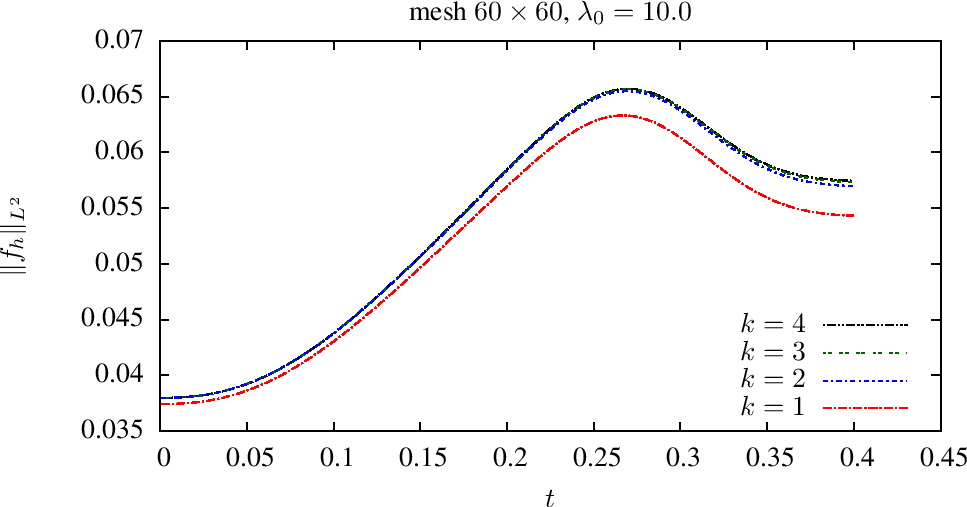}     
		  \hfill               
		  \includegraphics[scale=0.65, angle=0]{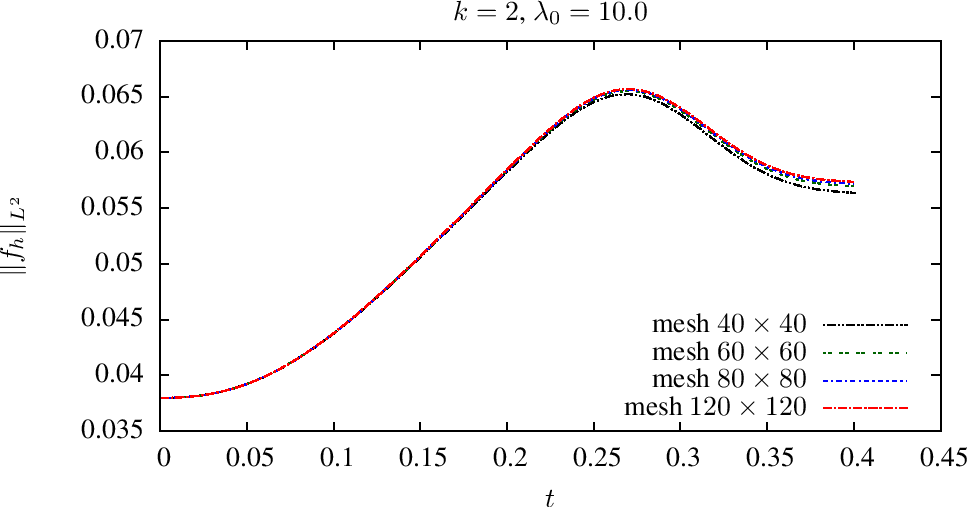}     	
		  \\
		  \includegraphics[scale=0.65, angle=0]{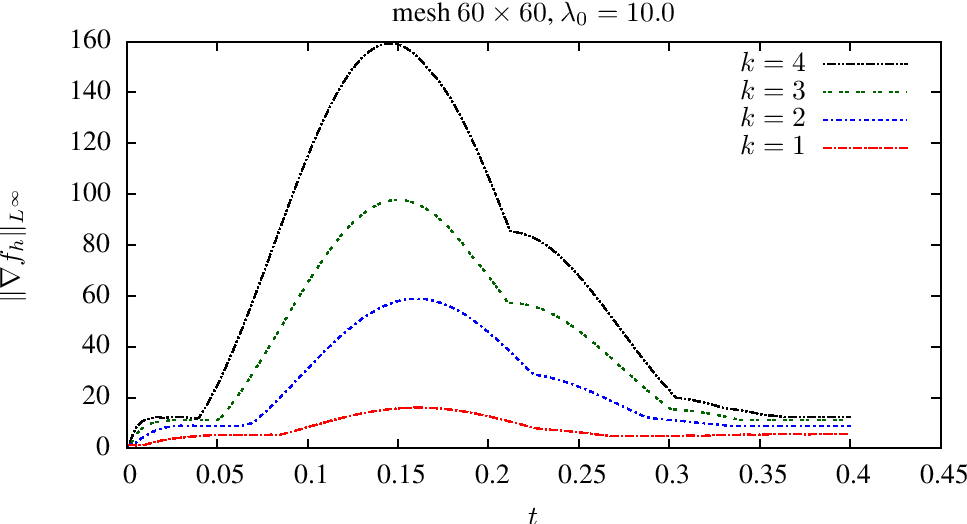}
		  \hfill                	  
		  \includegraphics[scale=0.65, angle=0]{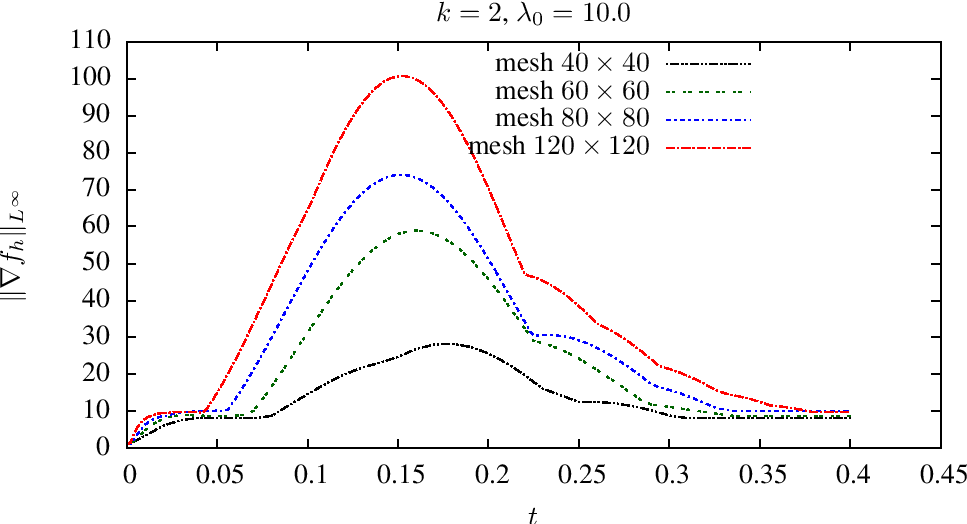}   
		\caption{ {\bf Plane Diode:} Evolution of $\Vert \nabla f_h \Vert_{L^2}$ (top), 
		$\Vert f_h \Vert_{L^2}$ (middle) and
		$\Vert \nabla f_h \Vert_{L^\infty}$ (bottom) for different mesh size and different $k$.
		}
		\label{fig:bvp3}
	\end{figure}
In Figure~\ref{fig:sim-non-smooth2} we represent the approximate solution in phase space $(x,v)$ for different times for  $\lambda(0)= 2.10947$.\\
\begin{figure}[!htb]
		\centering
		  \subfloat[$t=0.0$]{
		  \includegraphics[scale=0.32, angle=-90, clip=true, viewport=1.5in 1.5in 7.5in 10.5in]
		  {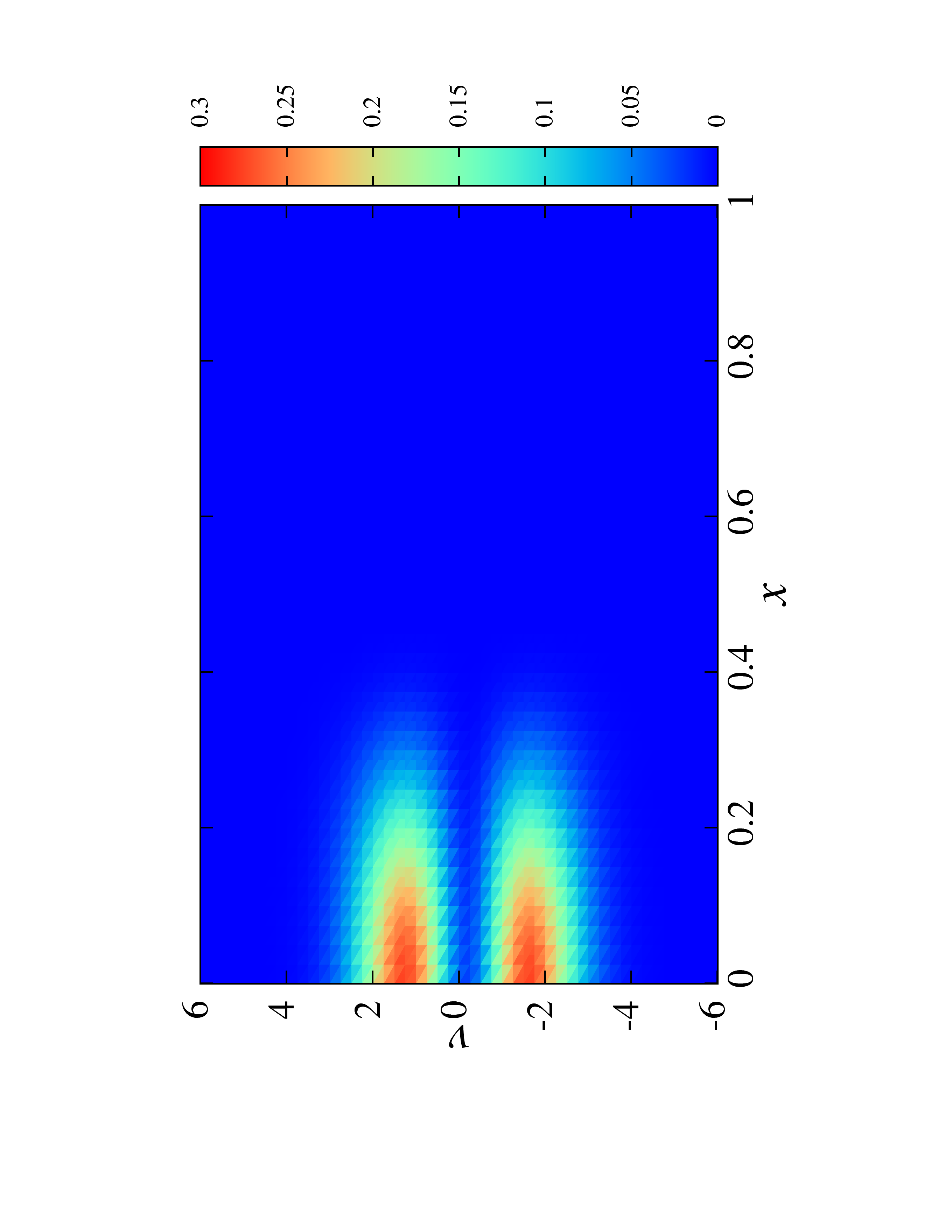} } 
		  \subfloat[$t=0.1$]{
		  \includegraphics[scale=0.32, angle=-90, clip=true, viewport=1.5in 1.5in 7.5in 10.5in]
		  {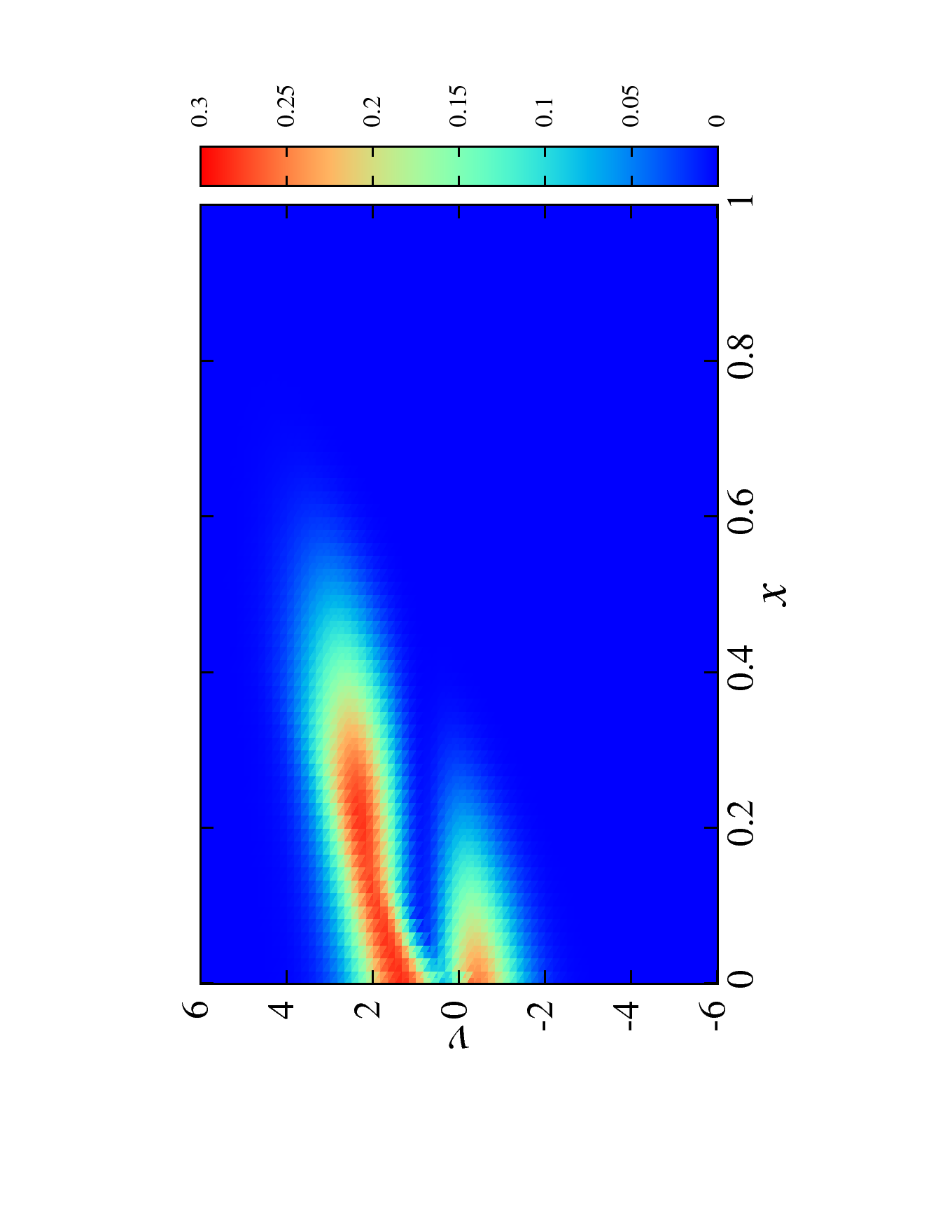} } 
		  \\ 	  
		  \subfloat[$t=0.2$]{
		  \includegraphics[scale=0.32, angle=-90, clip=true, viewport=1.5in 1.5in 7.5in 10.5in]
		  {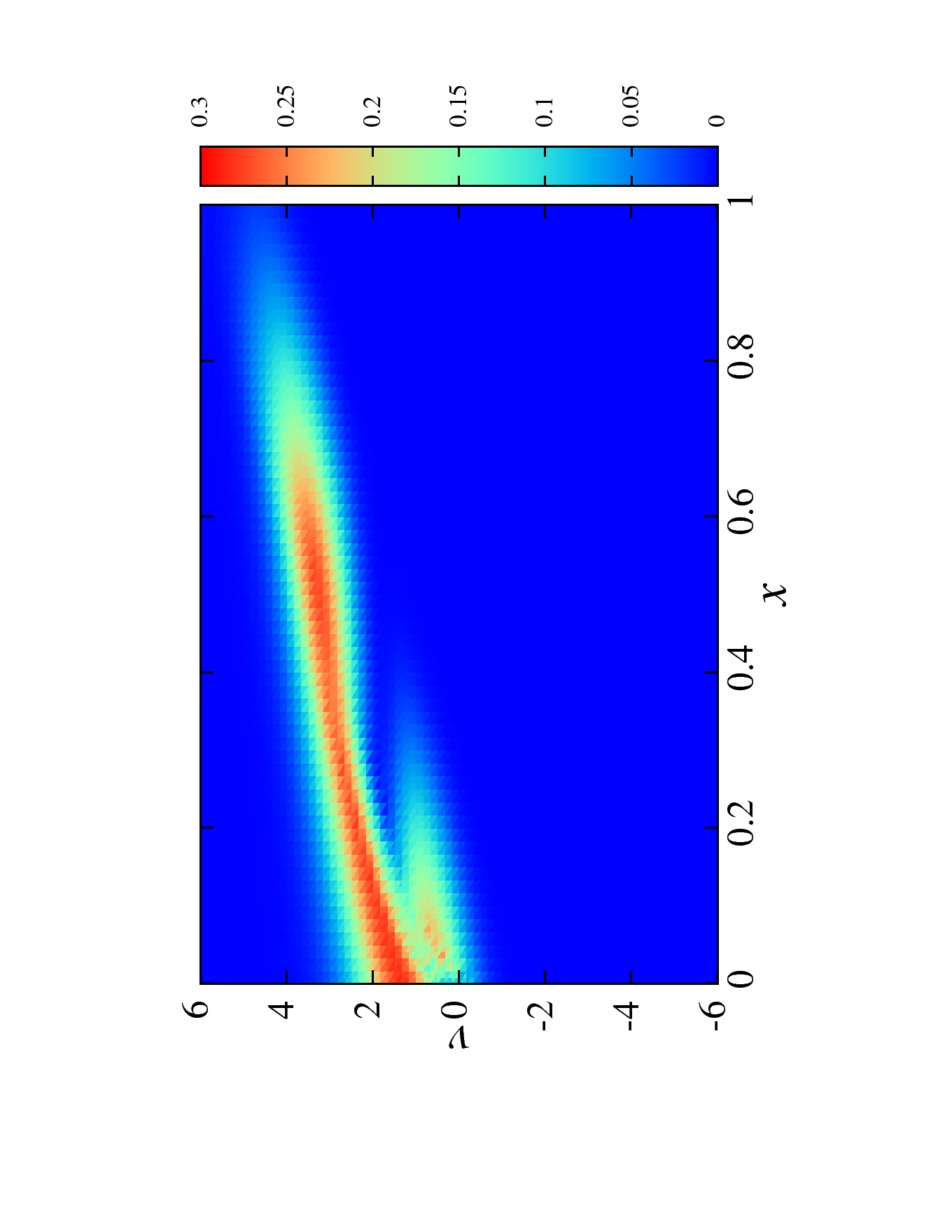} } 
		  \subfloat[$t=0.3$]{
		  \includegraphics[scale=0.32, angle=-90, clip=true, viewport=1.5in 1.5in 7.5in 10.5in]
		  {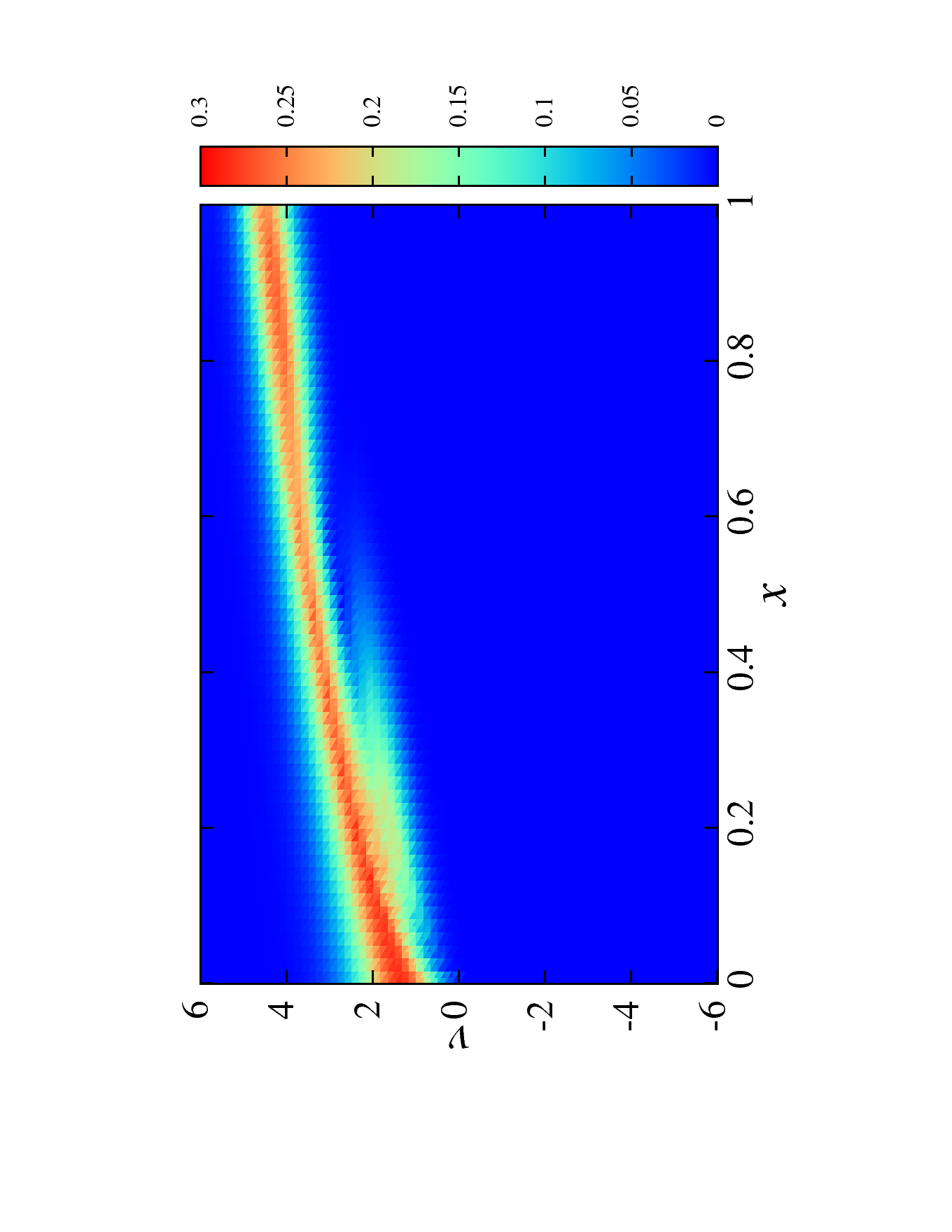} } 
		\caption{{\bf Plane Diode:} Solution of the VP system for non-smooth solution test case at different times using mesh $60 \times 60$ and 
		 $k=2$.}
		\label{fig:sim-non-smooth4}
	\end{figure}
	
The corresponding results and graphics for  $\lambda(0)=10$ are given in Figure \ref{fig:bvp3}.
Notice that in this case, by refining the mesh the quantity $\Vert \nabla f_h \Vert_{L^2}$ 
and $\Vert \nabla f_h \Vert_{L^\infty}$ also increases although after some time it decreases again.
This effect might be due to the fact that the singularity in the problem is very weak, and as the time evolves the full discretized DG scheme might be adding too much artificial diffusion to capture the singularity at all times (note that the explicit time discretization adds also numerical diffusion). The issue of coupling the scheme with some conservative integrator is currently under investigation.
\\
For $\lambda(0)=10$ we plot in Figure \ref{fig:sim-non-smooth4} the approximate density in phase space, computed with a mesh $60\times 60$ and polynomial degree $k=2$.
To further assess the possible ability of the DG schemes to capture the discontinuity in $f_h$ we have represented in Figure  \ref{fig:sim-non-smooth-profile}  the profile of the solution.
	\begin{figure}[!htb]
		\centering
		  \subfloat[mesh $80 \times 80$]{
		  \includegraphics[scale=0.65, angle=0]{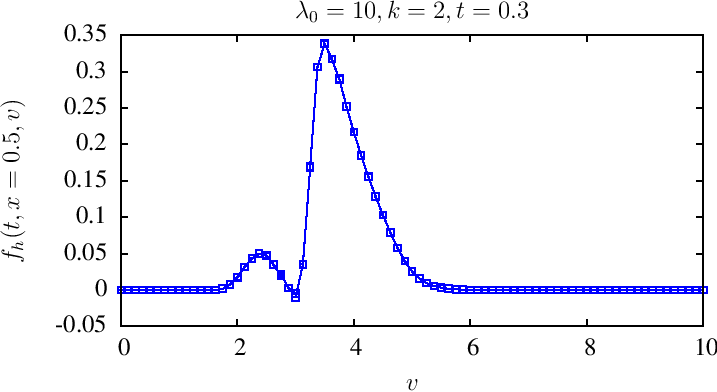}}    
		  \hfill              
		  \subfloat[mesh $120 \times 120$]{
		  \includegraphics[scale=0.65, angle=0]{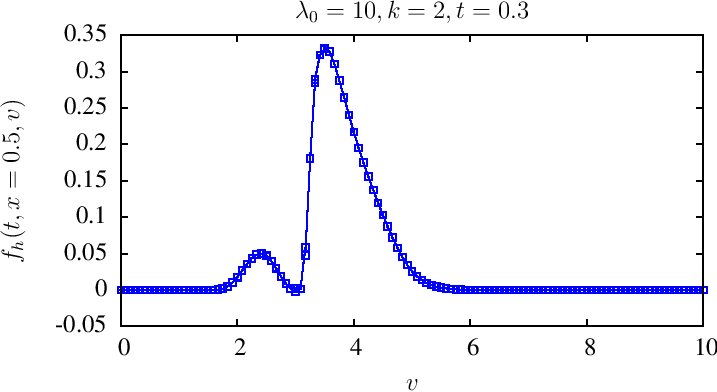}}                  \\
		  \subfloat[mesh $160 \times 160$]{
		  \includegraphics[scale=0.65, angle=0]{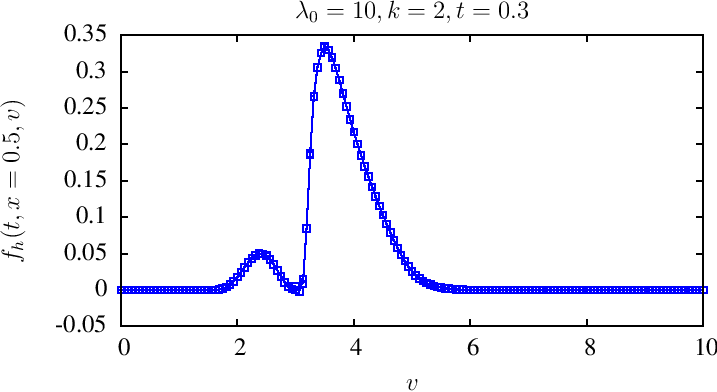}} 
		  \hfill                 
		  \subfloat[mesh $200 \times 200$]{
		  \includegraphics[scale=0.65, angle=0]{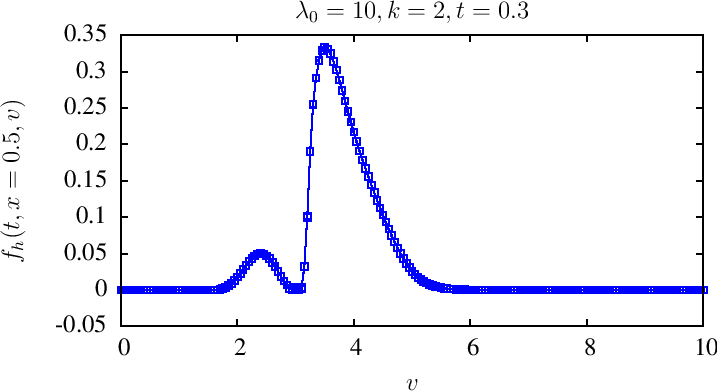}}                  		  
		\caption{{\bf Plane Diode:} The profile of the solution at $x=0.5$ and different mesh for $\lambda_0 = 10$.}
		\label{fig:sim-non-smooth-profile}
	\end{figure}
\section{Conclusions}\label{sec:fin}
We have studied the verification and validation of the high order DG methods introduced in \cite{acs0} for approximating the one-dimensional Vlasov-Poisson system with periodic boundary conditions.
We have proposed two possible modifications of the definition of the DG schemes, that allow for practical and efficient implementation. We have shown theoretically and demonstrated numerically, that with such modifications the resulting DG methods still preserve the total discrete mass and energy (this if $k\geq 2$). We have also verified numerically that in the case of smooth solutions, the approximate distribution function and the electrostatic field converge optimally in $L^{2}$. We have discussed the time integration for the schemes, demonstrating that there is no essential benefit in using a TVD RK integrator (rather than a standard  high order RK) for the practical simulations in plasma physics. We have shown how the fully discretized methods can be efficiently implemented in parallel. 
Moreover the performance of the introduced DG methods is validated, with several benchmark problems of plasma physics such as {\it linear} and {\it nonlinear Landau damping} and two different tests on {\it two stream instability}.
We have also discussed how the schemes could be adapted for approximating a Vlasov-Poisson boundary value problem. It is demonstrated  that the schemes have some potential ability for capturing the discontinuity of  the solution in this case, although some further tuning on the time integration seems to be required to reproduce  the correct weak singularity.  This is currently under research.

\section*{Acknowledgments}
Blanca Ayuso de Dios thanks Juan Jos\'e L\'opez-Velazquez  from Bonn University for helpful discussions on the boundary value problem. The first author has  been partially supported by MICINN grant MTM$2011-27739-$C$04-04$ and  Ag\`encia de Gesti\'o d'Ajuts Universitaris i de Recerca-Generalitat de Catalunya grant $2009-$SGR$-345$.  Soheil Hajian thanks MathMods consortium, Universitat Aut\`onoma de Barcelona (UAB) and Centre de Recerca Matem\`atica (CRM). The second author has been partially supported by Erasmus Mundus scholarship of the European commission and a grant from CRM.

\appendix
\section{Basis functions}\label{ap0}
\subsection{Lagrange polynomials}
In the following parts we introduces basis functions that we used in the actual implementation of the scheme.
Let $I:=[0,1]$ then in order to span the $\mathbb{P}^k(I)$ we introduce the Lagrange basis function defined by 
\begin{equation}
	\hat{l}_n(r) := 
	\prod_{1 \leq m \leq k+1, m \not = n} \frac{r - r_m}{r_n - r_m}
	\qquad r \in I, \forall n= 1,...,k+1,
\end{equation}
where $\{ r_1, ..., r_{k+1} \}$ is the set of distinct nodal coordinates in $I$. Recall that the Lagrange polynomials satisfy
\begin{equation}
	\hat{l}_n(r_m) = \delta_{n m} \quad \forall n,m = 1,...,k+1.
\end{equation}
Now let $f_h(.,.,t) \in \mathbb{P}^k(I) \times \mathbb{P}^k(I)$, then we can write it as a tensor product
\begin{equation}
	f_h(x,v,t) = \sum_{n,m=1}^{k+1} \alpha_{n,m}(t)\, \hat{l}_n(x) \hat{l}_m(v),
\end{equation}
where $\alpha_{n,m}(t)$ are interpolation coefficients of $f_h(x,v,t)$ at $(x_n,v_m)$ and 
\begin{equation}
	\{ (x_1,v_1), (x_1,v_2), ..., (x_{k+1},v_{k+1}) \},
\end{equation} is the set of distinct nodal coordinates in $I^2$. For more information about the other possible basis functions and methods to evaluate mass and gradient matrices using Lagrange polynomials we refer to \cite{Hesthaven:2007:NDG:1557392}.
\subsection{Bernstein polynomials}
Here we briefly introduce a basis function for $\mathbb{P}^k(I)$ that has a useful properties and will be used to facilitate evaluation of the term containing flux of electric field, i.e. $\widehat{E_h f_h}$, in \eqref{method0}.
Consider we are interested to determine whether or not $E_h \in \mathbb{P}^k(I)$ changes sign in $I:=[0,1]$. A simple approach to show that $E_h$ is positive (or negative) over whole $I$ is to express $E_h$ using polynomials that are positive in $I$. To do so we use Bernstein polynomials, $B^k_n(x)$, which are positive over $I$ and satisfy
\begin{equation}
	B_n^k(x) = (1-x) B_{n}^{k-1}(x) + x \, B_{n-1}^{k-1}(x)
	\quad \forall n = 0,...,k,
\end{equation}
for $B_{0}^{0}(x) = 1$ and $B_n^k(x)=0, \forall n < 0$ or $n > k$. 

We  use two properties of the Bernstein in the actual implementation. First one is
\begin{equation}
	B_n^k(x) \geq 0 \quad \forall n =0, ..., k-1, \quad x \in [0,1].
\end{equation}
More precisely, the Bernstein polynomials are all positive over $I$. Let us express $E_h(x)$ by
\begin{equation*}
	E_h(x) = \sum_{m=0}^{k} \beta_m \, B_{m}^{k}(x),
\end{equation*}
then $E_h(x)$ is non-negative over $I$ if $\beta_m \geq 0, \forall m$ (and similarly non-positive if $\beta_m \leq 0 $). Another properties says that $E_h(x)$ is located inside the convex hull produced by the set $\{ \beta_m \}$, which implies
\begin{equation}
\begin{array}{l}
	\displaystyle
	\min_{m} \beta_m \leq \min_{x \in I} E_h(x),
	\\ \displaystyle
	\max_{x \in I} E_h(x) \leq \max_{m} \beta_m	.
\end{array} 
\end{equation}
This last properties helps us to obtain an estimate of $\min$ and $\max$ of $E_h$. For more details on other properties of Bernstein method that we do not use here, we refer to \cite{phillips-interp}.

As we discussed before we  use Lagrange polynomials to span $\mathbb{P}^k$ and therefore for $E_h$. Now the question is how one can convert coefficients of expansion using Lagrange basis functions to Bernstein. The change of basis from Lagrange to Bernstein involves a matrix-vector multiplication, where the matrix is the inverse of a Vandermonde matrix.
We expand $E_h(x)$ in Lagrange polynomial basis
\begin{equation}
	E_h(x) = \sum_{m=1}^{k+1} \alpha_m \, \hat{l}_{m}(x),
\end{equation}
and define the vector of coefficients in Lagrange and Bernstein by
\begin{equation}
		{\boldsymbol \beta } := \left[ \beta_0, ..., \beta_{k} \right]^{T}, \quad
		{\boldsymbol \alpha } := \left[ \alpha_1, ..., \alpha_{k+1} \right]^{T},
\end{equation}
then we have
\begin{equation}
	{\boldsymbol \alpha } = V \cdot {\boldsymbol \beta },
\end{equation}
where the Vandermonde matrix is defined by
\begin{equation}
	V_{nm} := B_{m-1}^{k}(r_n) \quad \forall n,m \in 1,...,k+1.
\end{equation}
If we want to check positivity (similarly negativity) of $E_h$ expressed using Lagrange polynomials (given ${\boldsymbol \alpha }$), we have to compute, ${\boldsymbol \beta } = V^{-1} \cdot {\boldsymbol \alpha }$ and then we just check positivity of all entries of ${\boldsymbol \beta }$ (and similarly for negativity).

Note that since $V$ is a $(k+1)\cdot(k+1)$ Vandermonde matrix, its inversion is  cheap and done once for all,  before the time marching process starts and it is saved for later use. Therefore 
In Figure \ref{fig:basis}, we depict Lagrange, Legendre and Bernstein polynomials for $k=2$. Note the positivity of Bernstein polynomials on $I=[0,1]$.
\begin{figure}[!htb]
		\centering
		\includegraphics[scale=1]{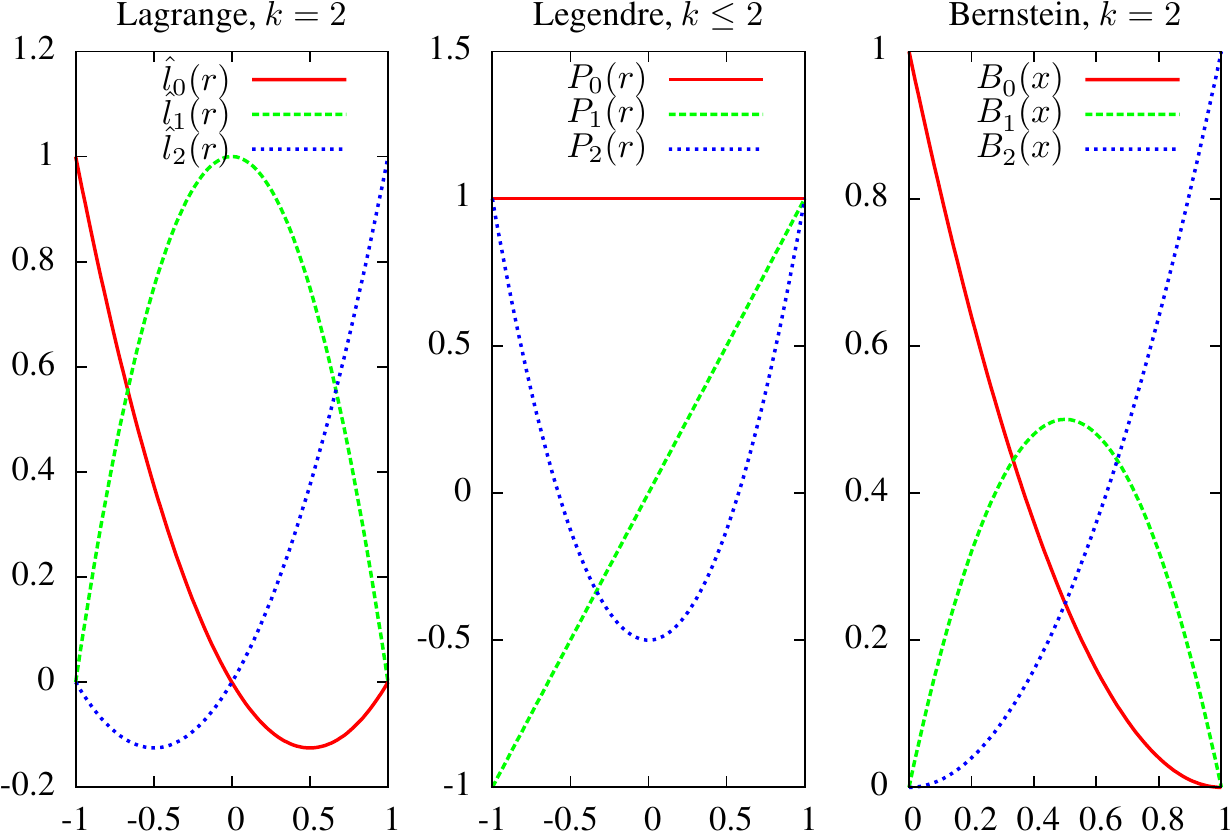}	
		\caption{Lagrange, Legendre and Bernstein basis functions of degree $2$.}
		\label{fig:basis}
\end{figure}


\end{document}